\documentclass[11pt, reqno]{amsart}
\usepackage{amssymb}
\usepackage{times}
\usepackage{amsmath,amsthm}
\usepackage{amsfonts}
\usepackage{leftidx}
\usepackage{mathrsfs}
\usepackage{enumerate}
\usepackage{abstract}
\usepackage{stmaryrd}
\usepackage{ulem}
\usepackage{graphicx}

\usepackage{float}
\usepackage{hyperref}
\def\R{\mathbb{R}}

\def\med{\textup{med}}
\def\d{|\nabla|}

\def\p{\partial}
\def\hx{\partial_x h}
\def\px{\partial_{x}}
\def\pt{\partial_{t}}
\def\vo{\vspace{1\baselineskip}}

\def\h{\frac{1}{2}}

\newtheorem{theorem}{Theorem}[section]

\newtheorem{lemma}[theorem]{Lemma}
\newtheorem{proposition}[theorem]{Proposition}
\theoremstyle{definition}
\newtheorem{definition}[theorem]{Definition}

\theoremstyle{remark}
\newtheorem{remark}[theorem]{Remark}

\numberwithin{equation}{section}
\begin{document}
 \title[Infinite energy solution of the 2D gravity water waves]{Global infinite energy solutions for the 2D gravity water waves system}
\author{Xuecheng Wang}
\address{Mathematics Department, Princeton University, Princeton, New Jersey,08544, USA}
\email{xuecheng@math.princeton.edu}
\thanks{}
\maketitle
\begin{abstract}
We prove global existence and modified scattering property for the solutions of the  $2D$  gravity water waves system  in the infinite depth setting for a class of initial data, which is only required to be small above the level $\dot{H}^{1/5}\times \dot{H}^{1/5+1/2}$. No assumption is assumed below this level, therefore, it allows to have infinite energy. As a direct consequence, the  momentum condition assumed on the physical velocity in all previous small energy results by Ionescu-Pusateri\cite{IP1}, Alazard-Delort\cite{alazard} and Ifrim-Tataru \cite{tataru3} is removed. 
\end{abstract}

\section{Introduction}

\subsection{The gravity water waves system}
We consider an incompressible irrotational inviscid fluid with  density one occupying a time dependent domain $\Omega(t)$ with the free interface $\Gamma(t)$ and without a bottom. Above $\Gamma(t)$, it is vacuum.  Assume that the interface $\Gamma(t)$ of the domain $\Omega_t$ is given as $\Gamma(t) =\partial \Omega(t) = \{ (x, h(t, x)): x \in \R \}$, then $\Omega(t)=\{ (x, y): y\leq h(t, x), x\in \R\}$.

 Let $v$ and $p$  denote the velocity and the pressure of the fluid respectively. As the evolution of the fluid is described by the Euler equations with free boundary, then $v$ and $p$ satisfy the following system of equations,
 \begin{equation}\label{euler}
 \left\{ \begin{array}{ll}
 \p_t v(t, X) + v(t,X) \cdot \nabla v(t, X) = - \nabla p(t, X) - g(0, 1) &  X\in \Omega(t) \\
 \nabla \cdot v(t, X) =0,\,\, \nabla \times v(t,X) =0 &  X\in \Omega(t) \\
 \p_t + v\cdot \nabla\,\, \textup{is tangent to} \cup_{t} \Gamma(t), \quad p(t, x)=0,& x \in \Gamma(t)\\
 v(0, X) = v_0(X) & X \in \Omega(0),\\
 \end{array}\right.
 \end{equation}
where $g$ is the gravitational constant, which is assumed to be $1$ throughout this paper, and $v_0(X)$ is the initial physical velocity field.

As the fluid is irrotational, we assume that the velocity field is given by the gradient of a potential function $\phi$.  Let $\psi(t,x):= \phi(t, x, h(t,x))$ be the restriction of potential $\phi$ to the boundary $\Gamma(t)$. From the incompressible condition, we have the following harmonic equation with a Dirichlet type boundary condition at the interface $\Gamma(t)$,
\begin{equation}\label{harmonicequation}
v(t, X) = \nabla \phi(t, X),\quad \Delta \phi(t, X) =0, X\in \Omega(t), \quad \phi(t,X)\big|_{\Gamma(t)}= \psi(t,x).
\end{equation}
Therefore, it is sufficient to study the evolution of system at the interface $\Gamma(t)$, i.e., the evolution of the height $h(t,x)$ and the  velocity potential $\psi(t,x)$ on the interface.

  Following Zakharov \cite{zakharov} and Craig-Sulem-Sulem\cite{css}, we can derive the  system of equations satisfied by $(h, \psi): \R_t \times \R_x \rightarrow \R$ as follows, 
\begin{equation}\label{waterwave}
\left\{\begin{array}{l}
\p_t h= G(h)\psi\\
\p_t \psi = -h- \frac{1}{2} |\p_x\psi|^2 + \displaystyle{\frac{(G(h)\psi + \p_x h\p_x \psi)^2}{2(1+ |\p_x h|^2)}},
\end{array}\right.
\end{equation}
where $G(h)\psi= \sqrt{1+ |\p_x h|^2} \mathcal{N}(h)\psi$ and $\mathcal{N}(h)\psi$ is the Dirichlet-Neumann operator associated with the domain $\Omega(t)$. The system (\ref{waterwave}) is generally referred as gravity water waves system, and it has the following conserved Hamiltonian,
\begin{equation}
\mathcal{H}(h, \psi):=  \int_{\R} \h \psi G(h)\psi + \h |h|^2 d x \approx \h[ \|h\|_{L^2}^2 + \| \d^{\h} \psi\|_{L^2}^2].
\end{equation}
\subsection{Previous results}
There is an extensive literature on the water waves problems. Without trying to be exhaustive, we only list some representative references on the Cauchy problem here. For blow up behavior and  splash singularity, please refer to \cite{fefferman} and references therein for more details.

We mention that  the motion of the boundary is subject to the Taylor instability when the surface tension effect is neglected. As proved by Wu \cite{wu1, wu2}, she showed that as long as the interface is nonself-intersecting, the Taylor sign condition $-{\p p}/{\p \mathbf{n}} \geq c_0> 0$  always holds  for the $n$-dimensional infinite depth gravity water wave equations, where $n\geq 2$ ($n=2,3$ for physical relevance). Hence the Taylor instability is not a issue here.

On the local theory side of the  water waves system, we  have the work of Nalimov \cite{nalimov}  and the work of Yosihara \cite{yosihara} for the small initial data, the works of Wu \cite{wu1, wu2} for the general initial data in Sobolev spaces and subsequent works by Christodoulou-Lindblad\cite{christodoulou}, Lannes \cite{lannes}, Lindblad\cite{lindblad}, Coutand-Shkoller \cite{coutand1}, Shatah-Zeng\cite{shatah1} and Alazard-Burq-Zuily\cite{alazard2}. Local wellposedness also holds when the surface tension effect is considered, see the work of Beyer-Gunther\cite{beyer}, Ambrose-Masmoudi\cite{ambrose}, Coutand-Shkoller\cite{coutand1},   Shatah-Zeng \cite{shatah1} and Alazard-Burq-Zuily\cite{alazard1}.

On the long time behavior side, starting with the breakthrough work of Wu\cite{wu3}, where she proved almost global existence for the $2D$ gravity water waves system for small initial data, then Germain-Masmoudi-Shatah \cite{germain2} and Wu \cite{wu4} proved global existence of gravity water waves in three dimensions. When the surface tension effect is considered while the gravity effect is neglected (also called capillary waves), Germain-Masmoudi-Shatah   \cite{germain3} proved the global existence of capillary waves in $3D$.

Due to the slower decay rate in  $2D$, it is considerably  more difficult to prove global existence than the 3D case. Until very recently, we have several results. Global existence of the 2D gravity water waves for small initial data has been first proved by Ionescu-Pusateri in \cite{IP1} and similar result was proved indepenedently by Alazard-Delort \cite{alazard} in Eulerian coordinates formulation. More recently, Hunter-Ifrim-Tataru \cite{tataru1} used a holomorphic coordinate formulation  to give a different proof of almost global existence result, then later Ifrim-Tataru \cite{tataru3} extended it to global existence in holomorphic  coordinates formulation.  Ionescu-Pusateri \cite{IP4} proved the global existence of 2D capillary waves for small initial data without momentum condition on the associated profile in the Eulerian coordinates formulation. Ifrim-Tataru  
\cite{tataru4} proved the global existence of the same system for small initial data  with momentum condition on the associated profile in the Holomorphic coordinates formulation.   
\subsection{Momentum condition}
We mentioned that the initial data  are all assumed to be small at the level of $L^2 \times \dot{H}^{1/2}$ in all previous results \cite{alazard,IP1, tataru3}. That is to say,  $\| \d^{1/2}\psi\|_{L^2}$ is not only finite but also small. Intuitively speaking, the following holds for the physical velocity, 
\begin{equation}
\| \d^{\h}\psi \|_{L^2}^2 =  \| \d^{-\h} v \|_{L^2}^2 = \int_{\R} \frac{1}{|\xi|} |\widehat{v}(\xi)|^2 d \xi < \infty.
\end{equation}
Hence, the finiteness of $\| \d^{\h}\psi \|_{L^2}^2$ implies that the physical velocity is neutral (i.e., $\widehat{v}(0)=\int v(x) dx =0$) and it behaves very nicely at the very low frequency part.  However, this assumption  is not generally true, even for a Schwartz function.  

So a natural question is what if the physical  velocity is not neutral. Do we still have global solution for the gravity water waves system (\ref{waterwave})? The main goal of this paper is to show that we do have global solution, infinite energy global solution.
 The $L^2$ norm of  $ \d^{\h}\psi$ is not necessary to be finite, hence the physical velocity at the interface is not necessary to be neutral. 
\subsection{The main result}

Before stating our main theorem, we define the main function spaces that will be used constantly.
Let $k_+$ denotes $\max\{k, 0\}$  and $k_{-}$ denotes $\min\{0, k\}$ throughout this paper. We define    function spaces as follows, 
\begin{equation}
\| f\|_{H^{N, p}} :=\big[ \sum_{k\in\mathbb{Z}} (2^{N k} + 2^{p k})^2  \| P_{k} f\|_{L^2}^2  \big]^{\h}\approx \| (|\p_x|^{p}+|\p_x|^{N}) f\|_{L^2},\quad 0\leq p\leq N,
\end{equation}
\begin{equation} 
\| f \|_{W^{\gamma, b}} := \sum_{k\in\mathbb{Z}} \big(2^{\gamma k_+ } + 2^{b k_{-}}\big) \| P_{k} f\|_{L^\infty} ,\quad \| f\|_{W^{\gamma}}: = \sum_{k \in \mathbb{Z}} (2^{N_2 k} + 1) \| P_k f \|_{L^\infty},
\end{equation}
\begin{equation}
\| f \|_{\widetilde{W^{\gamma}}}:= \| P_{\leq 0}f\|_{L^\infty} + \sum_{k\geq 0} 2^{\gamma k} \| P_k f\|_{L^\infty}, \quad \gamma \geq 0,
\end{equation}
\begin{equation}
\mathcal{C}_0:=\{ f: \R \rightarrow \mathbb{C}\, \textup{is continuous and } \lim_{|x|\rightarrow \infty} |f(x)| =0 \},\quad \|f \|_{\mathcal{C}_0}:=\| f \|_{L^\infty}.
\end{equation}
Note that the function space $\mathcal{C}_0$ is needed as the norm $H^{N,p}$ doesn't define natural spaces of distribution when $p > 1/2$.

Our main result is as follows, 
\begin{theorem}\label{theorem2}
Let $N_0=8$, $N_1=1$, $N_2=61/20$, $p =1/5$, and $p_0\in(0,10^{-10}]$.  Assume that $(h_0,\psi_0)\in {H}^{N_0+1/2, p}\times ({H}^{N_0+1/2,1/2+p}\cap \mathcal{C}_0)$ satisfies the following smallness condition,
\begin{equation}\label{sizeofinitialdata}
\| h_0\|_{{H}^{N_0+1/2, p}}  + \| \psi_0\|_{{H}^{N_0+1/2, 1/2+p}} + \| x\p_x h_0\|_{{H}^{N_1+1/2, p}} + \| x\p_x \psi_0\|_{{H}^{N_1+1/2,1/2+p}} \leq \epsilon_0,
\end{equation}
where ${\epsilon_0}$ is a sufficiently small constant, then there exists a unique global solution $(h, \psi)$ of the system \textup{(\ref{waterwave})} with initial data $(h_0, \psi_0)$. Moreover, the following estimates hold, 
\begin{equation}\label{eqn41}
\sup_{t\in[0, \infty)} (1+t)^{-p_0}\big[ \|h\|_{H^{N_0,p}} + \| \psi\|_{H^{N_0,1/2+p}} + \| S h\|_{H^{N_1,p}} + \| S \psi\|_{H^{N_1,1/2+p}}  \big]   \lesssim \epsilon_0,
\end{equation}
\begin{equation}\label{eqn42}
\sup_{t\in[0,\infty)} (1+t)^{1/2} \|( h,  \d^{\h}\psi)\|_{W^{N_2}}\lesssim \epsilon_0,
\end{equation}
where $S:= t\p_t +2 x \p_x$ is the scaling vector field associated with the system \textup{(\ref{waterwave})}.  
Furthermore, the solution possesses the modified scattering property as $t\rightarrow +\infty$. \end{theorem}

\begin{remark}
Note that $p=1/5$ in above theorem is not optimal, we can improve it after being more careful about the argument used here. However, given the methods used here, $p$ should strictly less than $1/4$.
\end{remark}
\begin{remark}
The property of modified scattering in the infinite energy setting is same as the small energy setting in \cite{IP1,IP5}. We first describe the modified scattering property and then give a comment. 
 Define $f(t):= h(t) + i \d^{1/2} \psi(t)$ and
\begin{equation}
G( t,\xi):= \frac{|\xi|^4}{\pi} \int_0^t |\widehat{f}( s,\xi)|^2 \frac{d s}{s+1},
\end{equation}
then there exist $\omega_{\infty} $ and $p_1 < p_0$ such that 
\begin{equation}\label{equation1}
\sup_{t\in[0,\infty)} (1+ t)^{p_1} \Big\| |\xi|^{1/4}(1+ |\xi|)^{N_2} \big( e^{i G(t,\xi) } \widehat{f}(t,\xi) - e^{-it |\xi|^{1/2}}\omega_{\infty}(\xi) \big) \Big\|_{L^2_\xi} \lesssim \epsilon_0.
\end{equation}
We can see that, after rotating the profile $\widehat{f}(t,\xi)$ with the angle of $G(t,\xi)$, it approaches to a linear solution.
We can also write this modified scattering behavior in the physical space $(t,x)$. Following the argument in \cite{IP5}, one 
can show that there exists a uniformly bounded function $f_{\infty}$ such that
\begin{equation}
\Big| f(t, x) - \frac{e^{- i t |t/4|x||}}{\sqrt{1+|t|}} f_{\infty}(\frac{x}{t}) \exp{\Big[- \frac{i |f_{\infty}(x/t)|^2}{64 |x/t|^5}  \log(1+|t|)\Big]} \Big| \lesssim \epsilon_0 (1+|t|)^{-1/2-p_1/2}.
\end{equation}
Note that the same formula has also been derived by Alazard-Delort\cite{alazard} and Ifrim-Tataru \cite{tataru3}. 

We comment that the modified scattering property is only a byproduct of the $L^\infty$ decay estimate. Note that the modified scattering  happens because of   the space-time resonance set, which is only a point. At this point, all inputs of cubic term have the same size of frequency, for this case, the symbol contributes ``$5/2$'' degrees  of  smallness. Although the spaces we use in the infinite energy setting are weaker than spaces used in the small energy setting, the effect of infinite energy setting is very little to the modified scattering.  We modify the profile in the same way as in the small energy setting to get the sharp $L^\infty$ decay rate and the modified scattering property. 
\end{remark}
\subsection{Main difficulties  and Main ideas of the proof}

A very essential observation, which  makes it possible to consider the infinite energy solution, is that the nonlinearities of the system (\ref{waterwave}) only depend on the steepness ``$\p_x h$'' and the physical velocity ``$\p_x \psi$'' in the infinite depth setting. In Appendix \ref{appdxb},  we will  also show this observation by using a fixed point type method to analyze the Dirichlet-Neumann operator.     This method  is not only interesting of its own, but also can be applied to other settings, for example, the flat bottom setting.

In the infinite energy setting, we also have to deal with the difficulties arose in the small energy setting. Those difficulties can be summarized as the losing derivatives issue and  the slow decay rate issue. Besides those difficulties, we also confront an additional difficulty, which comes from the fact that  we do not make any assumptions below $\dot{H}^{1/5}$. 
The additional difficulty is that we lose  $1/5$ derivatives at the low frequency part in the sense that only  $2^{ k/5}\| P_k u\|_{L^2}, k \leq 0$ is controlled when we put a input in $L^2$. The issue of losing derivative at the low frequency part is very problematic in the energy estimate of $Sh $ and $S\psi$. We will discuss more about this in a while.

To avoid losing derivatives at the high frequency part, we need to find out good substitution variables for $h$ and $\psi$. As the system (\ref{waterwave}) lacks symmetries, one fails at the beginning of energy estimate due to the quasilinear nature. Thanks to the work of Alazard-Delort \cite{alazard} and Alazard-Burq-Zuily\cite{alazard2}, their paralinearization method helps us to see the good structures inside the system (\ref{waterwave}) and to find out the good substitution variables $U^1$ and $U^2$. The system of equations satisfied by $U^1$ and $U^2$ has requisite  symmetries to avoid losing derivative when doing energy estimate. 

 For intuitive purpose, the  detailed formulas of $U^1$ and $U^2$ are postponed to (\ref{equation10}). It is enough to know at this moment that the difference between $(U^1, U^2)$ and $(\eta, \d^\h \psi)$ is at the level of quadratic and higher. Therefore, $(U^1, U^2)$ and  $(\eta, \d^\h \psi)$ have comparable sizes of  the $L^\infty$ decay rate  and the size of energy.

Since the local existence of (\ref{waterwave}) is already known, see \cite{wu1}. We will prove our main theorem by the standard bootstrap argument. To close the argument,  it is sufficient to control the following quantities over time, 
\[
\|(U^1, U^2)(t)\|_{H^{N_0, p}}, \quad \| (SU^1, SU^2)\|_{H^{N_1, p}},\quad \|(U^1, U^2)(t)\|_{W^{N_2}}.
\]

\subsubsection{Energy estimate for $(U^1, U^2)$} Now, we discuss main ideas of controlling \mbox{ $\|(U^1, U^2)(t)\|_{H^{N_0, p}}^2$} over time. As mentioned before, there are three difficulties: (i) losing derivatives at the high frequency part; (ii) slow decay rate; (iii) losing derivatives at the low frequency part. We will address how to get around those difficulties one by one.

Recall that the system of equations satisfied by $U^1$ and $U^2$ has requisite  symmetries inside, which help us to get around the potential difficulty of losing derivatives at the high frequency part. However, due to the slow $t^{-1/2}$ decay rate,  the  cubic terms inside the derivative of usual energy are troublesome. 

 To get around this issue, we will use the normal form transformation to find out  the cubic  correction terms, which can cancel out those  problematic terms. One very important observation is that the normal form transformation is not  singular. As the symbols of quadratic terms compensates the loss of doing the normal form transformation. 

 Then, we add those cubic correction terms to the usual energy. As a result, the derivative of the modified energy is quartic and higher, which has sufficient decay rate.

  Because of those cubic correction terms, due to the quasilinear nature, it is possible to lose derivatives again  after taking derivative with respect to time. To get around this issue, we also add quartic correction terms, which will enable us to utilize the symmetries inside the system of equations satisfied by $U^1$ and $U^2$ again to see cancellations. As a result, the derivative of the modified energy does not lose derivative any more and also has the critical $1/t$ decay rate. 

Now, we address the last issue, which is the losing $1/5$ derivatives at the low frequency part. For intuitive purpose, we use the following terms inside the time derivative of modified energy as an example,
\begin{equation}\label{equationdfdf2}
\int_{\R} \p_x^p \big( Q_1(U^1, U^2) + \mathcal{R}_1\big) \p_x^p A(U^1, U^1) d x, \quad p =1/5,
\end{equation}
where $Q_1(U^1, U^2)$ represents  the quadratic terms of $\p_t U^1$, $\mathcal{R}_1$ represents the cubic and higher remainder term of $\p_t U^1$,  $A(\cdot, \cdot)$ represents one of the normal forms that we will do.

The quartic term is not problematic. As we can always put the input with larger frequency in $L^2$ and the input with smaller frequency in $L^\infty$. Hence, the quartic term do not lose derivative at the low frequency part.

The remainder term is also not problematic. As the $L^2$-type and the $L^\infty$-type estimates of the Dirichlet-Neumann operator only depend on $\p_x h$ and $\p_x \psi$ as long as a certain smallness condition on $\p_x h$ is satisfied. This observation is sufficient to guarantee us to gain $1/5$ derivatives at low frequency for the profile $h + i \d^{1/2}\psi$.

\subsubsection{Energy estimate for $(SU^1, SU^2)$}
Many parts of the energy estimate of $SU^1$ and $ SU^2$  are same as the energy estimate of $U^1$ and $U^2$. We still need to utilize symmetries to avoid losing derivatives at the high frequency part and add correction terms to cancel out the slow decay cubic terms.  

However, a major difference is that inputs $SU^1$ and  $SU^2$  are forced to put in $L^2$ even they have relatively smaller frequency. For the case when $SU^1$ and $SU^2$ can be safely putted in $L^2$, we will redo these procedures that we did in the energy estimate of $U^1$ and $U^2$. Here we only concentrate on the case when $SU^1$ and $SU^2$ have relatively smaller frequencies and can not be safely putted in $L^2$. 

To illustrate this point, we use the following problematic cubic terms inside the derivative of usual energy of $(SU^1, SU^2)$ as example,
\begin{equation}\label{equation023i24}
\int_{\R} \p_x^p P_{k}(SU^1) \p_x^p P_{k}\big(Q_1(P_{k_1}(SU^1), P_{k_2}(U^1))\big), \quad k_1 \leq k_2-10, k_1, k_2, k\in \mathbb{Z}.
\end{equation}
Note that the symbols of quadratic terms contribute the smallness of $2^{k_1/2}$, which is very essential.

If one uses the normal form directly, then above problematic cubic terms vanish. However, we will have the following quartic term,
\begin{equation}\label{equation239489}
I:=\int_{\R} \p_x^p   Q_1(P_{k_1'} [SU^1], P_{k_2}U^2) \p_x^p A(P_{k_1}[SU^1], P_{k_2}[U^1]) d x, \quad k_1,k_1'\leq k_2-10.
\end{equation}
After putting $SU^1$ in $L^2$ and $U^i$, $i\in\{1,2\}$ in $L^\infty$, we have
\[
|I|\lesssim 2^{-pk_1-p k_1'} \frac{1}{t} E(t),\quad p=1/5,
\]
where $E(t)$ represents the energy of solution. Hence, this is not sufficient to close the argument when $k_1$ and $k_1'$ are sufficiently small. This difficulty does not appear in the small energy setting, as we have $p=0$ in the small energy setting.

To get around this difficulty, we will first identity the most problematic term by using a more subtle Fourier method and then show that the most problematic term is actually not bad in the sense that the associated phase has a good lower bound. As a result, we can divide the phase again to gain extra decay, which covers the lose of derivatives at the low frequency part.

 Note that if $2^{(1-p)k_1} \leq (1+t)^{-1}$, then the  cubic term in (\ref{equation023i24}) actually has $1/t$ decay. Therefore we only have to use the normal form transformation to cancel out (\ref{equation023i24})  when  $2^{(1-p)k_1} \geq (1+t)^{-1}$, i.e., $2^{k_1}\geq (1+t)^{-5/4}$.

  Next, we find out all problematic quartic terms. For intuitive purpose, we use quartic term in (\ref{equation239489}) as an example.  Note that the symbol in (\ref{equation239489}) contributes the smallness of $2^{k_1'/2}$. As a result, we know that when $k_1+10\geq k_1'$ or $2^{(1/2-p)k_1'} \leq (1+t)^{-1/4}$,  the loss of $2^{-p k_1-p k_1'}$ can be covered by the symbol. Hence, it is only problematic when $k_1+10\leq k_1'$ and  $2^{(1/2-p)k_1'} \geq (1+t)^{-1/4}$.

 A  very important observation for this problematic quartic term is that the size of associated phases has a good lower bound. More precisely,  the associated phases for quartic terms are given as follows, 
\[
\Phi^{\mu, \nu, \tau}(\xi,\eta, \sigma)= |\xi|^{1/2} -\mu |\xi-\eta|^{1/2} - \nu |\eta-\sigma|^{1/2} - \tau |\sigma|^{1/2}, \quad  \mu, \nu , \tau\in\{+,-\}.
\]
Recall that $k_2-10\geq k_1'\geq  k_1 +10$ in the problematic case. For this case, the following estimate holds,
\begin{equation}\label{equaiton239239}
|\Phi^{\mu, \nu, \tau}(\xi,\eta, \sigma)|\psi_{k_1'}(\xi)\psi_{k_2}(\xi-\eta)\psi_{k_1}(\eta-\sigma)\psi_{k_2}(\sigma) \gtrsim 2^{k_1'/2}.
\end{equation}
Hence, the price of dividing the phases $|\Phi^{\mu, \nu, \tau}(\xi,\eta, \sigma)|$ can be paid by the size of symbol in (\ref{equation239489}). As a result, we can gain an extra $t^{-1/2}$ decay rate, which is sufficient to cover the loss of $2^{-pk_1-pk_1'}\leq (1+t)^{1/4+1/6}$ from putting $SU^1$ in $\dot{H}^p$. Note that we used the fact that $2^{k_1}\geq (1+t)^{-5/4}$ and $2^{k_1'}\geq (1+t)^{-5/6}$.

For all other problematic quartic terms,  similar estimate as in (\ref{equaiton239239}) also holds, hence we can use the same argument to close the energy estimate of $SU^1$ and $SU^2$.

\subsubsection{The sharp $L^\infty$-decay estimate for $(U^1, U^2)$}

From the energy estimates we have proven and the linear dispersion estimate in Lemma \ref{lineardecay}, to prove sharp decay rate, it is sufficient to prove the  $L^\infty_\xi$-type norm(i.e., the $Z$-norm that will be precisely defined later) of profile does not grow with respect to time.

Note that one has to be very carefully when trying to define an appropriate $Z$-norm. Since  we cannot let $Z$-norm to be very strong, otherwise the finiteness of $Z$-norm implies the finiteness of energy. As a result, the $Z$-normed space we will use in this paper is weaker than the $Z$-normed space used in \cite{IP1}. To work in this weaker $Z$-normed space, we have to be more careful when doing estimates at the low frequency  part. It turns out that the difficulties are still manageable. The  ideas of proving the $L^\infty$ type decay estimate are mainly based on the works of Ionescu-Pusateri \cite{IP2, IP1,IP3,IP4}. To improve the understanding of $L^\infty$ decay estimate, many ideas are combined. As a result,  the argument we present here is more concise.

\subsection{Outline.} In section \ref{prelim}, we fix notations, find the  good unknown variables with good structure for the water waves system (\ref{waterwave}) , state the bootstrap assumption and then reduce the proof of main theorem \ref{theorem2} to two main propositions.  In section \ref{energyestimate}, we mainly prove that the energy of solution of (\ref{waterwave}) only grows appropriately.  In section \ref{improveddispersion}, we mainly prove the $L^\infty$-norm of solution of (\ref{waterwave}) decays sharply.   In appendix \ref{remainderestimate}, we show that the requisite estimates of all remainder terms also hold in the infinite energy setting. In appendix \ref{appdxb}, we use a fixed point type method to analyze the Dirichlet-Neumann operator. As a result, one can see that the Dirichlet-Neumann operator  only depends on the steepness of interface and the physical velocity, which is very essential to the existence of infinite energy solution. In appendix \ref{appendxc}, we do paralinearization for th full system (\ref{waterwave}) to show that the remainder terms do not lose derivatives.

\vo
\noindent\textbf{Acknowledgements.} 
The author would like to express gratitude to his Ph.D. advisor  Alexandru Ionescu for   many helpful discussions. Also he would like to thank the invitations and hospitalities  of Hausdorff Research Institute for Mathematics in Bonn and  Fudan University where a part of this work was done.

\section{Preliminaries}\label{prelim}

\subsection{Notations and the multilinear estimate.}
We  fix an even smooth function $\tilde{\psi}:\R \rightarrow [0,1]$ supported in $[-3/2,3/2]$ and equals to $1$ in $[-5/4, 5/4]$. For any $k\in \mathbb{Z}$, define
\[
\psi_{k}(x) := \tilde{\psi}(x/2^k) -\tilde{\psi}(x/2^{k-1}), \quad \psi_{\leq k}(x):= \tilde{\psi}(x/2^k), \quad \psi_{\geq k}(x):= 1-\psi_{\leq k-1}(x).
\]
Denote the projection operators $P_{k}$, $P_{\leq k}$ and $P_{\geq k}$ by the Fourier multipliers $\psi_{k},$ $\psi_{\leq k}$ and $\psi_{\geq k }$ respectively. We use  $f_k$ to abbreviate $P_k f$. For a well defined nonlinearity $\mathcal{N}$ and $p\in \mathbb{N}_{+}$, we will use $\Lambda_{p}(\mathcal{N})$  to denote the $p$-th order terms of the nonlinearity $\mathcal{N}$ when an Taylor expansion of this nonlinearity   is available. We use $\Lambda_{\geq p}(\mathcal{N}):= \sum_{q \geq p} \Lambda_{p}(\mathcal{N})$ ($\Lambda_{\leq p}(\mathcal{N})$) to denote the $p$-th and higher (lower) order terms of the nonlinearity $\mathcal{N}$. For example, $\Lambda_{3}(\mathcal{N})$    denotes the cubic terms of $\mathcal{N}$ and $\Lambda_{\geq 3}(\mathcal{N})$   denotes the cubic and higher terms of $\mathcal{N}$. The cubic and lower order terms of the Dirichlet-Neumann operator are given as follows, \[
\Lambda_{\leq 3}[G(h)\psi] = \d\psi - \d(h\d\psi)
\]
\begin{equation}\label{eqn1002}
-\p_x(h\p_x\psi) + \d(h\d(h\d\psi)) + \frac{\d(h^2 \p_x^2\psi) + \p_x^2(h^2\d\psi)}{2},
\end{equation}
 which can be found in \cite{alazard,IP1}.

The Fourier transform is defined as follows, 
\[
\widehat{f}(\xi)=\mathcal{F}(f)(\xi)= \int e^{-i x\xi} f(x) d x.
\]
For two well defined functions $f$ and $g$ and  a bilinear form  $Q(f,g)$, the symbol $q(\cdot, \cdot)$ of $Q(\cdot, \cdot)$  is defined in the following sense throughout this paper,
\begin{equation}
\mathcal{F}[Q(f,g)](\xi)= \frac{1}{2\pi}\int_{\R^2} \widehat{f}(\xi-\eta)\widehat{g}(\eta)q(\xi-\eta, \eta) d \eta. 
\end{equation}
For a trilinear form $C(f, g, h)$, its symbol $c(\cdot, \cdot, \cdot)$ is defined in the following sense, 
\[
\mathcal{F}[C(f,g,h)](\xi) = \frac{1}{4\pi^2} \int_{\R^2}\int_{\R^2} \widehat{f}(\xi-\eta)\widehat{g}(\eta-\sigma) \widehat{h}(\sigma) c(\xi-\eta, \eta-\sigma, \sigma) d \eta d \sigma.
\]

For  
$a, f\in L^2$ and pseudo differential operator $\tilde{a}(x,\xi)$, we define the operator $T_{a} f$ and $T_{\tilde{a}} f$ as follows,
\[
T_a f = \mathcal{F}^{-1}[\int_{\R} \widehat{a}(\xi-\eta) \theta(\xi-\eta, \eta)\widehat{f}(\eta) d\, \eta],
\]
\[
 T_{\tilde{a}} f = \mathcal{F}^{-1} [ \int_{\R} \mathcal{F}_x(\tilde{a})(\xi-\eta,\eta) \theta(\xi-\eta,\eta)\widehat{f}(\eta) d \eta  ],
\]
where the cut-off function is defined as follows, 
\[
\theta(|\xi-\eta|, \eta) = \left\{\begin{array}{ll}
1 & \textup{when}\,\,|\xi-\eta|\leq 2^{-10} |\eta|,\\
0 & \textup{when}\,\, |\xi-\eta| \geq 2^{10} |\eta|.\\
\end{array}\right.
\]
We also use $\tilde{\theta}(\xi-\eta, \eta)$ to denote cutoff function where two frequencies have comparable size inside the support. More precisely,
\begin{equation}\label{equation350}
\tilde{\theta}(\xi-\eta, \eta):= 1- \theta(\xi-\eta, \eta)-\theta(\eta, \xi-\eta).
\end{equation}

Let the bilinear operator $R_{\mathcal{B}}(\cdot, \cdot)$ to be defined by the Fourier multiplier $\tilde{\theta}(\cdot, \cdot).$  Hence, the following paraproduct decomposition holds for two well defined functions $a$ and $b$,
\begin{equation}\label{equation340}
a b = T_a b + T_b a + R_{\mathcal{B}}(a,b).
\end{equation}

\begin{definition}
Given $\rho\in \mathbb{N}_{+}, \rho \geq 0$ and $m\in \R$, we use $\Gamma^{m}_{\rho}(\R^n)$  to denote the space of locally bounded functions $a(x,\xi)$ on $\R^n\times (\R^n/\{0\})$, which are $C^{\infty}$ with respect to $\xi$ for $\xi\neq 0 $. Moreover, they satisfy the following estimate, 
\[
\forall |\xi|\geq 1/2, \| \p_{\xi}^{\alpha} a(\cdot, \xi)\|_{W^{\rho,\infty}}\lesssim_{\alpha} (1+|\xi|)^{m-|\alpha|}, \quad \alpha\in \mathbb{N}^n,
\]
where $W^{\rho, \infty}$ is the usual Sobolev space. 
\end{definition}
For symbol $a\in \Gamma^{m}_{\rho}$, we can define its norm as follows,
\[
M^{m}_{\rho}(a):= \sup_{|\alpha|\leq 2+\rho} \sup_{|\xi|\geq 1/2} \| (1+|\xi|)^{|\alpha|-m}\p_{\xi}^{\alpha} a (\cdot, \xi)\|_{W^{\rho, \infty}}.
\]

 We will use  the following composition lemma for paradifferential operators very often. It can be found, for example, in \cite{ para,alazard1}.
\begin{lemma}\label{composi}
Let $m\in \R$ and $\rho >0,$ if given symbols $a\in \Gamma_{\rho}^{m}(\R^d)$ and $b\in\Gamma_{\rho}^{m'}(\R^d)$,  we can define
\[
a\sharp b = \sum_{|\alpha|< \rho} \frac{1}{i^{|\alpha|} \alpha!} \p_{\xi}^{\alpha} a \p_{x}^{\alpha}b,
\]
then for all $\mu\in\R$, there exists a constant $K$ such that 
\begin{equation}\label{eqn700}
\| T_a T_b - T_{a\sharp b}\|_{H^{\mu}\rightarrow H^{\mu-m-m'+\rho}} \leq K M^{m}_{\rho}(a) M^{m'}_{\rho}(b).
\end{equation}
\end{lemma}

 We will use the following fact very often. Recall that the scaling vector field is defined as $S: = t \p_t + 2 x \p_x$.  For a bilinear operator $A(\cdot, \cdot)$ with  symbol $a(\cdot, \cdot)$, then we have
\[
S A(f, g)= A(Sf, g) + A(f, Sg) + \widetilde{A}(f, g), \quad \tilde{a}(\xi, \eta) = - 2 (\xi \p_x{\xi} a(\xi, \eta) + \eta \p_{\eta} a(\xi, \eta)).
\]
Especially,  when $a(\cdot, \cdot)$ is homogeneous of degree $\lambda$, we have
\[
S A(f, g) = A(S f, g) + A(f, S g) - 2 \lambda A(f, g),
\]
one can easily verify it after observing the equality $\xi\p_{\xi} a(\xi, \eta)+ \eta\p_{\eta} a(\xi, \eta)=  \lambda a(\xi, \eta).$

Define a class of symbol and its associated norms as follows,
\[
\mathcal{S}^\infty:=\{ m: \mathbb{R}^2\,\textup{or}\, \mathbb{R}^3 \rightarrow \mathbb{C}, m\,\textup{is continuous and }  \quad \| \mathcal{F}^{-1}(m)\|_{L^1} < \infty\},
\]
\[
\| m\|_{\mathcal{S}^\infty}:=\|\mathcal{F}^{-1}(m)\|_{L^1}, \quad \|m(\xi,\eta)\|_{\mathcal{S}^\infty_{k,k_1,k_2}}:=\|m(\xi, \eta)\psi_k(\xi)\psi_{k_1}(\xi-\eta)\psi_{k_2}(\eta)\|_{\mathcal{S}^\infty},
\]
\[
 \|m(\xi,\eta,\sigma)\|_{\mathcal{S}^\infty_{k,k_1,k_2,k_3}}:=\|m(\xi, \eta,\sigma)\psi_k(\xi)\psi_{k_1}(\xi-\eta)\psi_{k_2}(\eta-\sigma)\psi_{k_3}(\sigma)\|_{\mathcal{S}^\infty}.
\]
\begin{lemma}\label{boundness}
Assume that $m$, $m'\in S^\infty$, $p, q, r, s \in[1, \infty]$ and we have smooth well defined functions $f, g$, $h$ and $\tilde{f}$, then the following estimates hold:
\begin{equation}\label{productofsymbol}
\| m\cdot m'\|_{S^\infty} \lesssim \| m \|_{S^\infty}\| m'\|_{S^\infty},
\end{equation}
\begin{equation}\label{bilinearestimate}
\big\| \mathcal{F}^{-1}\big[ \int_{\R } m(\xi, \eta)  \widehat{g}(\eta)\widehat{h}(\xi-\eta)\, d \eta\big] \big\|_{L^r}\lesssim \| m\|_{\mathcal{S}^\infty} \| g \|_{L^p} \| h\|_{L^q}, \quad \textup{if}\,\,\, \frac{1}{r}=\frac{1}{p} + \frac{1}{q},
\end{equation}
\begin{equation}\label{trilinearestimate}
\big\| \mathcal{F}^{-1}\big[ \int_{\R^2 } m'(\xi, \eta,\sigma)  \widehat{f}(\sigma) \widehat{g}(\eta-\sigma)\widehat{h}(\xi-\eta) d \eta d\sigma \big\|_{L^s} \lesssim \| m'\|_{\mathcal{S}^\infty} \| f \|_{L^p}\| g \|_{L^q} \| h\|_{L^r},\,\, 
\end{equation}
if $ \displaystyle{ \frac{1}{s}=\frac{1}{p} + \frac{1}{q} + \frac{1}{r}}$.
\end{lemma}

\begin{proof}
The proof is standard, or see \cite{IP2, IP1} for detail.
\end{proof}
 To estimate the $\mathcal{S}^{\infty}_{k,k_1,k_2}$  or the $\mathcal{S}^{\infty}_{k,k_1,k_2,k_3}$ norms of symbols, we   constantly use the following lemma . 
\begin{lemma}\label{Snorm}
For $i\in\{2,3\}, $ if $f:\mathbb{R}^{i}\rightarrow \mathbb{C}$ is a smooth function and $k_1,\cdots, k_i\in\mathbb{Z}$, then the following estimate holds,
\begin{equation}\label{eqn61001}
\| \int_{\mathbb{R}^{i}} f(\xi_1,\cdots, \xi_i) \prod_{j=1}^{i} e^{i x_j \xi_j} \psi_{k_j}(\xi_j) d \xi_1\cdots  d\xi_i \|_{L^1_{x_1, \cdots, x_i}} \lesssim \sum_{m=0}^{i}\sum_{j=1}^i 2^{m k_j}\|\p_{\xi_j}^m f\|_{L^\infty} .
 \end{equation}
\end{lemma}
\begin{proof}
We only consider the case when $i=2$ here, as the proof of the case when $i=3$ is very similarly. Through scaling, it is  sufficient to prove above estimate for the case when $k_1=k_2=0$. From the Plancherel theorem, we have the following two estimates, 
\[
 \| \int_{\mathbb{R}^{i}} f(\xi_1,\xi_2) e^{i (x_1 \xi_1+ x_2 \xi_2)} \psi_{0}(\xi_1) \psi_{0}(\xi_2) d \xi_1 d\xi_2 \|_{L^2_{x_1, x_2}}\lesssim \| f(\xi_1, \xi_2)\|_{L^\infty_{\xi_1, \xi_2}},
\]
\[
 \| (|x_1|+|x_2|)^2 \int_{\mathbb{R}^{i}} f(\xi_1,\xi_2) e^{i (x_1 \xi_1+ x_2\xi_2)} \psi_{0}(\xi_1) \psi_{0}(\xi_2) d \xi_1 d\xi_2 \|_{L^2_{x_1, x_2}}\]
 \[\lesssim  \sum_{m=0}^2\big[\|\p_{\xi_1}^m f\|_{L^\infty} + \| \p_{\xi_2}^m f \|_{L^\infty}\big],
\]
which are sufficient to finish the proof of (\ref{eqn61001}).  
\end{proof}

\subsection{The good substitution variables}\label{mainsystemofequation}
 At the very beginning, we first mention that the content of this subsection is not new, we briefly introduce and explain the main ideas behind here. There are more details in the Appendix \ref{appendxc}. Also those interested readers can refer to \cite{alazard} for more elaborated details. 

The main process of deriving the good substitution variables can be summarized as paralinearization process and symmetrization process. Before we introducing main ideas of doing those processes, we  define the following variables,
\begin{equation}\label{eqn1003}
B:=B(h)\psi:= \frac{G(h)\psi + \p_x h \p_x \psi}{1+ |\p_x h|^2}, \,\, V:=V(h)\psi := \p_x\psi - \p_x h B(h)\psi,\,\,
\end{equation}
\begin{equation}\label{eqn1004}
 a=1+\p_t B+ V \p_x B, \quad \alpha = \sqrt{a}-1,
\end{equation}
where $B$ and $V$  represent the vertical derivative and the horizontal derivative of the velocity potential at the interface respectively and  $a$ is actually the so called Taylor coefficient.

As a result of paralinearization (one can also see Lemma \ref{paralinearizationDN1}), we have the following decomposition of the Dirichlet-Neumann operator, 
 \begin{equation}\label{remainderpara}
G(h)\psi = \d \omega - \p_x(T_{V(h)\psi}\eta) + F(h)\psi,\quad \omega := \psi - T_{B(h)\psi} h,
\end{equation}
where $F(h)\psi$ is a quadratic and higher good error term, which does not lose derivatives. Recall (\ref{waterwave}).
After doing paralinearization for the nonlinearity of $\p_t \psi$, we have
\[
- \frac{1}{2} |\p_x\psi|^2 + \displaystyle{\frac{(G(h)\psi + \p_x h\p_x \psi)^2}{2(1+ |\p_x h|^2)}} = - \frac{1}{2} |\p_x\psi|^2 + \h (1+|\p_x h |^2)(B(h)\psi)^2
\]
\begin{equation}\label{equation4}
= -T_{V}\p_x (\psi-T_{B} h) - T_{V\p_x B} h + T_{B}G(h)\psi + \mathcal{R},
\end{equation}
where $\mathcal{R}$ is also a quadratic and higher good error term, which does not lose derivatives. The main point of the  paralinearization process is to identify the cancellations and highlight the quasilinear structures inside the nonlinearities of  the  system (\ref{waterwave}).

Recall the definition of $\omega$ in (\ref{remainderpara}). From (\ref{equation4}), we have
\[
\p_t \omega = - h - T_{V}\p_x \omega - T_{V\p_x B} h  + T_{B}\p_t h- T_{\p_t B}h -T_{B}\p_t h + \mathcal{R} 
= - T_{a} h - T_{V}\p_x \omega + \mathcal{R}.
\]
To sum up, we have
\begin{equation}\label{eqn51200}
\left\{\begin{array}{l}
\p_t h= \d \omega - T_{V}\p_x h +F(h)\psi,\\
\p_t \omega=-T_{a} h -T_{V}\p_x \omega  + \mathcal{R}.
\end{array}\right.
\end{equation}
One might find that the system (\ref{eqn51200}) still cannot be used to do energy estimate, as the following term loses derivatives and it cannot be simply treated,
\begin{equation}\label{equation9}
\int \p_x^k h \p_x^k \d \omega -\p_x^k \d^{1/2} \omega \p_x^j \d^{1/2}(T_a h).
\end{equation}
To get around this issue, we need to symmetrizing the system (\ref{eqn51200}) by using the following good variables
\begin{equation}\label{equation10}
U^1= T_{\sqrt{a}} h=   h + T_{\alpha}h,\quad U^2= \d^\h \omega.
\end{equation}
As a result, from (\ref{eqn51200}), the system of equations satisfied by $U^1$ and $U^2$ can be formulated as follows, 
\begin{equation}\label{equation6}
\left\{\begin{array}{l}
\p_t U^1- \d^{1/2}U^2=  T_{\alpha} \d^{1/2} U^2 - T_{V}\p_x U^1 +\mathcal{R}_1,\\
\p_t U^2 + \d^{1/2} U^1=-T_{\alpha} \d^{1/2} U^1 -T_{V}\p_x U^2  + \mathcal{R}_2,
\end{array}\right.
\end{equation}
where $\mathcal{R}_1$ and $\mathcal{R}_2$ are quadratic and higher good error terms that do not lose derivatives. Their detailed formulas are given in (\ref{equation351}) and (\ref{equation352}).

 As a result of symmetrization, the bulk term (\ref{equation9}), essentially speaking, becomes to  the following term
\[
\int \p_x^k U^1 \p_x^k (T_{\sqrt{a}} \d^{1/2} U^2)- \p_x^k U^2 \p_x^k (T_{\sqrt{a}} \d^{1/2} U^1) dx,
\] 
which does not lose derivatives. One can verify this fact by   utilizing symmetries on the Fourier side. 

Since we will cancel out the quadratic terms inside the system (\ref{equation6}), it is important to know what quadratic terms are. Till the end of this section, \emph{we assume that the expansion of the  projection operator $\Lambda_{\geq 3}[\cdot]$ is taken with respect to $U^1$ and $U^2$}.

From (\ref{eqn1002}) and (\ref{equation10}),  we can  rewrite (\ref{equation6}) as follows, 
  \begin{equation}\label{mainequation}
\left\{ \begin{array}{l}
\p_{t} U^{1} - \d^{1/2} U^{2} =  Q_{1}(U^{1}, U^{2}) + C_1 +\Lambda_{\geq 3}[\mathcal{R}_1] \\
\pt U^{2} + \d^{1/2} U^{1} =  Q_{2}(U^{1}, U^{1} ) + Q_3 (U^2, U^2)  + C_2 + \Lambda_{\geq 3}[ \mathcal{R}_2], \\
\end{array}\right.
\end{equation}
where $Q_1(\cdot, \cdot)$, $Q_2(\cdot, \cdot)$ and $Q_3(\cdot, \cdot)$ denote quadratic terms and $C_1$ and $C_2$ denote the cubic and higher order terms that at most lose $1$ derivative. More precisely, 
\[
Q_{1}(U^{1}, U^{2}) = - T_{\px \d^{-\h} U^{2}} \px U^{1}- \h T_{\px^{2}\d^{-\h}U^{2}} U^{1}   - 
\h T_{\d U^{1}}\d^{\h} U^{2} +
\]
\[
-\d ( U^{1} \d^{\h} U^{2}) + \d (T_{ \d^{\h} U^{2}} U^{1}) - \px (U^{1} \px \d^{-\h} U^{2}) + \px(T_{\px \d^{-\h} U^{2}} U^{1} ),
\]
\[
Q_{2}(U^{1}, U^{1}) = \frac{1}{2} \d^{\h} T_{\d U^{1}} U^{1},\quad Q_3(U^2, U^2)=   - \d^{\h} T_{\px \d^{-\h} U^{2} } \px\d^{-1/2} U^{2},
\]
\begin{equation}\label{equation29}
 C_1=   T_{\Lambda_{\geq 2}[\alpha]}\d^{1/2} U^2 -T_{\Lambda_{\geq 2}[V]} \p_x U^1, C_2 = -T_{\Lambda_{\geq 2}[\alpha]} \d^{1/2} U^1- T_{\Lambda_{\geq 2}[V]} \p_x U^2. 
\end{equation}
We remark that the good error terms $\mathcal{R}_1$ and $\mathcal{R}_2$ contribute the following quadratic terms, 
\[
\Lambda_{2}[\mathcal{R}_1](U^1, U^2)= Q_1(U^1, U^2) +\frac{1}{2}T_{\d U^1}\d^{1/2} U^2 + T_{\p_x \d^{-1/2}U^2} \p_x U^1,
\] 
\[
\Lambda_{2}[\mathcal{R}_2](U^1, U^1) = Q_{2}(U^1, U^1) - \h T_{\d U^1}\d^{1/2} U^1, 
\]
\[
\Lambda_{2}[\mathcal{R}_2](U^2, U^2) = Q_{3}(U^2, U^2) +  T_{\p_x\d^{-1/2} U^2}\p_x U^2.
\]
The corresponding symbols of  quadratic terms are given as follows,
\begin{equation}\label{equation50000}
 q_{1}(\xi-\eta, \eta) =\sum_{i=1,2,3}  q_{1}^i(\xi-\eta, \eta),q_{1}^1(\xi-\eta, \eta)=\big((\xi-\eta)\eta |\eta|^{-1/2}+  |\eta|^{3/2}/2\big)\theta(\eta, \xi-\eta)
 \end{equation}
\begin{equation}\label{eqn50001}
q_{1}^2(\xi-\eta, \eta)= -  |\xi-\eta| |\eta|^{1/2} \theta(\xi-\eta, \eta)/2, q_{1}^3(\xi-\eta, \eta)= (-|\xi||\eta|^{1/2} + \xi\eta|\eta|^{-1/2})\tilde{\theta}(\eta, \xi-\eta),
\end{equation}
\begin{equation}\label{eqn50002}
q_2(\xi-\eta, \eta)=  |\xi|^{1/2} |\eta| \theta(\eta, \xi-\eta)/2,\end{equation}
\begin{equation}\label{eqn50003}
q_3(\xi-\eta, \eta) = |\xi|^{1/2} (\xi-\eta)\eta |\xi-\eta|^{-1/2} |\eta|^{-1/2} \theta(\eta,\
\xi-\eta),
\end{equation}
where $q_{i}(\cdot, \cdot)$  $i\in\{1,2,3\}$, are the symbols of operators $Q_{i}(\cdot, \cdot)$.
\begin{lemma}\label{sizeofsymbol}
If $|k_1-k_2|\geq 5$, the following estimates hold,
\begin{equation}\label{equation20}
\sum_{i=1,2,3}
\| q_i(\xi-\eta, \eta)\|_{\mathcal{S}^\infty_{k,k_1,k_2}} \lesssim 2^{\min\{k_1,k_2\}/2+ \max\{k_1,k_2\}}.
\end{equation}
\[
\| q_1^1(\xi-\eta, \eta)+ q_1^1(-\xi,\eta)\|_{\mathcal{S}^\infty_{k,k_1,k_2}} + \| q_2(\eta,\xi-\eta)+ q_1^2(\eta-\xi,\xi)\|_{\mathcal{S}^\infty_{k,k_1,k_2}}
\]
\begin{equation}\label{equation36}
 +\| q_3(\xi-\eta, \eta)+ q_3(-\xi,\eta)\|_{\mathcal{S}^\infty_{k,k_1,k_2}} \lesssim 2^{3\min\{k_1,k_2\}/2}.
\end{equation}
If $|k_1-k_2|\leq 5$, the following estimate holds,
\begin{equation}\label{equation21}
\sum_{i=1,2,3}
\| q_i(\xi-\eta, \eta)\|_{\mathcal{S}^\infty_{k,k_1,k_2}} \lesssim 2^{k+k_1/2}.
\end{equation}
\end{lemma}
\begin{proof}
From Lemma \ref{Snorm} and explicit  formulas in (\ref{equation50000}), (\ref{eqn50001}), (\ref{eqn50002}), and (\ref{eqn50003}), it is easy to see   that our stated estimates hold.
\end{proof}
\subsection{The bootstrap assumption and proof of the main theorem}
Since the sizes of $(U^1, U^2)$ and $(h, \d^{1/2}\psi)$ are comparable, 
from the smallness assumption of initial data $(\eta_0, \psi_0)$ in (\ref{sizeofinitialdata}), it's not difficult  to see the following estimate holds,
\begin{equation}
\| (U^1,  U^2)(0)\|_{H^{N_0, p}} + \| x\p_x (U^1, U^2)(0)\|_{H^{N_1, p}} \lesssim \epsilon_0.
\end{equation}
We will use the bootstrap argument to show the global existence of solutions of the system (\ref{mainequation}), which further gives us the global existence of  $(h, \psi)$. Same as in the small energy setting,  we expect that the energy grows appropriately and the decay rate of $L^\infty$-type norm is sharp. This expectation leads to the following bootstrap assumption,
\[
  \sup_{t\in[0,T]} (1+t)^{-p_0}[\| (U^1, U^2)(t)\|_{H^{N_0, p}} + \| S (U^1, U^2)(t)\|_{H^{N_1, p}}]
\]
\begin{equation}\label{assumption}
 +(1+t)^{1/2}\| (U^1, U^2)\|_{W^{N_2}} \lesssim \epsilon_1:=\epsilon_0^{5/6}\ll 1.
  \end{equation}

 In section \ref{energyestimate}, we will prove the following proposition, which is sufficient to show that the total energy  appropriately grows.
\begin{proposition}\label{mainproposition1}
Under the bootstrap assumption \textup{(\ref{assumption})}, we can define modified energies $E_{modi}(t) \approx \|(U^1, U^2)(t)\|_{H^{N_0,p}}^2$ and $E_{modi}^S(t)$ $\approx \| (SU^1, SU^2)(t)\|_{H^{N_1, p}}^2$ and have the following energy estimate, 
\begin{equation}\label{equation11}
\sup_{t\in[0,T]} (1+t)^{1-2p_0} \big[|\frac{d}{d t} E_{modi}(t)| + |\frac{d}{d t} E_{modi}^S(t)|\big] \lesssim \epsilon_0.
\end{equation}
Therefore, 
\begin{equation}\label{equation12}
\sup_{t\in[0,T]} (1+t)^{-p_0}\big[\| (U^1, U^2)\|_{H^{N_0, p}} + \| S(U^1, U^2)\|_{H^{N_1, p}} \big] \lesssim \epsilon_0.
\end{equation}
\end{proposition}

 In section \ref{improveddispersion}, we will prove the following decay estimate for the $L^\infty$-type norm.
\begin{proposition}\label{mainproposition2}
Under the bootstrap assumption \textup{(\ref{assumption})} and the improved energy estimate \textup{(\ref{equation12})},
we can derive the following improved decay estimate,
\begin{equation}
\sup_{t\in[0,T]} (1+t)^{1/2}\|(U^1, U^2)\|_{W^{N_2}} \lesssim \epsilon_0.
\end{equation}
\end{proposition}

\noindent With above two propositions, it's easy to see our main theorem holds.

\section{Energy estimate}\label{energyestimate}

\subsection{Normal form transformation}\label{normalform}
We  first find out the normal form transformations that can cancel out the quadratic terms. Let
\begin{equation}\label{eqn1600}
V_{1} := U^{1} + A_{1} (U^{1}, U^{1}) + A_{2} (U^{2}, U^{2}), \quad V_{2} := U^{2} + B(U^{1}, U^{2}),
\end{equation}
where $A_1(\cdot, \cdot) $ and $A_2(\cdot, \cdot)$ are two symmetric bilinear operators. It's easy to derive the quadratic terms inside the equations satisfied by $V_1$ and $V_2$ as follows,
\[
  Q_{1}(U^{1}, U^{2}) + 2 A_{1}(\d^{1/2} U^{2}, U^{1}) -2 A_{2}(\d^{1/2}U^{1}, U^{2})- \d^{1/2} B(U^{1}, U^{2}),\]
\[
 Q_{2}(U^{1}, U^{1}) + Q_{3} (U^{2}, U^{2}) + B(\d^{1/2} U^{2}, U^{2})  - B(U^{1}, \d^{1/2}U^{1})\]
\[
  +\d^{1/2} ( A_{1}(U^{1}, U^{1}) + A_{2}(U^{2}, U^{2})).
\]
To make above quadratic terms vanish, it would be sufficient if the symbols of  bilinear operators $A_1(\cdot,\cdot)$, $A_2(\cdot, \cdot)$ and $B(\cdot, \cdot)$ satisfy the following system of equations:
\begin{equation}\label{equation1100}
\left\{\begin{array}{c}
q_{1}(\xi-\eta,\eta) +  2   |\eta|^{1/2} a_1(\xi-\eta, \eta) -  2  |\xi-\eta|^{1/2} a_2(\xi-\eta, \eta) - |\xi|^{1/2} b(\xi-\eta, \eta)=0,\\
q_{2}(\xi-\eta,\eta)- b(\xi-\eta,\eta) |\eta|^{1/2} + q_{2}(\eta,\xi-\eta) - b(\eta,\xi-\eta)|\xi-\eta|^{1/2}\\
 + 2|\xi|^{1/2} a_{1}(\xi-\eta, \eta) =0,\\
q_{3}(\xi-\eta,\eta) +q_{3}(\eta,\xi-\eta) + b(\xi-\eta,\eta) |\xi-\eta|^{1/2} + b(\eta,\xi-\eta) |\eta|^{1/2} \\
+ 2 |\xi|^{1/2} a_{2}(\xi-\eta, \eta)=0.\\
\end{array}\right.
\end{equation}
The solution of above system is given as follows,
\begin{equation}\label{equation8900}
a_{1}(\xi-\eta, \eta) = \frac{b(\xi-\eta,\eta)|\eta|^{1/2} - q_{2}(\xi-\eta, \eta)+b(\eta,\xi-\eta)|\xi-\eta|^{1/2} - q_{2}(\eta,\xi- \eta)}{2|\xi|^{1/2}},
\end{equation}
\begin{equation}\label{equation8901}
 a_{2}(\xi-\eta, \eta) = -\frac{b(\xi-\eta,\eta)|\xi-\eta|^{1/2} + q_{3}(\xi-\eta, \eta)+b(\eta,\xi-\eta)|\eta|^{1/2} + q_{3}(\eta,\xi- \eta)}{2|\xi|^{1/2}} ,
\end{equation}
\begin{equation}\label{equation8902}
b(\xi-\eta,\eta) = \frac{ (|\xi-\eta| + |\eta| - |\xi|) A(\xi-\eta,\eta) - 2 A(\eta,\xi-\eta)|\xi-\eta|^{1/2}|\eta|^{1/2}  }{- (|\xi-\eta| + |\eta| - |\xi|)^{2} + 4 |\xi-\eta||\eta|},
\end{equation}
where
\[
A(\xi-\eta,\eta) := |\xi|^{1/2} q_{1}(\xi-\eta,\eta) - (q_{2}(\xi-\eta,\eta) + q_{2}(\eta,\xi-\eta))|\eta|^{1/2}\]
\[ + (q_{3}(\xi-\eta,\eta) + q_{3}(\eta,\xi-\eta))|\xi-\eta|^{1/2}.
\]

\begin{lemma}\label{auxilary1}
The following estimate holds,\begin{equation}\label{equation8905}
\|a_1(\xi-\eta, \eta)\|_{\mathcal{S}^\infty_{k,k_1,k_2}}+\|a_2(\xi-\eta, \eta)\|_{\mathcal{S}^\infty_{k,k_1,k_2}} + \|b(\xi-\eta, \eta)\|_{\mathcal{S}^\infty_{k,k_1,k_2}} \lesssim  2^{\max\{k_1,k_2\}}.
\end{equation}
\end{lemma}
\begin{proof}
If $|k_1-k_2|\geq 10$, then estimate (\ref{equation20}) in Lemma \ref{sizeofsymbol} holds. 
From (\ref{equation8900}), (\ref{equation8901}), (\ref{equation8902}) and Lemma \ref{Snorm}, the following estimate holds, 
\begin{equation}
\|a_1(\xi-\eta, \eta)\|_{\mathcal{S}^\infty_{k,k_1,k_2}}+\|a_2(\xi-\eta, \eta)\|_{\mathcal{S}^\infty_{k,k_1,k_2}} + \|b(\xi-\eta, \eta)\|_{\mathcal{S}^\infty_{k,k_1,k_2}} \lesssim 2^{\max\{k_1, k_2\}}. 
\end{equation}
If  $|k_1 -k_2|\leq 10$, then estimate (\ref{equation21}) in Lemma \ref{sizeofsymbol} holds. Note that $q_2(\xi-\eta, \eta)=q_3(\xi-\eta,\eta)=0$ for this case. Hence, from (\ref{productofsymbol}) in Lemma \ref{boundness}, the following estimate holds,
\begin{equation}\label{equation8903}
\|A(\xi-\eta, \eta)\|_{\mathcal{S}^\infty_{k,k_1,k_2}} + \|A(\eta, \xi-\eta)\|_{\mathcal{S}^\infty_{k,k_1,k_2}} \lesssim 2^{3k/2+k_1/2},
\end{equation}
which further gives us the following estimate, 
\begin{equation}
\|a_1(\xi-\eta, \eta)\|_{\mathcal{S}^\infty_{k,k_1,k_2}}+\|a_2(\xi-\eta, \eta)\|_{\mathcal{S}^\infty_{k,k_1,k_2}} + \|b(\xi-\eta, \eta)\|_{\mathcal{S}^\infty_{k,k_1,k_2}}\lesssim 2^{k_1}.
\end{equation}
\end{proof}

\subsection{Energy estimate of $U^1$ and $U^2$}\label{usualenergyestimate}
As the cubic terms inside the time derivative of energy decay slowly over time, they are problematic when doing energy estimate. Recall (\ref{mainequation}). The cubic terms at the top derivative level are given as follows, 
\begin{equation}\label{equation9989}
\mathcal{I}_{1}^{N_0} = \mathfrak{Re} \Big[  \int_{\R} \overline{ \p_x^{N_0} U^1  } \p_x^{N_0} \big[Q_1(U^1, U^2)\big] + \overline{\p_x^{N_0} U^2} \p_x^{N_0}\big[Q_2(U^1, U^1) + Q_3(U^2, U^2)\big]\Big].
\end{equation}
Due to the quasilinear nature, to avoid losing derivative again after adding the cubic correction terms, we utilize symmetries inside the system (\ref{mainequation}) first. After switching the roles of same type inputs inside $\mathcal{I}_{1}^{N_0}$ on  the Fourier side, we can  rewrite it as follows,
\[
  \mathcal{I}_{1}^{N_0} = \mathfrak{Re} \Big[ \int \frac{1}{2}\overline{\widehat{ U^1 }(\xi)} \tilde{q}_{N_0}^1 (\xi-\eta, \eta)
\widehat{U^1}(\xi-\eta) \widehat{U^2}(\eta) +  \overline{\widehat{U^2}(\xi)} \widetilde{q}_{N_0}^2 (\xi-\eta, \eta) \widehat{U^1}(\xi-\eta) \widehat{U^1}(\eta) \]
\begin{equation}
+ \int \h \overline{\widehat{U^2}(\xi)} \tilde{q}_{N_0}^3(\xi-\eta, \eta) \widehat{U^2}(\xi-\eta)\widehat{U^2}(\eta)\Big],
\end{equation}
where
\[
\tilde{q}_{N_0}^1(\xi-\eta, \eta)= q_1^1(\xi-\eta, \eta) |\xi|^{2N_0} + q_1^1(-\xi, \eta)|\xi-\eta|^{2N_0}+ 2 q_1^3(\xi-\eta,\eta) |\xi|^{2N_0},
\]
\[
\tilde{q}^2_{N_0}(\xi-\eta,\eta) := q_2(\eta, \xi-\eta)|\xi|^{2N_0} + q_1^2(\eta-\xi, \xi)|\eta|^{2N_0},
\]
\begin{equation}
\tilde{q}_{N_0}^3(\xi-\eta, \eta):= q_3(\xi-\eta, \eta)|\xi|^{2 N_0} + q_3(-\xi, \eta)|\xi-\eta|^{2 N_0}.
\end{equation}

 Recall (\ref{eqn50001}), (\ref{eqn50002}) and (\ref{eqn50003}). From (\ref{equation36}) in Lemma \ref{sizeofsymbol}, we can see that  cancellations happen in $\tilde{q}_{N_0}^i(\cdot, \cdot)$, $i\in\{1,2,3\},$. From  (\ref{equation36}) in Lemma \ref{sizeofsymbol} and (\ref{productofsymbol}) in Lemma \ref{boundness}, the following estimate holds,
\begin{equation}\label{equation8801}
\sum_{i=1,2,3}\|\tilde{q}_{N_0}^i(\xi-\eta, \eta)\|_{\mathcal{S}^\infty_{k,k_1,k_2}}\lesssim 2^{3\min\{k_1,k_2\}/2 + 2N_0 k}.
 \end{equation}

We define bilinear operators $\widetilde{Q}_1(U^1, U^2)$, $\widetilde{Q}_2(U^1, U^1), \widetilde{Q}_3(U^2, U^2)$ by the following symbols,
\begin{equation}\label{eqn423}
\widetilde{q}_i(\xi-\eta,\eta) = \frac{\tilde{q}^i_{N_0}(\xi-\eta, \eta)}{|\xi|^{2N_0}},\quad i \in\{1,2,3\}.
\end{equation}
From (\ref{equation8801}) and Lemma \ref{Snorm},  the following estimate holds:
\begin{equation}\label{equation26}
\sum_{i=1,2,3}\|\widetilde{q}_i(\xi-\eta, \eta)\|_{\mathcal{S}^\infty_{k,k_1,k_2}} \lesssim 2^{3\min\{k_1,k_2\}/2 }.
\end{equation}

Solving a similar system of equations as in (\ref{equation1100}) with $Q_{i}(\cdot, \cdot)$ replaced by $\widetilde{Q}_i(\cdot, \cdot)$, we can find bilinear operators $\widetilde{A}_1(U^1, U^1)$, $\widetilde{A}_2(U^2, U^2)$ and $\widetilde{B}(U^1, U^2)$ such that 
\begin{equation}\label{eqn427}
 2 \widetilde{A}_1( \d^\h U^2, U^1) - 2 \widetilde{A}_{2}(\d^\h U^1, U^2)+ \widetilde{Q}_1(U^1, U^2)- \d^\h \widetilde{B}(U^1, U^2)=0,
\end{equation}
\begin{equation}\label{eqn428}
 \widetilde{Q}_2(U^1, U^1)+ \d^{\h}\widetilde{A}_1(U^1, U^1)- \widetilde{B}(U^1, \d^\h U^1) =0,
 \end{equation}
\begin{equation}\label{eqn429}
\widetilde{Q}_{3}(U^2, U^2)+\d^{\h} \widetilde{A}_2(U^2, U^2)+ \widetilde{B}(\d^\h U^2, U^2)= 0.
\end{equation}
Very similar to the proof of Lemma \ref{auxilary1}, from (\ref{equation26}), we have  the following estimate, 
 \begin{equation}\label{equation8910}
\|\widetilde{a}_1(\xi-\eta, \eta)\|_{\mathcal{S}^\infty_{k,k_1,k_2}} + \|\widetilde{a}_2(\xi-\eta, \eta)\|_{\mathcal{S}^\infty_{k,k_1,k_2}} + \|\widetilde{b}(\xi-\eta, \eta)\|_{\mathcal{S}^\infty_{k,k_1,k_2}} \lesssim 2^{\min\{k_1, k_2\}}.
\end{equation}

So our strategy is to use normal form transformations $\widetilde{A}_1(U^1, U^1)$, $\widetilde{A}_2(U^2, U^2)$ and $\widetilde{B}(U^1, U^2)$ to cancel out the quadratic terms $\widetilde{Q}_i(\cdot, \cdot)$, $i\in\{1,2,3\}$, which will effectively cancel out the bulk cubic terms listed in (\ref{equation9989}) when doing energy estimate.   An advantage of utilizing symmetries first is that those normal form transformations do not lose derivative, see (\ref{equation8910}).

Define $W^1 = U^1 + \widetilde{A}_1(U^1, U^1) + \widetilde{A}_2(U^2, U^2), $ $W^2 = U^2 + \widetilde{B}(U^1, U^2)$. From the system of equations (\ref{mainequation}) satisfied by $U^1$ and $U^2$, we can first derive the system of equations satisfied by $W^1$ and $W^2$. Then, to see symmetries  inside the nonlinearities, we  substitute $(U^1, U^2)$  inside the quadratic terms and $C_1$ and $C_2$ by $(W^1, W^2)$. Recall (\ref{equation6}), (\ref{mainequation}) and  (\ref{equation29}). As a result, we can write the system of equations satisfied by $W^1$ and $W^2$ as follows, 
\begin{equation}\label{eqn120000}
\left\{\begin{array}{l}
\p_t W^1 - \d^\h W^2  =\mathfrak{Q}_1  +GC_1 +\widetilde{C}_{1}+ \widetilde{\mathcal{R}}_1  , \\
\p_t W^2 + \d^\h W^1  =\mathfrak{Q}_2+ GC_2 + \widetilde{C}_{2}+  \widetilde{\mathcal{R}}_2,\\
\end{array}\right.
\end{equation} 
where
\[
\mathfrak{Q}_1 =  {Q}_{1}(W^1,W^2)  - \widetilde{Q}_1(W^1, W^2),
\]

\[
\mathfrak{Q}_2 = Q_2(W^1, W^1) + Q_3(W^2, W^2) -\widetilde{Q}_2(W^1, W^1) -\widetilde{Q}_3(W^2,W^2),
\]
\begin{equation}\label{equation239358}
\widetilde{C}_{1} =  - T_{\Lambda_{\geq 2}[V]}\p_{x} W^{1} + T_{\Lambda_{\geq 2}[\alpha]}\d^{\h} W^{2},  \widetilde{C}_{2} =  - T_{\Lambda_{\geq 2}[V]} \p_x W^{2} - T_{\Lambda_{\geq  2}[\alpha]}\d^{\h} W^{1},
\end{equation}
\[
GC_1=  T_{V}\p_x(\widetilde{A}_1(U^1, U^1) + \widetilde{A}_2(U^2, U^2))- 2\widetilde{A}_1(T_V\p_x U^1, U^1) - 2\widetilde{A}_2(T_V \p_x U^2, U^2)\]
\begin{equation}\label{eqn450}
- T_{\alpha} \d^{\h}\widetilde{B}(U^1, U^2) + 2\widetilde{A}_1(T_{\alpha}\d^\h U^2, U^1) - 2 \widetilde{A}_2(T_{\alpha}\d^\h U^1, U^2), 
\end{equation}
\[
GC_2 =  T_{V}\p_x (\widetilde{B}(U^1, U^2))- \widetilde{B}(T_{V}\p_x U^1, U^2) -\widetilde{B}(U^1, T_{V}\p_x U^2)+T_{\alpha}\d^{\h}(\widetilde{A}_1(U^1, U^1)
\]
\begin{equation}
  + \widetilde{A}_2(U^2, U^2))+ \widetilde{B}(T_{\alpha}\d^\h U^2, U^2) - \widetilde{B}(U^1, T_{\alpha}\d^\h U^1).
\end{equation}

To improve presentation, we postpone the detailed formulas of $\widetilde{\mathcal{R}}_1$ and $\widetilde{\mathcal{R}}_2$  to Appendix \ref{remainderestimate}, as their formulas are tedious and not very easy to see structures inside without detailed explanation. It is good enough to know that they have sufficient decay rate and do not lose derivatives at the high frequency part or  at the low frequency part.
 
From the construction of bilinear operators $\widetilde{Q}_{i}(\cdot, \cdot), i\in \{1,2,3\}$, we can see that the quadratic terms $\mathfrak{Q}_1$ and $\mathfrak{Q}_2$ vanish when doing energy estimate. We will show that there are cancellations inside the good cubic terms $GC_1$ and $GC_2$ and there are symmetries inside cubic terms $\widetilde{C}_1$ and $\widetilde{C}_2$. Therefore, $W_1$ and $W_2$ are good substitution
variable of $U^1$ and $U^2$ at the top regularity. Moreover, as losing derivative is not a issue when we estimate the growth of $\dot{H}^p$ norm,  we can just cancel the quadratic terms directly by using the normal form transformation. Those intuitions motivate us to define the following  modified energy:
\[
E_{modi}(t) = \h \int \big[|\p_x^p U^1|^2 +|\p_x^p U^2|^2 \big] + \int \p_x^{p} U^1 \p_x^{p}\big[ {A}_1(U^1, U^1) + {A}_2(U^2, U^2) \big] 
 \]
\begin{equation}\label{eqn430}
 + \p_x^{p} U^2 \p_x^{p} \big[ {B}(U^1, U^2) \big]+\h \int \big[ |\p_x^{N_0} W^1|^2 + |\p_x^{N_0} W^2|^2\big].
\end{equation}

\begin{lemma}\label{sizeofmodified}
Under the bootstrap condition \textup{(\ref{assumption})}, the following estimate holds,
\begin{equation}\label{equation8916}
\sup_{t\in[0,T]}\big|E_{modi} (t )- \sum_{k=p, N_0} \h \int \big[|\p_x^k U^1|^2 +|\p_x^k U^2|^2 \big]\big| \lesssim \epsilon_0^2.
\end{equation}
\end{lemma}
\begin{proof}
From  (\ref{equation8905}) in Lemma \ref{auxilary1} and $L^2-L^\infty$-type estimate (\ref{bilinearestimate}) in Lemma \ref{boundness},  the following estimate holds after putting the input with larger frequency in $L^2$ and the smaller one in $L^\infty$,
\[
\Big|\int \p_x^{p} U^1 \p_x^{p}\big[ {A}_1(U^1, U^1) + {A}_2(U^2, U^2) \big]+ \p_x^{p}U^2 \p_x^{p} \big[ {B}(U^1, U^2) \big]\Big|
\]
\[
\lesssim \| (U^1, U^2)(t)\|_{H^{1+p,p}}^{2} \| ( U^1, U^2)(t)\|_{W^{0}}\lesssim (1+t)^{-1/2+2p_0}\epsilon_1^3.
\]
From (\ref{equation8910}) and $L^2-L^\infty$ estimate (\ref{bilinearestimate}) in Lemma \ref{boundness}, we have 
\[
\|\p_x^{N_0}( W^1 -U^1,  W^2-U^2)(t)\|_{L^2} \lesssim  \| (U^1, U^2)(t)\|_{H^{N_0,p}}\| ( U^1,  U^2)(t)\|_{W^{1}}\lesssim (1+t)^{-1/2+2p_0}\epsilon_1^2.
\]
Therefore, from  the definition of $E_{modi}(t)$ in (\ref{eqn430}) and above estimates, it's easy to see our desired estimate (\ref{equation8916}) holds. 
\end{proof}
After taking a derivative with respect to time for $E_{modi}(t)$, we have
\[
\frac{d}{d t}  E_{modi}(t) =\mathcal{J}_1 + \mathcal{J}_2 + \mathcal{J}_3, 
\]
where
\[
\mathcal{J}_1 = \int \p_x^p U^1 \p_x^p\big[ C_1 + \Lambda_{\geq 3}[\mathcal{R}_1]\big] + \p_x^p U^2 \p_x^p \big[ C_2 + \Lambda_{\geq 3}[\mathcal{R}_2]\big],\]
\[ \mathcal{J}_2 =   \int \p_x^p\Lambda_{\geq 2}[\p_t U^1]\p_x^p\big( A_1(U^1,U^1) + A_2(U^2, U^2)\big)+  \p_x^p\Lambda_{\geq 2}[\p_t U^2] \p_x^p \big(B(U^1, U^2)\big) \]
\[
+ 2\p_x^p U^1 \p_x^p \big(A_1(\Lambda_{\geq 2}[\p_t U^1], U^1)
 + 
 A_2(\Lambda_{\geq 2}[\p_t U^2], U^2) \big)\]
 \[+ \p_x^{p}U^2 \p_x^p\big(B(\Lambda_{\geq 2}[\p_t U^1], U^2) + B(U^1, \Lambda_{\geq 2}[\p_t U^2])\big),
\]
\begin{equation}\label{equation31}
\mathcal{J}_3 = \int  \p_x^{N_0}W^1 \p_x^{N_0}\big[\widetilde{C}_{1} + GC_1 +\widetilde{\mathcal{R}}_1\big] + \p_x^{N_0} W^2 \p_x^{N_0}\big[\widetilde{C}_{2} +GC_2 + \widetilde{\mathcal{R}}_2 \big]. 
\end{equation}

\begin{lemma}\label{energyestimate1}
Under the bootstrap  assumption \textup{(\ref{assumption})}, we have 
\begin{equation}\label{equation6200}
\sup_{t\in[0,T]} \sum_{i=1,2,3} (1+t)^{1-2p_0}|\mathcal{J}_i| \lesssim \epsilon_0^2.
\end{equation}
\end{lemma}
\begin{proof}
(\mbox{i}) To estimate $\mathcal{J}_1$ and $\mathcal{J}_2$,  as $p\in(0,1/4)$ is far away from $N_0$, we can estimate it straightforwardly by putting the input with higher frequency in $L^2$ and putting the input with lower frequency in $L^\infty$.

 From  (\ref{equation8905}) in Lemma \ref{auxilary1}, $L^2-L^\infty$-type estimate (\ref{bilinearestimate}) in Lemma \ref{boundness} and (\ref{eqn432}) in Lemma \ref{estimatesofremainder}, the following estimate holds,  
\[
|\mathcal{J}_1| + |\mathcal{J}_2| \lesssim \| \Lambda_{\geq 2}(\p_t U^1, \p_t U^2)\|_{{H}^{1+p,p}}\| (U^1, U^2)\|_{H^{1+p, p}}\| (U^1, U^2)\|_{W^{0}}+ \| (U^1, U^2)\|_{H^{1+p, p}}^2 
\]
\[
\times\| (U^1, U^2)\|_{W^{2}}^2 + \|(U^1, U^2)\|_{H^{N_0, p}} \|\Lambda_{\geq 3}[\mathcal{R}_1, \mathcal{R}_2]\|_{H^{N_0,p}}\lesssim (1+|t|)^{-1+2p_0} \epsilon_0^2.\]

(\mbox{ii}) Recall (\ref{equation239358}). We  can utilize symmetries to estimate $\widetilde{C}_{1} $ and $\widetilde{C}_{2}$ as follows,
\[
\Big| \int  \p_x^{N_0}W^1 \p_x^{N_0}\widetilde{C}_{1} + \p_x^{N_0} W^2 \p_x^{N_0}\widetilde{C}_{2}  \Big| \lesssim \sum_{i=1,2} \Big|  \p_x^{N_0}W^i   \p_x^{N_0} T_{\Lambda_{\geq 2}[V]}\p_{x} W^{i}\Big|  \]
\[
+ \Big| \p_x^{N_0}W^1 \p_x^{N_0}T_{\Lambda_{\geq 2}[\alpha]}\d^{\h} W^{2} - \p_x^{N_0} W^2 \p_x^{N_0}T_{\Lambda_{\geq  2}[\alpha]}\d^{\h} W^{1}\Big|
\]
\[ \lesssim \| (W^{1},W^{2})\|_{\dot{H}^{N_{0}}}^{2} [\|\Lambda_{\geq 2}[V]\|_{\widetilde{W^{1}}} + \| \Lambda_{\geq 2}[\alpha]\|_{\widetilde{W^{1/2}}}]\lesssim \| (U^1, U^2)\|_{H^{N_0, p}}^{2} \| (U^1, U^2)\|_{W^{3}}^2
\]
\begin{equation}
\lesssim (1+|t|)^{-1+2p_0}\epsilon_1^4\lesssim (1+|t|)^{-1+2p_0}\epsilon_0^2.
\end{equation}
 From (\ref{eqn427}) and (\ref{eqn450}), we can reduce $GC_1$ further  as follows,
\[
GC_1=  T_{V}\p_x(\widetilde{A}_1(U^1, U^1) + \widetilde{A}_2(U^2, U^2))- 2\widetilde{A}_1(T_V\p_x U^1, U^1) - 2\widetilde{A}_2(T_V \p_x U^2, U^2)\]
\[
 + T_{\alpha}[2 \widetilde{A}_1( \d^\h U^2, U^1) - 2 \widetilde{A}_{2}(\d^\h U^1, U^2)]
  - 2\widetilde{A}_1(T_{\alpha}\d^\h U^2, U^1)\]
  \[ + 2 \widetilde{A}_2(T_{\alpha}\d^\h U^1, U^2) +T_{\alpha}\widetilde{Q}_1(U^1, U^2).
\]
Recall that $\widetilde{Q}_1(\cdot, \cdot)$ and $\widetilde{A}_i(\cdot, \cdot)$, $i\in\{1,2\}$, do not lose derivation, see (\ref{equation26}) and (\ref{equation8910}). Now, it is easy to see there are cancellations inside $GC_1$. After utilizing symmetries on the Fourier side, the following equality holds, 
\[
\mathcal{F}(GC_1 )(\xi)= \int_{\R^2} \widehat{U^1}(\eta-\sigma) \widehat{U^1}(\sigma)\widehat{V}(\xi-\eta)e_1(\xi, \eta, \sigma) +
\]
\[
\widehat{U^2}(\eta-\sigma) \widehat{U^2}(\sigma)\widehat{V}(\xi-\eta) e_2(\xi,\eta, \sigma) +\widehat{U^1}(\eta-\sigma)\widehat{U^2}(\sigma) \widehat{\alpha}(\xi-\eta) e_3(\xi, \eta, \sigma),
\]
where, for $j\in\{1,2\},$
\[
e_j(\xi, \eta, \sigma)= 2i (\eta-\sigma)\big[ \theta(\xi-\eta, \eta)  \widetilde{a}_j(\eta-\sigma, \sigma)-  \widetilde{a}_j(\xi-\sigma, \sigma)  \theta(\xi-\eta, \eta-\sigma)\big], 
\]
\[
e_3(\xi, \eta, \sigma) = 2|\sigma|^{1/2}\big[\widetilde{a}_1( \sigma, \eta-\sigma)\theta(\xi-\eta, \eta)- \widetilde{a}_1(\xi-\eta+\sigma, \eta-\sigma) \theta(\xi-\eta,\sigma)\big]
\]
\[
- 2 |\eta-\sigma|^{1/2} \big[\widetilde{a}_2(\eta-\sigma, \sigma)\theta(\xi-\eta, \eta)-\widetilde{a}_2(\xi-\sigma, \sigma) \theta(\xi-\eta, \eta-\sigma)\big].
\]
From above explicit formulas and Lemma \ref{Snorm}, the following estimates hold, 
\[
\|e_1(\xi,\eta,\sigma)\|_{\mathcal{S}^\infty_{k,k_1,k_2,k_3}} + \|e_2(\xi,\eta,\sigma)\|_{\mathcal{S}^\infty_{k,k_1,k_2,k_3}}  \lesssim 2^{2\med\{k_1,k_2,k_3\}},
\]
\[
\|e_3(\xi, \eta, \sigma)\|_{\mathcal{S}^\infty_{k,k_1,k_2,k_3}} \lesssim  2^{3\med\{k_1,k_2,k_3\}/2},
\]
where $\med\{k_1, k_2, k_3\}$ denotes the medium number of $k_1, k_2,k_3$. Hence, from $L^2-L^\infty-L^\infty$ type estimate in Lemma \ref{boundness}, the following estimate holds,
\[
\| GC_1\|_{\dot{H}^{N_0}}\lesssim \| (U^1, U^2)\|_{H^{N_0, p}} \| \p_x(U^1, U^2)\|_{\widetilde{W^1}} [\| \alpha\|_{\widetilde{W^0}} +\| V\|_{\widetilde{W^0}}] \]
\begin{equation}\label{eqn460}
\lesssim \| (U^1, U^2)\|_{H^{N_0, p}} \| (U^1, U^2)\|_{W^{3}}^2\lesssim (1+|t|)^{-1+p_0}\epsilon_1^3\lesssim (1+|t|)^{-1+p_0}\epsilon_0^2. 
\end{equation}
The estimate of $GC_2$ is very similar,  the upper bound in the right hand side of (\ref{eqn460}) still good for  $GC_2$. We omit the details here.

  From (\ref{eqn432}) in Lemma \ref{estimatesofremainder},   it is easy to see the following estimate holds,
\[
 \big|\int  \p_x^{N_0}W^1 \p_x^{N_0} \widetilde{\mathcal{R}}_1  + \p_x^{N_0} W^2 \p_x^{N_0}  \widetilde{\mathcal{R}}_2  \big| \lesssim \|(W^1, W^2)\|_{H^{N_0,p}} \| (\widetilde{\mathcal{R}}_1, \widetilde{\mathcal{R}}_2)\|_{H^{N_0,p}} \lesssim (1+t)^{-1+2p_0}\epsilon_0^2.
\]
Recall (\ref{equation31}). Now, it is easy to see  our desired estimate (\ref{energyestimate1}) holds.

\end{proof}

\subsection{Energy estimate of $SU^1$ and $SU^2$.}

Note that $[ \p_t , S]= \p_t$ and $[\pm \d^\h,$ $ S]= \pm \d^\h $. On the one hand, from the system (\ref{equation6}), we can derive the system of equations satisfied by $SU^1$ and $SU^2$ with highlighted
 quasilinear structure 
 as follows, 
\begin{equation}\label{equationscaling}
\left\{\begin{array}{l}
\p_t SU^1- \d^{1/2}SU^2=  T_{\alpha} \d^{1/2} SU^2 - T_{V}\p_x SU^1 +\mathcal{R}_1^S,\\
\p_t SU^2 + \d^{1/2} SU^1=-T_{\alpha} \d^{1/2} SU^1 -T_{V}\p_x SU^2  + \mathcal{R}_2^S,
\end{array}\right.
\end{equation}
where $\mathcal{R}_1^S$ and $\mathcal{R}_2^S$ are  quadratic and higher good remainder terms that do  not lose derivatives in $SU^1$ and $SU^2$. More precisely, we have,
\begin{equation}\label{equation968}
\mathcal{R}_1^S:=(S+I)\mathcal{R}_1 + S(T_{\alpha}\d^{1/2} U^2- T_V \p_x U^1) - T_{\alpha}\d^{1/2} SU^2+ T_V \p_x SU^1,
\end{equation}
\begin{equation}\label{equation969}
\mathcal{R}_2^S:=(S+I)\mathcal{R}_2 + S(-T_{\alpha}\d^{1/2} U^1- T_V \p_x U^2) + T_{\alpha}\d^{1/2} SU^1+ T_V \p_x SU^2.
\end{equation}
On the other hand, from the system  (\ref{mainequation}), we can rewrite the system (\ref{equationscaling}) as follows to highlight the structures inside the quadratic terms,
\begin{equation}\label{systemofscaling}
\left\{ \begin{array}{l}
\p_{t} S U^{1} - \d^{\h} S U^{2} =  \mathfrak{Q}_1^S(SU^1, SU^2)+ \mathfrak{Q}_2^S +\mathcal{Q}_1 +\widetilde{C}_1+ {\mathfrak{R}}_1, \\
\pt S U^{2} + \d^{\h} S U^{1} = \mathfrak{Q}^S_3(SU^1, SU^2)+ \mathfrak{Q}_4^S +\mathcal{Q}_2 +\widetilde{C}_2 + {\mathfrak{R}}_2,  \\
\end{array}\right.
\end{equation}
where $\mathfrak{Q}_1^S(\cdot, \cdot)$ and $\mathfrak{Q}_3^S(\cdot, \cdot)$ are quadratic terms that have quasilinear structures inside, $\mathfrak{Q}_2^S$ and $\mathfrak{Q}_3^S$ are quadratic terms that have problematic low frequency part, $ \mathcal{Q}_1$ and $ \mathcal{Q}_2$ are commutator quadratic terms that do  not depend on the scaling vector field, $\widetilde{C}_1$ and $\widetilde{C}_2$ are cubic and higher order terms that at most lose one derivatives and $\mathfrak{R}_1$ and $\mathfrak{R}_2$ are good cubic and higher order terms that do not lose derivatives. Their detailed formulas are given as follows, 
\[ 
\mathfrak{Q}_1^S(SU^1,SU^2) = Q_{1}(SU^1, U^2)+Q_1(U^1, SU^2)-T_{\p_x \d^{-1/2}SU^2}\p_x U^1,\]
\[\mathfrak{Q}_3^S(SU^1,SU^2) = Q_2(SU^1, U^1)
+ Q_2(U^1, SU^1)\]
\[+ Q_3(SU^2, U^2)+ Q_3(U^2, SU^2) +  T_{\px \d^{-\h} SU^{2} } \px U^{2},\]
\begin{equation}\label{equation50}
  \mathfrak{Q}_2^S=T_{\p_x \d^{-1/2}SU^2}\p_x U^1,\,\,\mathfrak{Q}_4^S= -  T_{\px \d^{-\h} SU^{2} } \px U^{2}, 
  \end{equation}
\begin{equation}\label{eqn564}
 \mathcal{Q}_1= Q_{1}(U^1, U^2)+ SQ_{1}(U^1, U^2)-  Q_{1}(SU^1, U^2)- Q_{1}(U^1, SU^2) 
\end{equation}
\[
 \mathcal{Q}_2 = Q_2(U^1, U^1) + Q_3(U^2, U^2)
 + SQ_2(U^1,U^1)- Q_2(SU^1, U^1) - Q_2(U^1, SU^2)\]
\begin{equation}\label{eqn565}
+ SQ_2(U^1,U^1)- Q_2(SU^1, U^1) - Q_2(U^1, SU^2),
\end{equation}
\[
\widetilde{C}_1 = T_{\Lambda_{\geq 2}[\alpha]} \d^{\h} SU^2 - T_{\Lambda_{\geq 2}[V]} \p_x SU^1,
\widetilde{C}_2 =  T_{\Lambda_{\geq 2}[\alpha]} \d^{\h} SU^2 - T_{\Lambda_{\geq 2}[V]} \p_x SU^1,
\]
\[
\mathfrak{R}_1 = \Lambda_{\geq 3}[\mathcal{R}_1^S],\quad \mathfrak{R}_2= \Lambda_{\geq 3}[\mathcal{R}_2^S].\]

To better see which parts of  $\mathfrak{Q}_1^S(\cdot,\cdot)$ and $\mathfrak{Q}_{1}^S(\cdot, \cdot)$ lose derivatives in $SU^1$ and $SU^2$, especially in later energy estimate part, we define the following auxiliary bilinear forms,
\begin{equation}\label{equation9653}
\widetilde{\mathfrak{Q}}_1^S(SU^1, SU^2):= \mathfrak{Q}_1^S(SU^1,SU^2) - T_{\Lambda_1[\alpha]}\d^{1/2}SU^2 + T_{\Lambda_1[V]} \p_x SU^1,
\end{equation}
\begin{equation}\label{equation9652}
\widetilde{\mathfrak{Q}}_3^S(SU^1, SU^2):= \mathfrak{Q}_3^S(SU^1,SU^2) +T_{\Lambda_1[\alpha]}\d^{1/2}SU^1 + T_{\Lambda_1[V]} \p_x SU^2.
\end{equation}
From (\ref{equationscaling}), we can see that $\widetilde{\mathfrak{Q}}_1^S(SU^1, SU^2)$ and $\widetilde{\mathfrak{Q}}_3^S(SU^1, SU^2)$ do not lose derivatives in $SU^1$ and $SU^2$.
\subsubsection{Handling the bulk quadratic terms $\mathfrak{Q}_1^S$ and $\mathfrak{Q}_3^S$}
Similar to what we did in the subsection \ref{usualenergyestimate}, we will also utilize symmetries for the high frequency part.  For $\tau=p, N_1$, we can utilize symmetries of the same type of inputs and do changing of variables on the Fourier side. As a result,  the following equality holds, 
\[
\mathfrak{Re} \big[ \int \overline{\p_x^\tau S U^1} \p_x^\tau  \mathfrak{Q}_1^S  + \overline{\p_x^\tau SU^2}\p_x^\tau \mathfrak{Q}_3^S  \big] = \mathfrak{Re} \big[ \int \overline{\widehat{SU^1}(\xi)}\hat{q}_{\tau}^1(\xi-\eta, \eta) \widehat{SU^1}(\xi-\eta) \widehat{U^2}(\eta)\]
\begin{equation}\label{equation39}
+ \overline{\widehat{SU^1}(\xi)} \hat{q}_\tau^2(\xi-\eta, \eta) \widehat{SU^2}(\xi-\eta)\widehat{U^1}(\eta) + 
\overline{\widehat{SU^2}(\xi)} \hat{q}_\tau^3(\xi-\eta, \eta) \widehat{SU^2}(\xi-\eta)\widehat{U^2}(\eta)\big],
\end{equation}
where 
\[
\hat{q}_{\tau}^1(\xi-\eta, \eta) = |\xi|^{2\tau} \big(q_{1}^1(\xi-\eta, \eta)/2 + q_1^2(\xi-\eta, \eta) +q_1^3(\xi-\eta, \eta)\big)+q_1^1(-\xi, \eta)|\xi-\eta|^{2\tau}/2,
\]
\[
\hat{q}_\tau^2(\xi-\eta,\eta)= |\xi|^{2\tau}\big(q_1(\eta, \xi-\eta)+\eta(\xi-\eta)|\eta|^{-1/2}\theta(\xi-\eta, \eta)\big) + |\xi-\eta|^{2\tau} \big(q_2(\eta, -\xi) + q_2(-\xi,\eta) \big),
\]
\[
\hat{q}_\tau^3(\xi-\eta, \eta) = |\xi|^{2\tau}\big(q_3(\xi-\eta, \eta)+ q_3(\eta, \xi-\eta)  \big)/2+|\xi-\eta|^{2\tau} \big( q_3(\eta, -\xi)+q_3(-\xi, \eta)\big)/2
\]
\[
-(\xi-\eta)\eta|\xi-\eta|^{-1/2}|\eta|^{-1/2}|\xi|^{1/2+2\tau} \theta(\xi-\eta, \eta).
\]

Note that $q_1(\xi-\eta, \eta)\theta(\eta, \xi-\eta)=q_1^1(\xi-\eta, \eta)$ and $q_1(\xi-\eta, \eta)\theta( \xi-\eta,\eta)=q_1^2(\xi-\eta, \eta)$. From (\ref{equation36}) in Lemma \ref{sizeofsymbol}, we can see cancellations also happen when $|\eta|\ll |\xi-\eta|$. Moreover, when $|\xi-\eta|\ll |\eta|$, \emph{we can gain at least one degree of smallness from the symbols} as we put those exception terms into $\mathfrak{Q}_2^S$ and  $\mathfrak{Q}_4^S$.  For the Hight $\times$ High type interaction, all terms except $q_1^3(\cdot, \cdot)$ vanish, see (\ref{eqn50001}), (\ref{eqn50002}), and (\ref{eqn50003}).
As a result, the following estimate holds,
\begin{equation}\label{equation8991}
\sum_{i=1,2,3}\|\hat{q}_\tau^3(\xi-\eta, \eta) \|_{\mathcal{S}^{\infty}_{k,k_1,k_2}} \lesssim \left\{ \begin{array}{ll}
2^{3k_2/2+2\tau k} & \textup{if $k_2\leq k_1-5$} \\
2^{(2\tau+1)k+k_1/2} & \textup{if $|k_1-k_2|\leq 5$}\\
2^{k_1 + (2\tau+1/2 )k}& \textup{if $k_1\leq k_2-5$}. \\
\end{array}\right.
\end{equation}

For $\tau=p, N_1$, we define  $\widetilde{Q}_{1,\tau}(SU^1, U^2)$, $\widetilde{Q}_{2,\tau}(SU^2, U^1)$, $\widetilde{Q}_{3,\tau}(SU^1, U^1)$ and $\widetilde{Q}_{4,\tau}(SU^2, U^2)$  by the following symbols,
\[
\widetilde{q}_{1,\tau}(\xi-\eta, \eta)= \frac{\hat{q}_{\tau}^1(\xi-\eta, \eta)}{|\xi|^{2\tau}}, \quad \widetilde{q}_{2,\tau}(\xi-\eta, \eta)= \frac{\hat{q}_{\tau}^2(\xi -\eta,\eta)}{2|\xi|^{2\tau}},\]
\[
 \widetilde{q}_{3,\tau}(\xi-\eta, \eta)= \frac{\hat{q}_{\tau}^2(\xi, -\eta)}{2|\xi|^{2\tau}}, \quad \widetilde{q}_{4,\tau}(\xi-\eta, \eta) = \frac{\hat{q}_{\tau}^3(\xi-\eta, \eta)}{|\xi|^{2\tau}}.
\]
We also define the following substitution variables of $SU^1$ and $SU^2$,
\begin{equation}\label{equation41}
W_{1,\tau} := S U^{1} + \widetilde{C}_{1,\tau}(SU^1, U^1) + \widetilde{C}_{2,\tau}(SU^2, U^2),
\end{equation}
\begin{equation}\label{equation42}
 W_{2,\tau} := SU^2 + \widetilde{D}_{1,\tau}(SU^1, U^2)  + \widetilde{D}_{2,\tau}(SU^2, U^1).
 \end{equation}

The goal is to use bilinear operators $\widetilde{C}_{i, \tau}(\cdot, \cdot)$ and $\widetilde{D}_{i, \tau}(\cdot, \cdot)$, $i\in\{1,2\}$ ,  to cancel out $\widetilde{Q}_{j, \tau}(\cdot, \cdot)$, $j\in\{1,2,3,4\}$. Effectively speaking, it cancel out (\ref{equation39}) in the energy estimate. To this end, it would be sufficient if symbols $c_\tau^i (\cdot, \cdot)$ and $d_\tau^i(\cdot, \cdot)$  of bilinear operators $\widetilde{C}_{i, \tau}(\cdot, \cdot)$ and $\widetilde{D}_{i, \tau}(\cdot, \cdot)$ solve the following system of equations,
  \begin{equation}\label{equation79890}
\left\{\begin{array}{l}
|\eta|^{\h} c_\tau^1(\xi-\eta, \eta) - |\xi-\eta|^{\h} c_\tau^2(\xi-\eta, \eta) - |\xi|^{\h} d_\tau^1(\xi-\eta, \eta)+ \widetilde{q}_{1,\tau}(\xi-\eta, \eta) =0\\
|\xi-\eta|^{\h} c_\tau^1(\xi-\eta, \eta) - |\eta|^{\h} c_\tau^2(\xi-\eta, \eta) - |\xi|^{\h} d_\tau^2(\xi-\eta, \eta)+ \widetilde{q}_{2,\tau}(\xi-\eta, \eta) =0\\
|\xi|^{\h}c_\tau^1(\xi-\eta, \eta) - |\eta|^{\h} d_\tau^1(\xi-\eta,\eta) - |\xi-\eta|^{\h} d_\tau^2(\xi-\eta, \eta) + \widetilde{q}_{3,\tau}(\xi-\eta, \eta)=0\\
|\xi|^{\h}c_\tau^2(\xi-\eta, \eta) + |\xi-\eta|^{\h} d_\tau^1( \xi- \eta, \eta) + |\eta|^{\h} d_\tau^2(\xi-\eta, \eta) +\widetilde{q}_{4,\tau}(\xi-\eta, \eta) =0.
\end{array}\right.
\end{equation}
It is not difficult to solve above system of equations and derive the following,
\begin{equation}\label{equation68980}
d_\tau^1(\xi-\eta, \eta) = \frac{ (|\xi-\eta|+|\eta|-|\xi|) F_1(\xi-\eta, \eta) - 2 F_2(\xi-\eta, \eta)|\xi-\eta|^{\h}|\eta|^{\h}}{- (|\xi-\eta| + |\eta| - |\xi|)^{2} + 4 |\xi-\eta||\eta|},
\end{equation}
\begin{equation}\label{equation68981}
d_\tau^2(\xi-\eta, \eta) = \frac{ (|\xi-\eta|+|\eta|-|\xi|) F_2(\xi-\eta, \eta) - 2 F_1(\xi-\eta, \eta)|\xi-\eta|^{\h}|\eta|^{\h}}{- (|\xi-\eta| + |\eta| - |\xi|)^{2} + 4 |\xi-\eta||\eta|},
\end{equation}
\begin{equation}\label{equation68982}
c_\tau^1(\xi-\eta,\eta) = \frac{|\eta|^{\h} d_\tau^1(\xi-\eta, \eta) + |\xi-\eta|^{\h} d_\tau^2(\xi-\eta, \eta) - \widetilde{q}_{3,\tau}(\xi-\eta, \eta)}{|\xi|^{\h}},
\end{equation}
\begin{equation}\label{equation68983}
c_\tau^2(\xi-\eta, \eta)= \frac{- |\xi-\eta|^{\h} d_\tau^1(\xi-\eta, \eta)- |\eta|^{\h} d_\tau^2(\xi-\eta, \eta) -  \widetilde{q}_{4,\tau}(\xi-\eta, \eta)}{|\xi|^{\h}},
\end{equation}
where  $F_1(\xi-\eta,\eta)$ and $F_2(\xi-\eta, \eta)$ are defined as follows,
\[
F_1(\xi-\eta, \eta)= |\xi|^{\h} \widetilde{q}_{1,\tau}(\xi-\eta, \eta) -  |\eta|^{\h}  \widetilde{q}_{3,\tau}(\xi-\eta, \eta) + |\xi-\eta|^{\h} \widetilde{q}_{4,\tau}(\xi-\eta, \eta),
\]
\[
F_2(\xi-\eta, \eta)= |\xi|^{\h} \widetilde{q}_{2,\tau}(\xi-\eta, \eta) -  |\xi-\eta|^{\h}  \widetilde{q}_{3,\tau}(\xi-\eta, \eta) + |\eta|^{\h}  \widetilde{q}_{4,\tau}(\xi-\eta, \eta).
\]
\begin{lemma}\label{sizeinfolemma}
The following estimate holds for $i\in\{1,2\}$,
\begin{equation}\label{equation477}
\|c_\tau^i(\xi-\eta, \eta) \|_{\mathcal{S}^{\infty}_{k,k_1,k_2}} + \|d_\tau^i(\xi-\eta, \eta) \|_{\mathcal{S}^{\infty}_{k,k_1,k_2}} \lesssim \left\{ \begin{array}{ll}
2^{k_2} & \textup{if $k_2\leq k_1+5$} \\
2^{k_1/2 + k_2/2}& \textup{if $k_1\leq k_2-5$}. \\
\end{array}\right.
\end{equation}
\end{lemma}
\begin{proof}
Recall (\ref{equation68980}), (\ref{equation68981}), (\ref{equation68982}) and (\ref{equation68980}). From (\ref{equation8991}), Lemma \ref{Snorm}, and (\ref{productofsymbol}) in Lemma \ref{boundness}, our desired estimate (\ref{equation477}) follows straightforwardly.
\end{proof}

Recall (\ref{equationscaling}) and (\ref{systemofscaling}). From (\ref{equation41}) and (\ref{equation42}), after substituting $SU^i$ by $W_{i, \tau}, i\in\{1,2\}$, we can derive the equations satisfied by $W_{1,\tau}$ and$W_{2,\tau}$ with good structures as follows,
\[
\p_t W_{1,\tau} - \d^\h W_{2,\tau}= \mathfrak{Q}_1^S(W_{1,\tau}, U^2) + \mathfrak{Q}_1^S( W_{2,\tau},U^1)- \widetilde{Q}_{1,\tau}(W_{1,\tau}, U^2)- \widetilde{Q}_{2,\tau}(W_{2,\tau}, U^1)   \]
\begin{equation}\label{eqn53000}
-T_{\Lambda_{\geq2 }[V]} \p_{x} W_{1,\tau} + T_{\Lambda_{\geq2 }[\alpha]}\d^{\h} W_{2,\tau} +\mathfrak{Q}_2^S+ \mathcal{Q}_1+ \widetilde{GC}_1 + \widetilde{\mathfrak{R}}_1,  
\end{equation}
\[
\p_t W_{2,\tau} + \d^\h W_{1,\tau}=  \mathfrak{Q}_{3}^S(W_{1,\tau}, U^1) + \mathfrak{Q}_{3}^S(W_{2, \tau}, U^2)- \widetilde{Q}_{3,\tau}(W_{1,\tau}, U^1) - \widetilde{Q}_{4,\tau}(W_{2,\tau}, U^2)  \]
\begin{equation}\label{eqn53001}
-T_{\Lambda_{\geq2 }[V]} \p_{x} W_{2, \tau} - T_{\Lambda_{\geq2 }[\alpha]}\d^{\h} W_{1,\tau}+ \mathfrak{Q}_4^S + \mathcal{Q}_2+ \widetilde{GC}_2 +\widetilde{\mathfrak{R}}_2,
\end{equation}
where
\[
\widetilde{GC}_1 = T_{V}\p_x(\widetilde{C}_{1,\tau}(SU^1, U^1) + \widetilde{C}_{2,\tau}(SU^2, U^2))- T_{\alpha}\d^{\h}[\widetilde{D}_{1,\tau}(SU^1, U^2)\]
\[+ \widetilde{D}_{2,\tau}(SU^2, U^1)] 
- \widetilde{C}_{1,\tau}(T_{V}\p_x S U^1, U^1)+ \widetilde{C}_{1,\tau}(T_{\alpha}\d^{\h} SU^2, U^1)\]
\[- \widetilde{C}_{2,\tau}(T_{V} \p_x SU^2, U^2)
-\widetilde{C}_{2,\tau}(T_{\alpha}\d^{\h} SU^1, U^2 ),
\]
\[
\widetilde{GC}_2= T_{V}\p_x [\widetilde{D}_{1,\tau}(SU^1, U^2) + \widetilde{D}_{2,\tau}(SU^2, U^1)]+ T_{\alpha}\d^{\h}[\widetilde{C}_{1,\tau}(SU^1, U^1)\]
\[+ \widetilde{C}_{2,\tau}(SU^2, U^2)]
 - \widetilde{D}_{1,\tau}(T_{V} \p_x SU^1, U^2 )+ \widetilde{D}_{1,\tau}(T_{\alpha} \d^\h SU^2, U^2)\]
 \[- \widetilde{D}_{2,\tau}(T_V \p_x SU^2, U^1)
 -\widetilde{D}_{2,\tau}(T_{\alpha}\d^\h SU^1, U^1).
\]
To improve presentation, we postpone the formulas of 
 $\widetilde{\mathfrak{R}}_1$ and $\widetilde{\mathfrak{R}}_2$ to the Appendix \ref{remainderestimate}. It's enough to see that they are good cubic and higher remainder terms in the sense that they do not lose derivatives at the high frequency part or the low frequency part of    $SU^1$ and $ SU^2$.   Hence the $H^{N_1, p}$ norm of those remainder terms can be estimated straightforwardly in the energy estimate.

Similar to what we did in the energy estimate of $U^1$ and $U^2$, we define the high frequency part of the modified energy for $SU^1$ and $SU^2$ as follows:
\[
E_{modi}^{S, high}(t) := \sum_{\tau =p, N_1} \int \frac{1}{2} \Big[ |\p_x^\tau W_{1,\tau}|^2 + \int \p_x^\tau W_{1,\tau} \p_x^\tau[E_1(U^1, U^1) + E_2(U^2, U^2)]
\]
\begin{equation}\label{eqn550}
 + \p_x^\tau W_{2,k} \p_x^\tau F(U^1, U^2),
\end{equation}
where  bilinear forms $E_1(U^1, U^1)$, $E_2(U^2, U^2)$ and $F(U^1, U^2)$  are the normal form transformations, which aim to cancel out  the commutator terms $\mathcal{Q}_1$ and $\mathcal{Q}_2$ in (\ref{equation50}) and (\ref{eqn594}).

By solving a similar system of equations as in  (\ref{equation1100}) with $Q_{1}(U^1, U^2)$ replaced by $\mathcal{Q}_1$  and $Q_{2}(U^1, U^1)+ Q_{3}(U^2, U^2)$ replaced by $\mathcal{Q}_2$, we can  explicitly solve symbols $e_1(\xi-\eta,\eta)$, $e_2(\xi-\eta,\eta)$, and $f(\xi-\eta,\eta)$ of bilinear operators $E_1(\cdot, \cdot)$, $E_2(\cdot, \cdot)$ and  $F(\cdot, \cdot)$.   Their precise formulas are not so important, hence we omit them here. Very similar to the proof of estimate (\ref{equation8905}) in Lemma \ref{auxilary1}, the following estimate holds,  
\begin{equation}\label{eqn530}
\|e_1(\xi-\eta, \eta)\|_{\mathcal{S}^\infty_{k,k_1,k_2}} +\|e_2(\xi-\eta, \eta)\|_{\mathcal{S}^\infty_{k,k_1,k_2}} + \|f(\xi-\eta, \eta)\|_{\mathcal{S}^\infty_{k,k_1,k_2}} \lesssim 2^{\max\{k_1,k_2\}}.
\end{equation}

\begin{lemma}\label{scalingauxiliary1}
Under the bootstrap smallness assumption \textup{(\ref{assumption})}, we have
\begin{equation}\label{eqn540}
\sup_{t\in[0,T]} \big| E_{modi}^{S, high}(t) -  \sum_{k=p, N_0} \h \int \big[|\p_x^k U^1|^2 +|\p_x^k U^2|^2 \big]\big| \lesssim \epsilon_0^2.   
\end{equation}
\end{lemma}

\begin{proof}
From  (\ref{equation477}) in Lemma \ref{sizeinfolemma} and the $L^2-L^\infty$-type bilinear estimate (\ref{bilinearestimate}) in Lemma \ref{boundness}, the following estimate holds,
\[
\sum_{\tau=p, N_1}
\|W_{1,\tau}(t)- SU^1(t)\|_{\dot{H}^\tau} + \| V_{2,\tau}(t)- SU^2(t)\|_{\dot{H}^\tau} 
\]
\begin{equation}\label{eqn534}
\lesssim \| (SU^1, SU^2)(t)\|_{H^{N_1, p}} \| (U^1, U^2)(t)\|_{W^1}\lesssim (1+|t|)^{-1/2+p_0} \epsilon_1^2.
\end{equation}
From  (\ref{eqn530}), (\ref{eqn534}), and the $L^2-L^\infty$-type bilinear estimate (\ref{bilinearestimate}) in Lemma \ref{boundness}, our desired estimate holds as follows,
\[
\Big|E_{modi}^{S, high}(t) -\sum_{k=p, N_1} \int |\p_x^k SU^1|^2 + \| \p_x^k SU^2|^2 \Big|\lesssim  (1+|t|)^{-1/2+p_0} \epsilon_1^3
\]
\[+\| (SU^1, SU^2)(t)\|_{H^{N_1, p}} \|(U^1, U^2)\|_{H^{N_1+1,p}} (U^1, U^2)(t)\|_{W^1}\lesssim \epsilon_0^2.
\]
\end{proof}

From direct computations, we have the following estimate,
\begin{equation}\label{eqn1232}
 \Big| \frac{d }{d t} E_{modi}^{S, high} (t) +\sum_{\tau =p, N_1} \int \p_x^\tau SU^1 \p_x^\tau \mathfrak{Q}_2^S+ \p_x^\tau SU^2 \p_x^\tau  \mathfrak{Q}_4^S \Big|\lesssim \sum_{i=1,2,3}\big| \mathcal{J}_{i}^S\big|,
\end{equation}
where
\[
\mathcal{J}_{1}^S = \sum_{\tau=p, N_1} \int \p_x^k W_{1,\tau} \p_x^\tau [\widetilde{GC}_1 + \widetilde{\mathfrak{R}_1}] + \p_x^\tau W_{2,\tau}\p_x^\tau[\widetilde{GC}_2 + \widetilde{\mathfrak{R}_2}] -\p_x^\tau \big( \widetilde{C}_{1,\tau}(SU^1, U^1),\]
\[
 + \widetilde{C}_{2,\tau}(SU^2, U^2)\big)\p_x^\tau \mathfrak{Q}_2^S + \p_x^\tau\big( \widetilde{D}_{1,\tau}(SU^1, U^2)  + \widetilde{D}_{2,\tau}(SU^2, U^1)\big)\p_x^\tau  \mathfrak{Q}_4^S
\]
\[
\mathcal{J}_2^S = \sum_{\tau=p, N_1} \int \p_x^\tau \Lambda_{\geq 2}[\p_t W_{1,\tau}]\p_x^\tau [E_1(U^1, U^1)+ E_{2}(U^2, U^2)] + \p_x^\tau \Lambda_{\geq 2}[\p_t W_{2,\tau}] \times\]
\[ \p_x^\tau  F(U^1, U^2)+
2\p_x^\tau W_{1,k} \p_x^\tau [E_1(\Lambda_{\geq 2}[\p_t U^1], U^1) + E_{2}(\Lambda_{\geq 2}[\p_t U^2], U^2)] \]
\[+ \p_x^\tau W_{2,k} \p_x^\tau[F(\Lambda_{\geq 2}[\p_t U^1], U^2) + F(U^1, \Lambda_{\geq 2}[\p_t U^2])],
\]
\[
\mathcal{J}_3^S = \sum_{\tau=p, N_1} \int \p_x^\tau W_{1,\tau} \p_x^\tau[-T_{\Lambda_{\geq2 }[V]} \p_{x} W_{1, \tau} + T_{\Lambda_{\geq2 }[\alpha]}\d^{\h} W_{2,\tau} ]\]
\[- \p_x^\tau W_{2,\tau} \p_x^\tau [T_{\Lambda_{\geq2 }[V]} \p_{x} W_{2, \tau} + T_{\Lambda_{\geq2 }[\alpha]}\d^{\h} W_{1,\tau}  ].
\]
We remark that the quartic terms that depend on $\mathfrak{D}_2^S$ and $\mathfrak{D}_4^S$ inside $\mathcal{J}_1^S$ are resulted from replacing $SU^1$ and $SU^2$ by $W_{1,\tau}$ and $W_{2, \tau}$ respectively.
\begin{lemma}\label{scalingauxiliary2}
Under the bootstrap condition \textup{(\ref{assumption})}, we have the following estimate 
\begin{equation}\label{eqn570}
\sum_{i=1} |\mathcal{J}_i^S|\lesssim (1+t)^{-1+2p_0}\epsilon_0^2.
\end{equation}
\end{lemma}
\begin{proof}

We first  estimate $\mathcal{J}_1^S$ and $\mathcal{J}_3^S$. From (\ref{vectorfieldestimate2}) in Lemma \ref{finalestimate}, we have the estimate of $H^{N_1, p}$ norm of $\widetilde{\mathfrak{R}}_1$ and  $\widetilde{\mathfrak{R}}_2$.  Similar to the estimate of $GC_i$ we did in the proof of Lemma \ref{energyestimate1}, one can see cancellation happens once write $\widetilde{GC}_i$ on the Fourier side after utilizing the equalities satisfied by normal form transformations in (\ref{equation79890}). Also after utilizing symmetries, we can see cancellations inside $\mathcal{J}_3^S$. As a result, the following estimate holds,
\[  
|\mathcal{J}_1^S|+|\mathcal{J}_3^S| \lesssim \big[ \| (SU^1, SU^2)\|_{H^{N_1, p}} + \| (U^1, U^2)\|_{H^{N_0, p}}\big]^2 \| (U^1,U^2)\|_{{W^{3}}}^2 +(1+t)^{-1+2p_0}\epsilon_0^2 \] 
\[\lesssim (1+t)^{-1+2p_0}\epsilon_0^2.
\]

Next, we proceed to  estimate $\mathcal{J}_2$. Note that the inputs inside bilinear operators $E_1(\cdot, \cdot)$, $E_2(\cdot, \cdot)$ and $F(\cdot, \cdot)$ only depend on $(U^1, U^2)$. Hence we can always put the input with higher frequency in $L^2$ to avoid losing $p$ derivatives of smallness. As a result, we have
\[  
|\mathcal{J}_2|\lesssim \big[ \| (SU^1, SU^2)\|_{H^{N_1, p}} + \| (U^1, U^2)\|_{H^{N_0, p}}\big]^2 \| (U^1,U^2)\|_{{W^{3}}}^2 \lesssim (1+t)^{-1+2p_0}\epsilon_0^2.
\]
To sum up, we can see the desired estimate (\ref{eqn570}) holds.
\end{proof}
\subsubsection{Handling the bulk quadratic terms $\mathfrak{Q}_2^S$ and $\mathfrak{Q}_4^S$}

It remains to  deal with the quadratic terms $\mathfrak{Q}_2^S$ and $\mathfrak{Q}_4^S$. As losing a derivative is not a issue,  it is unnecessary to utilize symmetries to see cancellations any more.  We choose to work in the complex variables setting, which is more convenient in the Fourier transform based method. We define 
\[
U = U^1 +  i U^2,\quad SU = SU^1 + i SU^2, \quad c_{\mu}= \mu/(2i),\,\,\mu\in\{+, -\}.
\]
From (\ref{mainequation}), we can write the equations satisfied by $U$ and $SU$ as follows
\[
\p_t U +i \d^{1/2} U = \sum_{\mu, \nu\in\{+,-\}} Q_{\mu, \nu} (U_{\mu}, U_{\nu}) + \Lambda_{\geq 3}[\p_t U],
\]
\[ 
\p_{t} SU + i \d^{1/2} SU =\sum_{\mu, \nu\in\{+,-\}}  Q_{\mu, \nu}^1 ((SU)_{\mu}, U_{\nu}) + Q_{\mu, \nu}^2(U_{\mu}, U_{\nu}) + \Lambda_{\geq 3}[\p_t SU],
\]
where $f_{+}:= f$ and $f_{-}:= \overline{f}$. 
From  (\ref{equation20}) in Lemma \ref{sizeofsymbol}, it's easy to see the following estimate holds,
\[
 \sum_{\mu, \nu\in\{+,-\}} \|q_{\mu, \nu}(\xi-\eta, \eta)\|_{\mathcal{S}^\infty_{k,k_1,k_2}} + \|q_{\mu, \nu}^1(\xi-\eta, \eta)\|_{\mathcal{S}^\infty_{k,k_1,k_2}} + \|q_{\mu, \nu}^2(\xi-\eta,\eta)\|_{\mathcal{S}^\infty_{k,k_1,k_2}} \]
 \begin{equation}\label{equation92}
 \lesssim 2^{\min\{k_1,k_2\}/2+\max\{k_1,k_2\}}.
\end{equation}
Recall (\ref{equation50}). We can write those problematic cubic terms  on the Fourier side and in terms of $SU$ and $U$ as follows,
\[
\sum_{\tau=p, N_1} \mathfrak{Re} \big[ \int \overline{\p_x^\tau S U^1} \p_x^\tau  \mathfrak{Q}_2^S  + \overline{\p_x^\tau SU^2}\p_x^\tau \mathfrak{Q}_4^S  \big] \]
\begin{equation}\label{equation67}
=\sum_{\tau=p, N_1} \sum_{\mu, \nu, \kappa \in\{+,-\}} \mathfrak{Re} \big[\int \overline{\widehat{(SU)_{\mu}}}(\xi)\widehat{(SU)_{\nu}}(\xi-\eta) \widehat{U_{\kappa}}(\eta) p_{\mu, \nu}^{\kappa,\tau}(\xi-\eta, \eta) d \eta d \xi, 
\end{equation}
where 
\begin{equation}\label{equation69}
p_{\mu, \nu}^{\kappa,\tau}(\xi-\eta, \eta) = -c_{\nu}(\frac{1}{4} + c_{-\mu} c_{\kappa}) (\xi-\eta)\eta |\xi-\eta|^{-1/2} |\xi|^{2\tau}  \theta(\xi-\eta, \eta).
\end{equation}
Note that
\[
\big| \int \overline{\widehat{(SU)_{\mu}}}(\xi)\widehat{(SU)_{\nu}}(\xi-\eta) \widehat{U_{\kappa}}(\eta) p_{\mu, \nu}^{\kappa,\tau}(\xi-\eta, \eta) \psi_{\leq 0 }((1+t)^{5/4+3p_0}|\xi-\eta|)d \eta d \xi\big|
\]
\[
\lesssim \sum_{k_1\leq k-10} \sum_{2^{k_1}\leq (1+t)^{-5/4-3p_0}} 2^{k_1/2+k} \| P_{k}[SU]\|_{L^2} \| P_{k_1}[SU]\|_{L^\infty} \|P_{k}U\|_{L^2}
\]
\begin{equation}\label{equation110}
\lesssim \sum_{2^{k_1}\leq (1+t)^{-5/4-3p_0}} 2^{(1-p)k_1} \| SU\|_{H^{N_1,p}}^2\| U\|_{H^{N_1,p}} \lesssim (1+t)^{-1+2p_0}\epsilon_0^2.
\end{equation}
From above estimate, we can see that the low frequency part already has sufficient decay rate. It is  not needed to cancel it out. This observation motives us to define the following cubic correction terms with time dependent cutoff functions, which only cancel out the case when $|\xi-\eta|\gtrsim (1+t)^{-5/4-3p_0}$,
\begin{equation}\label{eqn593}
E_{modi}^{S, low 1}(t) :=\sum_{\tau=p,N_1} \sum_{\mu, \nu, \kappa \in\{+,-\}} \mathfrak{Re} \big[ \int \overline{\widehat{(SU)_{\mu}}(\xi)}  \widehat{(SU)_{\nu}}(\xi-\eta) \widehat{U_{\kappa}}(\eta) q_{\mu,\nu}^{\kappa,\tau}(t,\xi-\eta, \eta) 
 d \eta d \xi\big],
\end{equation}
where
\begin{equation}\label{equation99898}
q_{\mu, \nu}^{\kappa,\tau}(t,\xi-\eta, \eta) = \frac{p_{\mu, \nu}^{\kappa,\tau}(\xi-\eta, \eta)\psi_{\geq 0} ((1+t)^{5/4+3p_0}|\xi-\eta|)}{\mu|\xi|^{1/2}-\nu |\xi-\eta|^{1/2} -\kappa |\eta|^{1/2}}.
\end{equation}

Note that the following estimate holds for any $\mu, \nu, \kappa\in \{+,-\}$,
\[
\big|\mu|\xi|^{1/2}-\nu |\xi-\eta|^{1/2} -\kappa |\eta|^{1/2}\big|\theta(\xi-\eta, \eta) \gtrsim \big(\big|\xi-\eta|^{1/2}- \big| \frac{|\xi|-|\eta|}{|\xi|^{1/2}+|\eta|^{1/2} }\big| \big)\theta(\xi-\eta, \eta)
\]
\[
\gtrsim |\xi-\eta|^{1/2}\theta(\xi-\eta, \eta).
\]
With above estimate, from Lemma \ref{Snorm} and (\ref{equation69}), the following estimates hold for $\tau\in\{p, N_1\},$
\begin{equation}\label{eqn581}
\sum_{\mu, \nu, \kappa\in\{+,-\}} \| q_{\mu, \nu}^{\kappa,\tau}(\xi-\eta, \eta)\|_{\mathcal{S}^\infty_{k,k_1,k_2}} \lesssim   2^{(2\tau+1) k},
\end{equation}
\begin{equation}\label{equation72}
\sum_{\mu, \nu, \kappa\in\{+,-\}} \| \p_t q_{\mu, \nu}^{\kappa,\tau}(t,\xi-\eta, \eta)\|_{\mathcal{S}^\infty_{k,k_1,k_2}} \lesssim (1+t)^{-1}2^{(2\tau+1) k}.
\end{equation}
Note that, from (\ref{eqn581}) and (\ref{equation72}), the symbols $q_{\mu, \nu}^{\kappa,\tau}(\xi-\eta, \eta)$ and $\p_t q_{\mu, \nu}^{\kappa,\tau}(t,\xi-\eta, \eta)$ now do not contribute any smallness when $|\xi-\eta|\ll |\eta|$. 
 As $SU$ is forced to  be putted in $L^2$ and the symbol can not cover the loss of  $p$-derivatives of smallness,  there is a potential problem when $|\xi-\eta|$ is the smallest inside the time derivative of $E_{modi}^{S, low 1}(t)$.

  To get around this issue, we will add quartic correction terms to cancel problematic quartic terms inside $d E_{modi}^{S, low 1}(t)/dt$. We first identify those problematic quartic terms by calculating $dE_{modi}^{S,low1}/dt$ as follows,
\[
\Lambda_{\geq 4}[\frac{d }{d t}E_{modi}^{S, low1} ]=\sum_{\tau=p, N_1} \sum_{\mu, \nu, \kappa\in\{+,-\}} 
 \mathfrak{G}_{\mu, \nu}^{\kappa,\tau} +\sum_{\mu, \nu, \mu',\nu' \in\{+,-\}} \mathfrak{G}_{\mu, \nu;\mu',\nu'}^{\tau},
\]
where
\[
 \mathfrak{G}_{\mu, \nu}^{\kappa,\tau} = \mathfrak{Re}\Big[   \int \Big(\overline{ \widehat{\Lambda_{\geq 3}[\p_t(SU)_{\mu}]}(\xi)}  \widehat{(SU)_{\nu}}(\xi-\eta) \widehat{U_{\kappa}}(\eta) 
\]
\[
+ \overline{\widehat{(SU)_{\mu}}(\xi)} \widehat{\Lambda_{\geq 2}[(\p_t SU)_{\nu}]}(\xi-\eta) \widehat{U_{\kappa}}(\eta)+\overline{\widehat{(SU)_{\mu}}(\xi) } \widehat{(SU)_{\nu}}(\xi-\eta) \widehat{\Lambda_{\geq 3}[\p_t U_{\kappa}]}(\eta) \Big)
\]
\[
\times q_{\mu,\nu}^{\kappa,\tau}(t,\xi-\eta, \eta)+ \overline{ \widehat{(SU)_{\mu}}(\xi)}  \widehat{(SU)_{\nu}}(\xi-\eta) \widehat{U_{\kappa}}(\eta) \p_t q_{\mu,\nu}^{\kappa,k}(t,\xi-\eta, \eta)
 d \eta d \xi\Big],
\]
\[
\mathfrak{G}_{\mu, \nu;\mu',\nu'}^{ \tau} = \mathfrak{Re}\Big[  \int \overline{\widehat{(SU)_{\mu}}(\xi)} \widehat{(SU)_{\nu}}(\xi-\eta) \widehat{{U}_{\mu'}}(\eta-\sigma) \widehat{{U}_{\nu'}}(\sigma) e_{\mu, \nu;\mu',\nu'}^{  \tau,1}(t,\xi,\eta, \sigma)  
\]
\begin{equation}\label{equation99001}
+  \overline{ \widehat{U_{\mu}}(\xi)} \widehat{(SU)_{\nu}}(\xi-\eta) \widehat{{U}_{\mu'}}(\eta-\sigma)  \widehat{{U}_{\nu'}}(\sigma)e_{\mu, \nu;\mu',\nu'}^{ \tau,2}(t,\xi,\eta, \sigma)  d \sigma d \eta d \xi\Big],
\end{equation}
where 
\[
e_{\mu, \nu}^{ \tau,1}(t,\xi,\eta,\sigma) = \sum_{\kappa\in\{+,-\}} q_{\mu, \nu}^{\kappa, \tau}(t,\xi-\eta, \eta) P_{\kappa}[q_{\kappa\mu',\kappa\nu'}(\eta-\sigma, \sigma)] 
\]
\begin{equation}\label{equation83}
 + q_{\kappa, \nu}^{\mu', \tau}(t,\xi-\eta, \eta-\sigma) P_{-\kappa}[q_{\mu, -\kappa\nu'}^1(\xi,-\sigma)],
\end{equation}
\begin{equation}\label{equation84}
e_{\mu, \nu;\mu',\nu'}^{\tau,2}(t,\xi,\eta, \sigma) = \sum_{\kappa\in\{+,-\}} q_{\kappa, \nu}^{\mu', \tau}(t,\xi-\eta, \eta-\sigma) P_{-\kappa}[q_{\mu, -\kappa\nu'}^2(\xi,-\sigma)]
\end{equation}
and $P_\mu[f]=f_{\mu}$. From (\ref{eqn581}), (\ref{equation72}), (\ref{equation92}), and  (\ref{productofsymbol}) in Lemma \ref{boundness}, the following estimate holds, 
\begin{equation}\label{equation99000}
\| e_{\mu, \nu}^{ \tau,1}(t,\xi,\eta,\sigma)\|_{\mathcal{S}^\infty_{k,k_1,k_2,k_3}} +\| e_{\mu, \nu}^{ \tau,2}(t,\xi,\eta,\sigma)\|_{\mathcal{S}^\infty_{k,k_1,k_2,k_3}} \lesssim 2^{k/2 + (2\tau+2)\max\{k_1,k_2,k_3\}}.
\end{equation}

Note that $ \mathfrak{G}_{\mu, \nu}^{\kappa,\tau}$ is not problematic. There are two types of terms inside $ \mathfrak{G}_{\mu, \nu}^{\kappa,\tau}$: (i) term like $\overline{\widehat{(SU)_{\mu}}(\xi)} $ $\widehat{\Lambda_{\geq 2}[(\p_t SU)_{\nu}]}(\xi-\eta) \widehat{U_{\kappa}}(\eta)$, which is quartic and higher. As we can gain half degree of smallness from the symbol of $\widehat{\Lambda_{\geq 2}[(\p_t SU)_{\nu}]}(\xi-\eta)$, losing $p$ derivatives in $SU$ is not a issue; (ii) all other terms inside $ \mathfrak{G}_{\mu, \nu}^{\kappa,\tau}$. Note that the decay rate of those terms  is at least $(1+t)^{-3/2+2p_0}$. The extra gain of $(1+t)^{-1/2}$  can cover the loss of $p$ derivatives in $SU$. More precisely, we at most lose $|\xi-\eta|^{-p}$. Recall that $|\xi-\eta| \gtrsim (1+t)^{-5/4-3p_0}$ inside the support of cut-off function, hence we at most lose $(1+t)^{1/4+p_0}$, which can be covered from the extra gain of $(1+t)^{-1/2}$. 

Recall (\ref{equation83}), (\ref{equation84}), (\ref{equation92}), (\ref{eqn581}) and (\ref{equation72}).  We know that the new introduced symbols $q_{\mu, \nu}(\cdot, \cdot)$ and $q_{\mu, \nu}^i(\cdot, \cdot)$ inside $e_{\mu, \nu}^{ \tau,i}(t,\xi,\eta,\sigma)$, $i\in\{1,2\}$, contribute half degree of smallness, which is less than or equal to  the second smallest number among  $\xi$, $\xi-\eta$, $\eta-\sigma$, and $\sigma$. Intuitively speaking, as $\xi-(\xi-\eta+\eta-\sigma+\sigma)=0$, we know the largest two number are comparable and the half degree of smallness is less than the biggest number, hence the second largest. Therefore, it less than or equal to the second smallest number.

As a result, after combining the estimate (\ref{equation99000}), the following estimate holds,
\[
\| e_{\mu, \nu}^{ \tau,1}(t,\xi,\eta,\sigma)\psi_k(\xi)\psi_{k_1}(\xi-\eta)\psi_{k_2}(\eta-\sigma)\psi_{k_3}(\sigma)\|_{\mathcal{S}^\infty } +\| \psi_k(\xi)\psi_{k_1}(\xi-\eta)\psi_{k_2}(\eta-\sigma)\psi_{k_3}(\sigma)\|_{\mathcal{S}^\infty }
 \]
\begin{equation}\label{equation99}
\lesssim \min\{ 2^{k/2 + (2\tau+2)\max\{k_1,k_2,k_3\}},  2^{\textup{Smin}\{k,k_1,k_2,k_3\}/2 + (2\tau+2)\max\{k_1,k_2,k_3\}}\},
\end{equation}
where ``$\textup{Smin}\{k,k_1,k_2,k_3\}$"denotes the second smallest number among $k,k_1,k_2,$ and $k_3$. 

Recall the detail formula of $\mathfrak{G}_{\mu, \nu;\mu',\nu'}^{ \tau}$ in (\ref{equation99001}). By multilinear estimate, we put all $SU$ in $L^2$ and all  $U$ in $L^\infty$. As a result, the total loss is at most of size $|\xi|^{-p}|\xi-\eta|^{-p}=2^{-pk-p k_1}$. From (\ref{equation99}), we can see that the loss of $2^{-pk-p k_1}$ can be covered by the symbol if $ 2^{k_1}\gtrsim 2^{\textup{Smin}\{k,k_1,k_2,k_3\}}$.

 Recall that, due to the time dependent cutoff function of $|\xi-\eta|$ in (\ref{equation99898}), we have $|\xi-\eta|\geq (1+t)^{-5/4-3p_0}$.  Again, from (\ref{equation99}), we know that the loss can also be covered by the symbol if $2^{(1/2-p)k}$ is less than $2^{-pk_1}\sim |\xi-\eta|^{-p}\leq (1+t)^{1/4+p_0}$.  Hence,  we can further rule out the case when $|\xi|\sim 2^k \leq (1+t)^{-5/6-5p_0}$. 

To sum up,  the quartic terms inside $\mathfrak{G}_{\mu, \nu;\mu',\nu'}^{ \tau} $  are  only problematic if $|\xi-\eta|$ is the smallest number and not comparable with the second smallest number among  $\xi$, $\xi-\eta$, $\eta-\sigma$, and $\sigma$ and $|\xi|\sim 2^k \gtrsim(1+t)^{-5/6-5p_0} $.

For this problematic scenario, a key observation is that the size of phases in this case is greater than the second smallest number instead of the smallest number among $\xi$, $\xi-\eta$, $\eta-\sigma$, and $\sigma$. More precisely, the following estimate holds,
\[
\big|\mu |\xi|^{1/2} - \nu|\xi-\eta|^{1/2} - \mu' |\eta-\sigma|^{1/2}-\nu'|\sigma|^{1/2}\big|\theta(\xi-\eta, \eta-\sigma)\theta(\xi-\eta, \sigma)\theta(\xi-\eta, \xi)
\]
\begin{equation}\label{equation98}
 \gtrsim \min\{|\xi|, |\eta-\sigma|, |\sigma|\}^{1/2},\quad \mu, \nu, \mu', \nu'\in\{+,-\}.
\end{equation}
Above estimate and a similar type estimate will be used  in the $L^\infty$ decay estimate part, hence we postpone the prove of (\ref{equation98}) to the end of  subsection \ref{Znormcubicterms}.

As the size of phase is not small, we can dividing the phase again to gain another $t^{-1/2}$ decay with the price of $2^{-\textup{Smin}\{k,k_1,k_2,k_3\}/2}$, which can be covered from the symbol, see (\ref{equation99}). This observation motives us to define the following quartic correction terms, which cancel out the problematic quartic terms inside $\mathfrak{G}_{\mu, \nu;\mu',\nu'}^{ \tau}$,
\begin{equation}\label{eqn594}
E_{modi}^{S, low2}(t) = \sum_{\tau=p,N_1}\sum_{\mu, \nu, \mu',\nu' \in\{+,-\}}\mathfrak{Re}\Big[ \int \overline{\widehat{(SU)_{\mu}}(\xi)} \widehat{(SU)_{\nu}}(\xi-\eta) \widehat{{U}_{\mu'}}(\eta-\sigma) \widehat{{U}_{\nu'}}(\sigma) 
\end{equation}
\[ 
\times \tilde{e}_{\mu, \nu;\mu',\nu'}^{\tau,1}(t,\xi,\eta, \sigma)+ \overline{ \widehat{U_{\mu}}(\xi-\sigma)} \widehat{(SU)_{\nu}}(\xi-\eta) \widehat{{U}_{\mu'}}(\eta)  \widehat{{U}_{\nu'}}(\sigma)\tilde{e}_{\mu, \nu;\mu',\nu'}^{ \tau,2}(t,\xi,\eta, \sigma)  d \sigma d \eta d \xi\Big],
\]
where for $i\in\{1,2\}$,
\[
\tilde{e}_{\mu, \nu;\mu',\nu'}^{  \tau,i}(t,\xi,\eta, \sigma) = \frac{{e}_{\mu, \nu;\mu',\nu'}^{  \tau,i}(t,\xi,\eta, \sigma) \theta(\xi-\eta, \min\{|\sigma|,|\eta-\sigma|,|\xi|\} )\psi_{\geq 0}(\xi(1+t)^{5/6+5p_0})}{\mu |\xi|^{1/2} -\nu|\xi-\eta|^{1/2} - \mu' |\eta-\sigma|^{1/2}-\nu'|\sigma|^{1/2}}.
\]
From (\ref{equation98}) and (\ref{equation99}), we can see that the loss from the denominator of $\tilde{e}_{\mu, \nu;\mu',\nu'}^{  \tau,i} $ can be covered by the size of $e_{\mu, \nu}^{ \tau,i}$, $i\in\{1,2\}.$ As a result, from (\ref{equation92}), (\ref{equation83}),  (\ref{equation84}), (\ref{eqn581}), (\ref{equation72}), Lemma \ref{Snorm}, and (\ref{productofsymbol}) in Lemma \ref{boundness}, the following estimate holds,
\begin{equation}\label{equation100}
 \| \tilde{e}_{\mu, \nu;\mu',\nu'}^{  \tau,1}(t,\xi,\eta, \sigma)\|_{\mathcal{S}^\infty_{k,k_1,k_2,k_3}} +  \| \tilde{e}_{\mu, \nu;\mu',\nu'}^{  \tau,2}(t,\xi,\eta, \sigma)\|_{\mathcal{S}^\infty_{k,k_1,k_2,k_3}} \lesssim 2^{(2\tau+2)\max\{k_1,k_2,k_3\}},
\end{equation}
\[
  \| \p_t \tilde{e}_{\mu, \nu;\mu',\nu'}^{  \tau,1}(t,\xi,\eta, \sigma)\|_{\mathcal{S}^\infty_{k,k_1,k_2,k_3}}+ \| \p_t \tilde{e}_{\mu, \nu;\mu',\nu'}^{  \tau,2}(t,\xi,\eta, \sigma)\|_{\mathcal{S}^\infty_{k,k_1,k_2,k_3}}\]
 \begin{equation}\label{equation102} \lesssim (1+t)^{-1} 2^{(2\tau+2)\max\{k_1,k_2,k_3\}}.
\end{equation}

Note that $|\xi|^{-p}|\xi-\eta|^{-p}\lesssim (1+t)^{5/12+2p_0}\leq (1+t)^{1/2-100p_0}$ inside the size of support, hence the extra gain of $t^{-1/2}$ is sufficient to close the argument. 

From above discussions, we define the total correction terms for the low frequency part as follows,
\begin{equation}\label{eqn1230}
E_{modi}^{S, low}(t) := E_{modi}^{S, low1}(t) + E_{modi}^{S, low2}(t).
\end{equation}
\begin{lemma}\label{scalingauxiliary3}
Under the bootstrap assumption \textup{(\ref{assumption})}, we have
\begin{equation}\label{eqn592}
\sup_{t\in[0,T]} |E_{modi}^{S, low}(t)| \lesssim \epsilon_0^2,
\end{equation}
\begin{equation}\label{eqn591}
\Big| \frac{d E_{modi}^{S, low}(t)}{d t}  - \sum_{\tau=p, N_1} \mathfrak{Re} \big[ \int \overline{\p_x^\tau S U^1} \p_x^\tau  \mathfrak{Q}_2^S  + \overline{\p_x^\tau SU^2}\p_x^\tau \mathfrak{Q}_4^S  \big]  \Big|\lesssim (1+t)^{-1+2p_0} \epsilon_1^3.
\end{equation}
\end{lemma}
\begin{proof}
Recall (\ref{eqn593}). From (\ref{eqn581}) and the $L^2-L^\infty$ type bilinear estimate (\ref{bilinearestimate}) in Lemma \ref{boundness}, the following estimate holds,
\begin{equation}\label{eqn1770}
|E_{modi}^{S, low1}(t)|\lesssim \| S U(t)\|_{H^{N_1, p}}^2 \| U(t)\|_{W^3} \sum_{k\in\mathbb{Z},  2^k\geq (1+t)^{-5/4-3p_0}} 2^{-pk}\lesssim (1+t)^{-1/4+10p_0} \epsilon_1^3\lesssim \epsilon_0^2.
\end{equation}
Recall (\ref{eqn594}). From (\ref{equation100}) and the $L^2-L^\infty-L^\infty$ type trilinear estimate (\ref{trilinearestimate}) in Lemma \ref{boundness}, we have
\[
|E_{modi}^{S, low2}(t)| \lesssim \sum_{2^{k_1}\geq (1+t)^{-5/4-3p_0}} \sum_{2^k \geq (1+t)^{-5/6-5p_0}}2^{-p k_1-pk} (\| SU\|_{H^{N_1, p}}+ \| U\|_{H^{N_0,p}})^{2}\| U\|_{W^3}^2
\]
\begin{equation}\label{eqn1771}
\lesssim (1+t)^{-1/2}\epsilon_1^4 \lesssim \epsilon_0^2.
\end{equation}
From (\ref{eqn1770}) and (\ref{eqn1771}), we can see the desired estimate (\ref{eqn592}) holds. From (\ref{equation110}), (\ref{equation100}), (\ref{equation102}) and $L^2-L^\infty$ type multilinear estimate in Lemma \ref{boundness}, the following estimate holds,
\[
\textup{L.H.S. of (\ref{eqn591}) } \lesssim\sum_{\tau=p, N_1} \sum_{\mu, \nu, \kappa\in\{+,-\}} 
 |\mathfrak{G}_{\mu, \nu}^{\kappa,\tau}| +\big| \sum_{\tau=p,N_1} \sum_{\mu, \nu, \mu',\nu' \in\{+,-\}} \mathfrak{G}_{\mu, \nu;\mu',\nu'}^{\tau} + \frac{d}{d t} E_{modi}^{S, low2}(t)|
 \]
\[
 + \sum_{\mu, \nu, \kappa \in\{+,-\}}  \Big|  \int \overline{\widehat{(SU)_{\mu}}}(\xi)\widehat{(SU)_{\nu}}(\xi-\eta) \widehat{U_{\kappa}}(\eta) p_{\mu, \nu}^{\kappa,\tau}(\xi-\eta, \eta)  \psi_{\leq 0} ((1+t)^{5/4+3p_0}|\xi-\eta|) d \eta d \xi  \Big|
\]
\[
\lesssim \| SU\|_{H^{N_1,p}}^2[ (1+t)^{-1/2} \| U\|_{W^3} + \| U\|_{H^{N_0,p}} (1+t)^{-(1+2p_0)}]  + (\| SU\|_{H^{N_1, p}} + \|U\|_{H^{N_0, p}})^2\times \]
\[
\| U\|_{W^3}^2 
+  \sum_{2^{k_1}\geq (1+t)^{-5/4-3p_0}} \sum_{2^k \geq (1+t)^{-5/6-5p_0}}2^{-p k_1-pk} \big( \| \Lambda_{\geq 2}[\p_t SU]\|_{H^{N_1,p}}\| SU\|_{H^{N_1,p}}\|U\|_{W^3}^2 
\]
\[
+ \| SU\|_{H^{N_1,p}}^2 \| U\|_{W^2}\| \Lambda_{\geq 2}[\p_t U]\|_{W^2}\big) +\sum_{2^{k_1}\geq (1+t)^{-5/4-3p_0}} 2^{-p k_1} \| SU\|_{H^{N_1,p}}^2 \| U\|_{W^3}\| \Lambda_{\geq 2}[\p_t U]\|_{W^3}
\]
\begin{equation}\label{eqn1662}
 \lesssim (1+t)^{-1+2p_0} \epsilon_1^3.
\end{equation}
In above estimate, we used the following fact, which is derived from Sobolev embedding,
\[
\| \Lambda_{\geq 2}[\p_t U]\|_{W^3} \lesssim \| U\|_{W^3} \| U\|_{W^4}\lesssim \| U\|_{W^{N_2}}^{3/2+10p_0}\| U\|_{H^{N_0,p}}^{1/2-2p_0}.
\]
\end{proof}
Combining (\ref{equation8916}), (\ref{equation6200}),  (\ref{eqn540}), (\ref{eqn1232}), (\ref{eqn592}) and (\ref{eqn591}), we can see our desired energy estimate (\ref{equation12}) in Proposition (\ref{mainproposition1}) holds. 

\section{Improved dispersion estimate }\label{improveddispersion}
The goal of this section is to derive the sharp $L^\infty$-type estimate under the bootstrap  assumptions (\ref{assumption}), hence finish the bootstrap argument.  From (\ref{equation12}) in Proposition (\ref{mainproposition1}), we have
\begin{equation}\label{improvedenergy}
\sup_{t\in[0,T]} (1+t)^{-p_0}\big[ \| (U^1, U^2) \|_{H^{N_0,p}} + \| S(U^1, U^2)\|_{H^{N_1, p}}\big] \lesssim \epsilon_0.
\end{equation}
From the improved energy estimate (\ref{improvedenergy}), we can first rule out the very low frequency case and very high frequency case as follows,
\[
\sum_{2^k \leq (1+t)^{-5/3-2p_0} \,\textup{or}\,2^{k}\geq (1+t)^{10/89+p_0}} \|P_k(U^1, U^2)\|_{W^{N_2}} 
\]
\begin{equation}\label{eqn1510}
\lesssim  \sum_{2^k \leq (1+t)^{-5/3-2p_0} \,\textup{or}\,2^{k}\geq (1+t)^{10/89+p_0}} 2^{k/2 + N_2 k_+}   \|P_k(U^1, U^2)\|_{L^2} \lesssim (1+t)^{-1/2}\epsilon_0.
\end{equation}
To derive the sharp decay estimate  for the case  when $(1+t)^{-5/3-2p_0}\leq |\xi|\leq (1+t)^{10/89+p_0}$, we  use the following auxiliary $Z$-normed  space \begin{equation}\label{eqn1120}
\| h\|_{Z}:=\| |\xi|^{\beta}(1 +|\xi|^{\gamma})  {h}(\xi)\|_{L^\infty_\xi},\quad \beta=\frac{3}{4}-p_0, \gamma= N_2 + 2p_0,
\end{equation}
and the  linear decay estimate (\ref{eqn400}) in Lemma \ref{lineardecay}. To simplify notations, we define $\Lambda:=\d^{1/2}$ and $\Lambda(\xi):=|\xi|^{1/2}$ throughout this section.
\begin{lemma}\label{lineardecay}
For any $t, |t|\geq 1$, $k \in \mathbb{Z}$,  $\epsilon \in(0,1/2]$ and $f\in L^2(\R)$, the following estimate holds,
\begin{equation}\label{eqn400}
\| e^{i t \Lambda} P_k f\|_{L^\infty} \lesssim |t|^{-\h} 2^{3k/4} \| \widehat{f}\|_{L^\infty} + |t|^{-(1/2+\epsilon/2)} 2^{(1-\epsilon)k/4}
[2^k\| \p_\xi \widehat{f}\|_{L^2} + \| f\|_{L^2}].
\end{equation}
\end{lemma}
\begin{proof}
Note that estimate  (\ref{eqn400})  is scaling invariant, it is sufficient to prove it for the case when $k=0$. The proof is very similar to the proof of linear decay estimates in \cite{IP2,IP3,IP4}. With minor modifications, one can rerun the argument used  in \cite{IP2} to derive (\ref{eqn400}) with any problems.  
\end{proof}

The strategy that we will use to get the sharp decay estimate for $U^1$ and $U^2$ is as follows:  (i) Recall (\ref{eqn1600}). We first derive the equation satisfied by the normal form transformation($V_1$ and $V_2$) of $U^1$ and $U^2$, which  is cubic and higher. Then we show that the decay rates of $(U^1,U^2)$ and $(V_1,V_2)$ are comparable. (\mbox{ii}) We  prove that the $Z$-norm of the  profile of $(V_1, V_2)$ does not grow with respect to time and of size $\epsilon_0$. Hence $V^1$ and $V^2$ decay sharply from Lemma \ref{lineardecay}. As the decay rate of $(U^1,U^2)$ is same as $(V_1, V_2)$, we know that  $U^1$ and $U^2$ also decay sharply.

 Recall (\ref{eqn1600}). Define  $V=V_1+ i\, V_2$ and $\mathcal{S}:=\{(+,+,+), (+, +, -), (+,-,-), (-,-,-) \}$. As the equation satisfied by $V$ is cubic and higher, we can write  it as follows, 
\begin{equation}\label{eqn610}
(\p_{t} + i \Lambda ) V = \sum_{(\iota_{1}, \iota_{2}, \iota_{3})\in\mathcal{S}}C_{\iota_1, \iota_2, \iota_3}(V^{\iota_{1}}, V^{\iota_{2}}, V^{\iota_{3}}) +  R(t),
\end{equation}
where $R(t)$ represents the quintic and higher order terms. Define $f(t):= e^{it \Lambda}V$. We use  $\widehat{f_k}(t,\xi)$ to abbreviate $\widehat{f}(t,\xi)\psi_k(\xi)$ throughout this section.

\begin{lemma}\label{sizeinfo}
Under the smallness condition \textup{(\ref{assumption})} and the improved energy estimate \textup{(\ref{improvedenergy})}, the following estimates hold for $t\in[0,T]$,
\begin{equation}\label{eqn107}
\| (V_1-U_1, V_2-U_2)(t)\|_{W^{N_2}} \lesssim (1+t)^{-5/8} \epsilon_1^2,\quad  \| V(t) \|_{W^{N_2}} \lesssim  (1+t)^{-1/2}\epsilon_1,\end{equation}
\begin{equation}\label{eqn1245}
 \sup_{k\in\mathbb{Z}} \big[ 2^{k_{-}/5+ 7k_{+}}\| P_kf(t)\|_{L^2} + 2^{k_{-}/5}\| \xi \p_\xi \widehat{f}( \xi, t)\psi_k(\xi)\|_{L^2}\big] \lesssim (1+t)^{p_0} \epsilon_0.
\end{equation}
\end{lemma}

\begin{proof}
From estimate (\ref{equation8905}), bilinear estimate (\ref{bilinearestimate}) in Lemma \ref{boundness}, and Sobolev embedding, our desired estimate (\ref{eqn107}) holds as follows,
\[
\| A_1(U^1, U^1)\|_{W^{N_2}} + \|A_2(U^2, U^2)\|_{W^{N_2}}  + \| B(U^1, U^2)\|_{W^{N_2}}  \]
\[\lesssim \sum_{k\in \mathbb{Z}} 2^{5k/4 + N_2 k_{+}}  \| (P_{\leq k} U^1, P_{\leq k} U^2)\|_{W^{0}} \| (P_k U^1, P_k U^2)\|_{L^\infty}^{1/2}\| (P_k U^1, P_k U^2)\|_{L^2}^{1/2} \]
\[
+ \sum_{k\in\mathbb{Z}}\sum_{k_1 \leq k} 2^{k_1/4+k + N_2 k_{1,+}} \| P_{k} U^1\|_{L^4} \| P_k U^2\|_{L^\infty}\]
\[ \lesssim \| (U^1, U^2)\|_{H^{N_0, p}}^{1/2} \| (U^1, U^2)\|_{W^{N_2}}^{3/2} \lesssim (1+t)^{-5/8}\epsilon_1^2,
\]
We also have the following $L^2$- type estimate,
\[
\| A_1(U^1, U^1)(t)\|_{H^{N_0-1, p}} + \| A_2(U^2, U^2)(t)\|_{H^{N_0-1, p}} + \| B(U^1, U^2)(t)\|_{H^{N_0-1, p}}
\]
\[
\lesssim \| (U^1, U^2)(t)\|_{H^{N_0, p}}\| (U^1, U^2)(t)\|_{W^{0}}\lesssim (1+t)^{-(1/2-p_0)} \epsilon_1^2,
\]
which further gives us the following estimate,
\begin{equation}\label{eqn1242}
 \sup_{t\in[0, T]}\sup_{k\in\mathbb{Z}} (1+t)^{-p_0}  2^{k_{-}/5+7k_{+}}\| P_kf(t)\|_{L^2}\lesssim \epsilon_0.
\end{equation} 
 Recall that  $A(\cdot, \cdot)$ is a symmetric bilinear operator, the following estimate holds for fixed $k\in\mathbb{Z}$,  
\[
\| P_{k}[S A_1(U^1, U^1)]\|_{L^2} \lesssim \| P_{k}[A_1(SU^1, U^1)]\|_{L^2} + \| P_{k}[  A_1( U^1, U^1)]\|_{L^2}
\]
\[
\lesssim \sum_{k_1 \leq k-4} 2^k \big[ \| P_k U^2\|_{L^2} \| P_{k_1}(SU^1)\|_{L^\infty}\| P_k U^1\|_{L^\infty} \| P_{k_1}(SU^1)\|_{L^2} \big]^{1/2} 
\]
\[
+ 2^{k}\| P_k(SU^1,  U^1)\|_{L^2} \| U^1\|_{W^0}\lesssim \sum_{k_1\leq k-4} 2^{k_1/4} 2^k \| P_{k_1}(SU^1)\|_{L^2} \| P_{k} U^1\|_{L^2}^{1/2}\|P_{k}U^1\|_{L^\infty}^{1/2} \]
\[+ (1+|t|)^{-1/2+p_0}\epsilon_1^2 \lesssim (1+t)^{-1/4+2p_0}\epsilon_1^2 \lesssim \epsilon_0.
\]
Following the same arugment, we can show that
\begin{equation}\label{eqn1241}
\sup_{k\in \mathbb{Z}}\| P_{k}[S A_2(U^2, U^2)]\|_{L^2} + \| P_{k}[S B(U^1, U^2)]\|_{L^2}\lesssim \epsilon_0.
\end{equation}
Note that $ \widehat{SV}(\xi)= e^{-it |\xi|^{1/2}}(t\p_t -2\xi\p_{\xi}-2)\widehat{f}(t,\xi).$ Hence
\begin{equation}\label{eqn1640}
 \sup_{t\in[0, T]}\sup_{k\in\mathbb{Z}} (1+t)^{-p_0}  2^{k_{-}/5}\| (t\p_t -2 \xi \p_\xi) \widehat{f}(t, \xi)\psi_k(\xi)\|_{L^2} \lesssim \epsilon_0.
\end{equation}
From estimate (\ref{equation4000}) in Lemma \ref{symbolbound}, $L^2-L^\infty-L^\infty$ type trilinear estimate (\ref{trilinearestimate}) in Lemma \ref{boundness} and (\ref{eqn433}) in Lemma \ref{estimatesofremainder}, the following estimate holds, 
\[
\| t \p_t f \|_{L^2}\lesssim t \sum_{k_1, k_2, k_3 \in \mathbb{Z }} 2^{\med\{k_i\}/2} 2^{2\max\{k_i\}} \| P_{\min\{k_i\}}(e^{-i t\Lambda} f)\|_{L^\infty}\| P_{\med\{k_i\}}(e^{-it \Lambda} f)\|_{L^\infty}\times
\]
\begin{equation}\label{eqn1240}
\| P_{\max\{k_i\}} f \|_{L^2} + t \| R\|_{L^2}\lesssim t \| V(t) \|_{W^{1/2}}^2  \| V \|_{H^{2, p}}\lesssim (1+|t|)^{p_0}\epsilon_1^3\lesssim (1+|t|)^{p_0} \epsilon_0.
\end{equation}
Combining (\ref{eqn1242}), (\ref{eqn1640}) and (\ref{eqn1240}), it is  easy to see the desired estimate (\ref{eqn1245}) holds. With minor modifications, we can also derive the following estimate, 
\begin{equation}\label{eqn1641}
  \sup_{ (1+t)^{-5/3-2p_0} \leq \,2^{k}\leq (1+t)^{10/89+p_0} }  2^{k_{-}/5+ N_1 k}\| \xi \p_\xi \widehat{f}(t, \xi)\psi_k(\xi)\|_{L^2}\lesssim (1+t)^{p_0}\epsilon_0.
\end{equation}

\end{proof}
From (\ref{eqn1245}) and  (\ref{eqn1641}),  the following estimate holds, 
\[ 
\sum_{(1+t)^{-5/3-2p_0} \leq \,2^{k}\leq (1+t)^{10/89+p_0}}   2^{k/8+N_2 k_+}
[2^k\| \p_\xi \widehat{f}\psi_{k}(\xi)\|_{L^2} + \| P_{k}f\|_{L^2}]\lesssim (1+|t|)^{1/4-2p_0}\epsilon_0.
\]
Hence, from above estimate and   (\ref{eqn400}) in Lemma \ref{lineardecay},  it is  sufficient to derive the improved  $Z$-norm estimate of $f$ to derive  the sharp decay estimate for $V$.

\subsection{Set-up of the $Z$-norm estimate}

 After writing the cubic terms on the Fourier side and do Littlewood-Paley decomposition for all inputs, we have
\[
 \mathcal{F}(e^{ i t \Lambda}C_{\iota_1, \iota_2, \iota_3}(V^{\iota_{1}}, V^{\iota_{2}}, V^{\iota_{3}})) (\xi):=iI^{\iota_1,\iota_2,\iota_3}( t,\xi) = i\sum_{ k_{1}, k_{2}, k_{3}\in \mathbb{Z}} I^{\iota_1,\iota_2,\iota_3}_{k_{1}, k_{2}, k_{3}}( t,\xi) ,
\]
\[
I^{\iota_1, \iota_2, \iota_3}_{k_{1}, k_{2}, k_{3} } ( t,\xi) = \int_{\R\times \R} e^{it\,\Phi_{\iota_1, \iota_2, \iota_3}(\xi, \eta, \sigma)}\,c^{\iota_1, \iota_2, \iota_3}(\xi-\eta, \eta-\sigma,\sigma) \widehat{f^{\iota_1}_{k_{1}}}(t,\xi-\eta) \times\]
\[\widehat{f^{\iota_2}_{k_{2}}}(t,\eta-\sigma) \widehat{f^{\iota_3}_{k_{3}}}(t,\sigma)d\eta d\sigma,
\]
where the phase $\Phi_{\iota_1, \iota_2, \iota_3}(\cdot, \cdot, \cdot)$ is defined as follows,
\[
\Phi_{\iota_1, \iota_2, \iota_3}(\xi, \eta, \sigma) = \Lambda(\xi) - \iota_1 \Lambda(\xi-\eta) -\iota_2\Lambda(\eta-\sigma) - \iota_3 \Lambda(\sigma), \quad (\iota_1, \iota_2, \iota_3)\in \mathcal{S}.
\]
The precise formulas of symbols $c^{\iota_1, \iota_2, \iota_3}(\cdot, \cdot, \cdot)$, $(\iota_1, \iota_2, \iota_3)\in\mathcal{S}$ do not play many roles,  it is enough to know the estimate of the $\mathcal{S}^\infty$ norm of symbols  $c^{\iota_1, \iota_2, \iota_3}(\cdot, \cdot, \cdot)$. More precisely, the following lemma  holds,
\begin{lemma}\label{symbolbound}
For any $(\iota_1, \iota_2, \iota_3)\in \mathcal{S}$, $k,k_1,k_2, k_3\in\mathbb{Z}$,  the following estimate holds,
\begin{equation}\label{equation4000}
\|c^{\iota_{1}, \iota_{2}, \iota_{3}}(\xi-\eta, \eta-\sigma, \sigma)\|_{\mathcal{S}^\infty_{k,k_1,k_2,k_3}} \lesssim 2^{\med\{k_{1}, k_2,k_3\}/2} 2^{2\max\{k_{1},k_2,k_3\}},\end{equation}
where $\med\{k_{1}, k_2,k_3\}$ denotes the medium number among $k_1,k_2$, and $k_3.$
\end{lemma}
\begin{proof}
 Recall that $V=V_1+iV_2$, $V_1=U^1 + A_{1}(U^1, U^1)+A_2(U^2, U^2)$, and $V_2= U^2+B(U^1, U^2)$.    From the Taylor expansion of the Dirichlet-Neumann operator in (\ref{eqn1002}), definitions of $V_i$ and $U_i$, $i\in\{1,2\}$, in (\ref{equation10}) and (\ref{eqn1600}), we can calculate explicit formulas of $c^{\iota_{1}, \iota_{2}, \iota_{3}}(\xi-\eta, \eta-\sigma, \sigma)$. Then, we can see our desired estimate (\ref{equation4000}) holds after applying  the estimate (\ref{eqn61001}) in Lemma \ref{Snorm}. Unfortunately, the detail formulas of $c^{\iota_{1}, \iota_{2}, \iota_{3}}(\xi-\eta, \eta-\sigma, \sigma)$ are very tedious. So, in the following, we provide a more intuitive ``proof'', which explains why our desired estimate (\ref{equation4000}) should hold.  

Note that there are two sources of cubic terms that contribute to the cubic terms inside the equation,(\ref{eqn610}), satisfied by $V$. One of them comes from $\Lambda_{3}[\p_t U_1]$ and $\Lambda_{3}[\p_t U_2]$. The other one comes from $A_{1}(\Lambda_2[\p_t U^1], U^1)$, $A_{2}(\Lambda_2[\p_t U^2], U^2)$, $B(\Lambda_2[\p_t U^1], U^2)$, and $B(U^1, \Lambda_2[\p_t U^1])$.

 Our desired estimate (\ref{equation4000}) follows from the following facts: (i) The symbols of   cubic terms are all homogeneous of degree ``$5/2$'' in terms of $U^1$ and $U^2$. Hence to prove (\ref{equation4000}), we only have to consider the case when $\med\{k_i\}\leq \max\{k_i\}-10$. (ii) Cubic terms at most lose ``$2$'' derivatives. 
  From (\ref{equation6}), we know that  $\Lambda_{3}[\p_t U_1]$  and $\Lambda_{3}[\p_t U_2]$ at most lose one derivative. From (\ref{equation8905}) in Lemma \ref{auxilary1} and (\ref{equation20}) in Lemma \ref{sizeofsymbol}, we know that cubic terms, $A_{1}(\Lambda_2[\p_t U^1], U^1)$, $A_{2}(\Lambda_2[\p_t U^2], U^2)$, $B(\Lambda_2[\p_t U^1], U^2)$, and $B(U^1, \Lambda_2[\p_t U^1])$, at most lose two derivatives. To sum up, we know that the size of symbol contributes at least half degree smallness of the medium frequency. 

 To see why fact (i) holds. Note that the symbols of normal form transformations $A_1(\cdot, \cdot)$, $A_2(\cdot, \cdot)$ and $B(\cdot, \cdot)$ are all homogeneous of degree ``$1$'' and the symbols of quadratic terms $\Lambda_{2}[\p_t U^1]$ and $\Lambda_{2}[\p_t U^2]$ are all homogeneous of degree ``$3/2$''. Therefore the symbols of all cubic terms that come from the normal form transformation are homogeneous of degree ``$5/2$''. Intuitively speaking, the  cubic term  of $\p_t U^1$ is either of type $\p_x h \p_x h \p_x \psi$ or of type $\p_x\d\psi \d \psi \d\psi$; the cubic term of $\p_t U^2$ is of type $\d^{1/2}(\p_x \psi \p_x \psi \p_x h)$ or of type $\d^{1/2}(h \p_x h \p_x h)$. Recall that $\Lambda_{1}[\psi]=\d^{-1/2}U^2, \Lambda_{1}[h]=U^1$. Now it is easy to see that they are all homogeneous of degree `` $5/2$''.

\end{proof}

Same as in \cite{IP3}, to successfully close the argument and see the modified scattering property, we need to modified the phase of the profile first.   We define
\begin{equation}\label{equation152}
 c^{\ast}(\xi,x, y) := c^{+,+,-}(\xi+x,\xi+ y, -\xi-x-y),\quad \widetilde{c}(\xi):= - 8 \pi |\xi|^{3/2} c^{\ast}(\xi,0,0) ,
\end{equation}
and the modified phase is defined as
\begin{equation}\label{eqn1000}
L(t,\xi) := {\widetilde{c}(\xi)}\int^{t}_{0} |\widehat{f}(s,\xi) |^{2} \frac{d\,s}{1+s},\quad g(t,\xi) := e^{i L(t,\xi)} \widehat{f}(t,\xi),
\end{equation}
hence
\[
\p_{t} g(\xi, t) = i e^{i L(t,\xi)} \big[ I^{+,+,-}(t,\xi) + \widetilde{c}(\xi) \frac{|\widehat{f}(t,\xi)|^{2}}{1+t} \widehat{f}(t,\xi)\big] +
\]
\begin{equation}\label{equation9980}
 i e^{i L(t,\xi)} \Big[ I^{+,+,+}(t,\xi) + I^{+,-,-}(t,\xi) + I^{-,-,-}(t,\xi)\Big] + e^{ i L(t,\xi) + i t \Lambda(\xi)}\widehat{R}(t,\xi).
\end{equation}

From  (\ref{equation9980}), we can see that the modified phase is only effective for the cubic  term $I^{+,+,-}(t,\xi)$, which is the only one that has non empty space-time resonance set. 

Now, our main goal is  to prove the following proposition.

\begin{proposition}\label{propositionZnorm2}
Under the bootstrap assumption \textup{ (\ref{assumption})} and the following assumption 
\begin{equation}\label{equation140}
\sup_{t\in[0,T']} \| \widehat{f} (t)\|_{Z}\lesssim \epsilon_1, \quad T'\in[0,T],
\end{equation}
  there exist $p_1 > 0$ such that the following estimate holds  for any $m \in\mathbb{N}$ and any $t_1, t_2 \in [2^{m-1}, 2^{m+1} ]\subset [0, T']$.
\begin{equation}\label{desiredznorm}
\| |\xi|^{3/4-p_0}(1+|\xi|^{N_2+2p_0}) \big( g(t_2,\xi) - g(t_1,\xi )  \big)\|_{L^\infty_\xi} \lesssim 2^{-p_1 m} \epsilon_0,
\end{equation}
Hence, we have $T'=T$ and
\begin{equation}
\sup_{t\in[0,T]} \| \widehat{f}(t)\|_{Z} =\sup_{t\in[0,T]} \| g (t)\|_{Z} \lesssim \epsilon_0, \quad \sup_{t\in[0,T]} (1+t)^{1/2} \| V(t)\|_{W^{N_2}} \lesssim \epsilon_0.
\end{equation}
\end{proposition}

 Now we restrict ourself inside the time interval $[2^{m-1}, 2^{m+1}]\subset[0,T']$ and reformulate estimates (\ref{eqn107}), (\ref{eqn1245}), and (\ref{equation140}) as follows,\begin{equation}\label{equation7010}
\sup_{k \in\mathbb{Z}} \|P_k f \|_{L^2} \lesssim \epsilon_0 2^{p_0 m} 2^{- k_{-}/5-7k_+}, \quad \sup_{k \in \mathbb{Z}} \| \mathcal{F}( P_k f)(\xi) \|_{L^\infty_\xi} \lesssim \epsilon_1 2^{-\beta k} 2^{-\gamma k_+},
\end{equation}
\begin{equation}\label{equation7012}
\sup_{k \in \mathbb{Z}} \|\xi \p_\xi \widehat{f}(\xi) \psi_k(\xi)\|_{L^2} \lesssim \epsilon_0 2^{p_0 m} 2^{- k_{-}/5},\quad \sup_{k\in\mathbb{Z}} \| e^{-it \Lambda} P_k f \|_{L^\infty_x} \lesssim \epsilon_1 2^{-m/2}  2^{-N_2 k_+}.  
\end{equation}
\begin{lemma}\label{znorm6}
For any $k \in \mathbb{Z}$ and $f\in L^2(\R)$, we have
\begin{equation}\label{equation141}
\| \widehat{P_k f}\|_{L^\infty_\xi}^2 \lesssim 2^{- k}\| \widehat{f}\|_{L^2} \big[ 2^k \| \p_\xi \widehat{f}\|_{L^2} + \| \widehat{f}\|_{L^2} \big].
\end{equation}
\end{lemma}
\begin{proof}
The proof is standard. The desired estimate (\ref{equation141}) is scale invariant, we only have to prove it for the case when $k=0$, which follows from the Cauchy-Schwartz inequality directly. Or one can find the detailed proof in \cite{IP3,IP4}.
\end{proof}
From estimate (\ref{equation141}) in  Lemma \ref{znorm6}, the following estimates hold when $|\xi|\leq 2^{- 21 p_0 m }$ or $|\xi|\geq 2^{20 p_0 m}$,
\[
\sup_{|\xi|\leq 2^{- 21 p_0 m } \textup{or} |\xi|\geq 2^{20 p_0 m} }  | |\xi|^{3/4-p_0}(1+|\xi|^{N_2+2p_0}) g(t, \xi)|\]
\[= \sup_{|\xi|\leq 2^{- 21 p_0 m } \textup{or} |\xi|\geq 2^{20 p_0 m} }  | |\xi|^{3/4-p_0}(1+|\xi|^{N_2+2p_0}) f(t, \xi)|\lesssim \sup_{k \leq -21p_0 m} \epsilon_0 2^{(1/4-p_0)k-k/5} 2^{p_0 m }\]
\[+\sup_{k \geq  20 p_0 m} \epsilon_0 2^{(1/4+N_2+p_0)k-(N_0-1)k/2} 2^{p_0 m } \lesssim \epsilon_0 2^{-p_0m}.
\]

Hence, to prove (\ref{desiredznorm}),  it remains to consider the case when $2^{-21p_0 m} \leq |\xi|\leq 2^{20 p_0 m}$. For this case, we need to use the equation satisfied by the modified profile ``$g(t,\xi)$". Recall (\ref{equation9980}), we have the following identity, 
\[
g(t_2, \xi) - g(t_1, \xi)=\sum_{k_1,k_2,k_3\in \mathbb{Z}}\sum_{(\iota_1, \iota_2,\iota_3)\in \mathcal{S}} J^{\iota_1, \iota_2,\iota_3 }_{k_1,k_2, k_3} + \int_{t_1}^{t_2} e^{ i L(t,\xi) + i t \Lambda(\xi)}\widehat{R}(t,\xi) d t,\]
\[J^{+, +,- }_{k_1,k_2, k_3}:=\int_{t_1}^{t_2} i e^{i L(t,\xi)} \big[ I^{+,+,-}_{k_1,k_2, k_3}(t,\xi) + \widetilde{c}(\xi) \frac{\widehat{f_{k_1}}(t,\xi)\widehat{\overline{f}}_{k_3} (t, -\xi)}{1+t} \widehat{f_{k_2}}(t,\xi)\big] d s,
\]
\[
J^{\iota_1, \iota_2,\iota_3 }_{k_1,k_2, k_3}:= \int_{t_1}^{t_2} i e^{i L(t,\xi)} I^{\iota_1, \iota_2,\iota_3}_{k_1,k_2, k_3}(t,\xi) d t, (\iota_1, \iota_2,\iota_3)\in \{(+,+,+),(+,-,-), (-,-,-)\}.
\]

The argument naturally splits into two parts : $Z$-norm estimate for the cubic terms and $Z$-norm estimate for the remainder term.
\subsection{$Z$-norm estimate for the cubic terms}\label{Znormcubicterms}
The goal of this subsection is to prove the following proposition, 
\begin{proposition}\label{cubicZnorm}
For $t_{1}, t_{2}\in[2^{m-1}, 2^{m+1}]\subset[0,T']$, $ |\xi|:=2^k \in [2^{-21p_0m} , 2^{20 p_0 m}]$,  the following estimates hold under the bootstrap assumptions \textup{(\ref{assumption})} and  \textup{(\ref{equation140})}.
\begin{equation}\label{equation4200}
\sum_{k_1,k_2,k_3\in \mathbb{Z}}\sum_{(\iota_1, \iota_2,\iota_3)\in \mathcal{S}} \|J^{\iota_1, \iota_2,\iota_3 }_{k_1,k_2, k_3} \psi_k(\xi)\|_{Z}\lesssim 2^{-p_{0}m}\epsilon_{1}^{3}.
\end{equation}
\end{proposition}

\begin{lemma}\label{auxiliarylemma1}
Under the assumptions in Proposition \textup{\ref{cubicZnorm}}, the following estimate holds if $k_{1}, k_{2}, k_{3}\in  [ k-10, k+10] $,
\begin{equation}\label{equation193}
\| J^{+, +,- }_{k_1,k_2, k_3}\psi_k(\xi)\|_{Z}\lesssim 2^{-2p_0 m }.
\end{equation}
\end{lemma}

\begin{proof}
The idea of proof is very similar to the proof of \cite{IP2}[Lemma 6.4]. Because the setting of function spaces is changed, and also for self-completeness, we still give a detailed proof here. 

 We first do change of variables to transform $(\xi-\eta, \eta-\sigma, \sigma)$ into $(\xi+\eta, \xi+\sigma, -\xi-\eta-\sigma)$, hence near the critical point $(\xi, \xi, -\xi)$ ( corresponds to the space-time resonance set), $(\eta, \sigma)=(0,0)$. After changing of variables, we decompose $I^{+,+,-}_{k_{1}, k_{2}, k_{3}}(t,\xi) $ as follows,
 \begin{equation}\label{equation7024}
I^{+,+,-}_{k_{1},k_{2}, k_{3}}(t,\xi ) = \sum_{l_{1}, l_{2}=\bar{l}}^{k+10} 
J_{l_{1}, l_{2}}(t,\xi),\quad \bar{l}=-(1-100p_0)m/2+3k/4,
\end{equation}
where
\[
J_{l_{1}, l_{2}}(t,\xi) =  \int_{ \R^{2}} 
e^{ it\Phi(\xi, \eta, \sigma)} \widehat{f}_{k_{1}}(t,\xi+\eta )
\widehat{f}_{k_{2}}(t,\xi+\sigma ) \widehat{\overline{f}}_{k_{3}}(t,-\xi-\eta-\sigma)\]
\[ c^{\ast}(\xi, \eta, \sigma) \psi_{l_{1}}^{\bar{l}}(\eta) \psi_{l_{2}}^{\bar{l}}(\sigma) d\eta d\sigma,
\]
where $c^{\ast}(\cdot, \cdot, \cdot)$ is defined in (\ref{equation152}), the phase $\Phi(\xi, \eta, \sigma)$ is defined as follows, 
\[
\Phi(\xi, \eta, \sigma):= \Lambda(\xi) -\Lambda(\xi+\eta) -\Lambda(\xi+\sigma) + \Lambda(-\xi-\eta-\sigma),
\]
and  the cutoff function $\psi^{l_1}_{l_2}(\cdot)$ is defined as follows, 
\[
\psi^{l_1}_{l_2}(\xi)= \left\{ \begin{array}{ll}
\psi_{l_2}(\xi) & \textup{if} l_2> l_1\\
\psi_{\leq l_1}(\xi) & \textup{if} l_2=l_1.\\
\end{array}\right.
\]
\noindent \textbf{Case 1:} We first consider the case when $l_{2}\geq \max\{l_{1}, \bar{l}+1\}$. For this case,  we have
\begin{equation}\label{equation4203}
|\p_{\eta}\Phi(\xi, \eta,\sigma)|\psi_{l_{2}}^{\bar{l}}(\sigma) = | \Lambda'(\xi+\eta+\sigma)- \Lambda'(\xi+\eta)|\psi_{l_{2}}^{\bar{l}}(\sigma) \gtrsim 2^{ l_{2}} 2^{-3k/2 }.
\end{equation}
After integrating by parts in $\eta$, we can derive the following:
\begin{equation}
|J_{l_{1}, l_{2}}(t,\xi)| \lesssim 2^{-m}  \big[ 
|J_{l_{1}, l_{2}}^{1}(t,\xi)| +  |J_{l_{1}, l_{2}}^{2}(t,\xi)|  \big]  
\big],
\end{equation}
where 
\[
J^{1}_{l_{1}, l_{2}}(t,\xi) = \int_{\R^{2}} e^{i t \Phi(\xi, \eta, \sigma)} \widehat{f}_{k_{1}}(t,\xi+\eta) \widehat{f}_{k_{2}}(t,\xi+\sigma) 
\widehat{\overline{f}}_{k_{3}}(t,-\xi-\eta-\sigma) \p_{\eta} r_{1}(\xi, \eta,\sigma)  \, d\eta d \sigma,
\]
\[
J^{2}_{l_{1}, l_{2}}(\xi, s) = \int_{\R^{2}} e^{i t \Phi(\xi, \eta, \sigma)} \p_{\eta}\big(\widehat{f}_{k_{1}}(t,\xi+\eta)
\widehat{\overline{f}}_{k_{3}}(t,-\xi-\eta-\sigma) \big) \widehat{f}_{k_{2}}(t,\xi+\sigma)   r_{1}(\xi,\eta,\sigma)  d\eta d \sigma,
\]
where
\[
r_{1}(\xi, \eta, \sigma) := \frac{c^{\ast}(\xi, \eta, \sigma) \psi_{l_{1}}^{\bar{l}}(\eta) \psi_{l_{2}}^{\bar{l}}(\sigma). }{\p_{\eta}\Phi(\xi, \eta, \sigma)}.
\]
From Lemma \ref{Snorm}, (\ref{equation4000}) in Lemma \ref{symbolbound}, and (\ref{equation4203}), the following estimate holds,
\begin{equation}\label{equation162}
\|r_1(\xi, \eta, \sigma)\psi_{k}(\xi)\|_{\mathcal{S}^\infty } \lesssim 2^{-l_2 + 4k}, \quad \|r_1(\xi, \eta, \sigma)\psi_k(\xi)\|_{\mathcal{S}^\infty } \lesssim 2^{-l_2 -l_1 + 4k} + 2^{-l_2+3k}.
\end{equation}
Therefore, from  (\ref{equation162}) and the $L^2-L^2-L^\infty$ type trilinear estimate (\ref{trilinearestimate}) in Lemma \ref{boundness}, the following estimates hold, 
\[
|J_{l_1, l_2}^1(t,\xi)| \lesssim \big( 2^{-l_2-l_1} 2^{4k} + 2^{-l_2+3k}\big)\| \widehat{f_{k_1}}(t,\xi+\eta)\psi_{l_{1}}^{\bar{l}}(\eta)
\|_{L^{2}} \| \widehat{f}_{k_{2}}(t,\xi+\sigma)\psi_{l_{2}}^{\bar{l}}(\sigma)\|_{L^{2}} 
\]
\[
\times \| e^{ -i t \Lambda}P_{k_3} f \|_{L^{\infty}} \lesssim \big( 2^{-l_2-l_1} 2^{4k} + 2^{-l_2+3k}\big)2^{(l_1+l_2)/2} \| \widehat{f}_k(t,\xi)\|_{L^\infty}^2 2^{-m/2-N_2 k_+}\epsilon_1
\]
\[
\lesssim 2^{-50p_0 m }\epsilon_1^3 2^{-6k_+},
\]
 \[
|J_{l_{1}, l_{2}}^{2}(t,\xi)|   \lesssim \sum_{\{i,j\}=\{1,3\}}2^{-l_2 + 4k} \| \p_\xi \widehat{f}(\xi,s)\psi_{k_i}(\xi)\|_{L^2}\| P_{k_j} f\|_{L^2} \| e^{-it\Lambda} P_{k_2} f\|_{L^\infty}
\]
\[
\lesssim 2^{-l_2 -m/2 +2p_0m} 2^{(3-2p)k - (N_2+N_0-1)k_+}\epsilon_1^3 \lesssim 2^{-50p_0 m} 2^{-6 k_+} \epsilon^3_1.
\]
To sum up, we have
\begin{equation}\label{equation5300}
|J_{l_{1}, l_{2}}(t,\xi)| \lesssim  \epsilon_{1}^{3} 2^{-(1+50p_{0})m} 2^{-6k_+}.
\end{equation}
 The symmetric case when $l_{1}\geq \max\{l_{2}, \bar{l}\}$ can be handled very similarly.

\noindent\textbf{Case 2:} It  remains  to consider the case when $l_1 = l_2 =\bar{l}$. For this case, note that the following estimate holds,
\begin{equation}\label{equation7020}
|J_{\bar{l}, \bar{l}}(\xi, s) + \frac{\tilde{c}(\xi) \widehat{f}_{k_{1}}(t,\xi) \widehat{f}_{k_{2}}(t,\xi)\widehat{\overline{f}}_{k_{3}}(t,-\xi)   }{t+1}|\lesssim |\mathcal{I}_{1}| + |\mathcal{I}_{2}|
+ |\mathcal{I}_{3}|,
\end{equation}
where
\[
\mathcal{I}_{1} =   \int_{\R^{2}} \big(e^{it \Phi(\xi, \eta, \sigma)} -e^{-i {t\eta\sigma}/(4|\xi|^{3/2})} \big) \widehat{f}_{k_{1}}(t,\xi+\eta) \widehat{f}_{k_{2}}(t,\xi+\sigma) \]
\[\times 
\widehat{\overline{f}}_{k_{3}}(t,-\xi-\eta-\sigma) c^{\ast}(\xi, \eta, \sigma) \psi_{\bar{l}}^{\bar{l}}(\eta) \psi_{\bar{l}}^{\bar{l}}(\sigma) \,d\eta d\sigma，
\]
\[
\mathcal{I}_{2} =   \int_{\R^{2}} e^{-i {t\eta\sigma}/(4|\xi|^{3/2})} \big[\widehat{f}_{k_{1}}(t,\xi+\eta) \widehat{f}_{k_{2}}(t,\xi+\sigma)\widehat{\overline{f}}_{k_{3}}(t,-\xi-\eta-\sigma)   c^{\ast}(\xi, \eta, \sigma) - 
\]
\[
\widehat{f}_{k_{1}}(t,\xi) \widehat{f}_{k_{2}}(t,\xi)\widehat{\overline{f}}_{k_{3}}(t,-\xi)  c^{\ast}(\xi,0,0) \big]  \psi_{\bar{l}}^{\bar{l}}(\eta)\psi_{\bar{l}}^{\bar{l}}(\sigma) \,d\eta d\sigma  ,
\]
\[
\mathcal{I}_{3} =  \int_{\R^{2} } e^{-i {s\eta\sigma}/(4|\xi|^{3/2})}\widehat{f}_{k_{1}}(t,\xi) \widehat{f}_{k_{2}}(t,\xi)\widehat{\overline{f}}_{k_{3}}(t,-\xi)  c^{\ast}(\xi,0,0)\psi_{\bar{l}}^{\bar{l}}(\eta)\psi_{\bar{l}}^{\bar{l}}(\sigma) \,d\eta d\sigma + 
\]
\[
 \frac{\tilde{c}(\xi) \widehat{f}_{k_{1}}(t,\xi) \widehat{f}_{k_{2}}(t,\xi)\widehat{\overline{f}}_{k_{3}}(t,-\xi)   }{t+1}
  .
\]
Note that, 
\[
\Big| \Phi(\xi, \eta, \sigma) + \frac{\eta\sigma}{4|\xi|^{3/2}} \Big| \lesssim 2^{-5k/2} (|\eta|+|\sigma|)^3.
\]
Hence, after  using estimate (\ref{equation4000}) in Lemma \ref{symbolbound} and the size of support of $\eta$ and $\sigma$, the following estimate holds,
\begin{equation}\label{equation7021}
|\mathcal{I}_1| \lesssim \epsilon_1^3 2^m 2^{5 \bar{l}} 2^{-9k/4 -3N_2 k_+} \lesssim \epsilon_1^3 2^{-6k_+ -(3-1000p_0)m/2} \lesssim  \epsilon_0  2^{-6k_+} 2^{-(1+50p_0)m}.
\end{equation}
Note that
\[
\big| \widehat{f}_l(\xi+\rho, s) - \widehat{f}_{l}(\xi, s) \big|  \lesssim |\rho|^{\h} \| \p_{\xi} \widehat{f}_l (s)\|_{L^2},   |c^{\ast}(\xi, \eta, \sigma) - c^{\ast}(\xi, 0, 0)|\psi_{\bar{l}}^{\bar{l}}(\eta) \psi_{\bar{l}}^{\bar{l}}(\sigma)\lesssim 
2^{3k/2}2^{\bar{l}}.
\]
From above estimates, the following estimate holds after using estimate (\ref{equation4000}) in Lemma \ref{symbolbound} and the size of support of $\eta$ and $\sigma$,
\begin{equation}\label{equation7022}
|\mathcal{I}_2| \lesssim 2^{-k_{-}/5}2^{5\bar{l}/2 -2N_2 k_+ +k} \epsilon_1^3 + 2^{3k/2 -9k/4+3\bar{l} -3N_2 k_+}\epsilon_1^3 \lesssim \epsilon_1^3 2^{-6k_+ -(1+50p_0)m}.
\end{equation}
We proceed to estimate $\mathcal{I}_3$. Note the fact (see also \cite{IP1}) that
\begin{equation}\label{equation178}
\int_{\R\times\R} e^{ -i xy}\tilde{\psi}(x/N) \tilde{\psi}(y/N) d x dy = 2\pi + \mathcal{O}(N^{-1/2}).
\end{equation}
Through  scaling,  the following estimate holds from (\ref{equation178}),
\begin{equation}
\Big| \int_{\R^2} e^{-i s \eta \sigma/(4|\xi|^{3/2})}\psi_{\bar{l}}^{\bar{l}}(\eta)\psi_{\bar{l}}^{\bar{l}}(\sigma)d\eta d\sigma  - \frac{4|\xi|^{3/2}}{s}(2\pi) \Big|\lesssim 2^{-(1+25p_0)m}2^{3k/2},
\end{equation}
which further gives us
\begin{equation}\label{equation7023}
|\mathcal{I}_3| \lesssim \epsilon_1^3 2^{7k/4-3N_2 k_+} 2^{-(1+25p_0)m} \lesssim \epsilon_1^3 2^{-6k_+} 2^{-(1+25p_0)m}.
\end{equation}
To sum up, from (\ref{equation7021}), (\ref{equation7022}) and (\ref{equation7023}), the following estimate holds, 
\[
(\ref{equation7020}) \lesssim \epsilon_1^3 2^{-(1+25p_0)m}2^{ -6 k_+}.
\]
Therefore finishing the proof.
\end{proof}

Now, it remains to consider the case when $k_i \notin[k-10, k+10]$ for some $i\in\{1,2,3\}$.  Therefore the term $\widehat{f_{k_1}}(t,\xi)\widehat{\overline{f}}_{k_3} (t, -\xi) \widehat{f_{k_2}}(t,\xi)\psi_k(\xi)$ vanishes. From now on, we can drop this term inside $J^{+,+,-}_{k_1,k_2,k_3} $.

\begin{lemma}
For any $(\iota_{1}, \iota_{2}, \iota_{3})\in \mathcal{S}$, we have the following rough estimates:
\begin{equation}\label{eqn115}
|I_{k_1, k_2, k_3}^{\iota_1, \iota_2, \iota_3}(t,\xi)| \lesssim 2^{\min\{k_i\}/4} 2^{3\med\{k_i\}/4-N_2\med\{k_i\}_+} 2^{5\max\{k_i\}/4 - N_2\max\{k_i\}_+} \epsilon_1^3,
\end{equation}
\begin{equation}\label{eqn116}
|I_{k_1, k_2, k_3}^{\iota_1, \iota_2, \iota_3}(t,\xi)| \lesssim 2^{\med\{k_i\}/2} 2^{2\max\{k_i\}}\| e^{-it \Lambda} P_{\min\{k_{i}\}} f\|_{L^\infty} \|P_{\med\{k_{i}\}} f\|_{L^2} \| P_{\max\{k_{i}\}}f\|_{L^2}.
\end{equation}
\end{lemma}
\begin{proof}
 From (\ref{equation4000}) in Lemma \ref{symbolbound}, the desired estimate  (\ref{eqn115}) follows straightforwardly after putting all inputs into the $Z$-normed space. From (\ref{equation4000}) in Lemma \ref{symbolbound} and $L^2-L^\infty-L^\infty$ type trilinear estimate (\ref{trilinearestimate}) in Lemma \ref{boundness}, (\ref{eqn116}) holds straightforwardly. 
\end{proof}

Recall that $ 2^{- 21 p_0 m }\leq |\xi|\leq 2^{20p_0m}$.  From (\ref{eqn115}) and (\ref{eqn116}), the following estimate holds when $\med\{k_{i}\} \leq -(1+100p_0) m$ or $\min\{k_{i}\} \leq - 4(1+100p_0)m$ or  $\max\{k_i\} \geq (1+100p_0)m/5$, 
\begin{equation}\label{equation9901}
\sup_{t\in[2^{m-1}, 2^{m+1}]}\big||\xi|^{3/4-p_0}(1+|\xi|^{N_2+2p_0}) I^{\iota_{1}, \iota_{2}, \iota_{3}}_{k_{1}, k_{2}, k_{3}}(\xi, t)\big|  \lesssim \epsilon_1^3 2^{-(1+p_0)m},
\end{equation}
which is sufficient for most of cases in our desired estimate (\ref{equation4200}). Therefore, it is sufficient to consider fixed $k_1,k_2,$ and $k_3$  in the following range:
\begin{equation}\label{equation194}
-(1+100p_0)m \leq\med\{k_i\}\leq \max\{k_i\} \leq (1+100p_0)m/5, \min\{k_{i}\} \geq - 4(1+100p_0)m.
\end{equation}

\begin{lemma}\label{znorm2}
If $k_{1}, k_{2}, k_{3}$ satisfies the estimate \textup{(\ref{equation194})} and  one of the following two conditions is satisfied:
\[
\max\{ |k_{1}-k|, |k_{2}-k|, |k_{3}-k|\}\geq 20, \quad \med\{k_i\} -\min\{k_i\} \leq 10,
\]
\[
\max\{ |k_{1}-k|, |k_{2}-k|, |k_{3}-k|\}\geq 20, \quad \min\{k_i\} \geq -(5/7-1000p_0) m,\]
then the following estimate holds,
\begin{equation}\label{equation7031}
\sup_{t\in[2^{m-1}, 2^{m+1}]} \big||\xi|^{3/4-p_0}(1+|\xi|)^{N_2+2p_0}I^{+,+,-}_{k_{1}, k_{2}, k_{3}}(t,\xi )  \big| 
\lesssim \epsilon_{1}^{3} 2^{ - (1+p_{0})m}.
\end{equation}
\end{lemma}
\begin{proof}
\noindent\textbf{Case 1:\quad} We  first consider the case when $\max\{ |k_{1}-k_{3}|, |k_{2}-k_{3}|\} \geq 5$ and $\min \{k_i\} \geq -(5/7-1000p_0) m$. Note that
\[
\p_{\eta}\Phi = \Lambda'(\xi+\eta+\sigma) - \Lambda'(\xi+\eta),\quad \p_{\sigma}\Phi(\xi, \eta, \sigma) = \Lambda'(\xi+\eta+\sigma) -\Lambda'(\xi+\sigma).
\]
From symmetry and without loss of generality, we   assume that $|k_1 - k_3| \geq 5$. Hence 
\begin{equation}\label{equation7051}
|\p_\eta \Phi(\xi, \eta, \sigma)| = \big| \Lambda'(\xi+\eta)- \Lambda'(\xi+\eta+\sigma) \big| \gtrsim 2^{-\min\{k_1,k_3\}/2}.
\end{equation}
After integration by parts in $\eta$, we can derive the following estimate:
\begin{equation}\label{equation7034}
|I_{k_1, k_2, k_3}^{+,+,-}(t,\xi)| \lesssim   \big( |\mathcal{J}_{k_1, k_2, k_3}^{1} |+|\mathcal{J}_{k_1, k_2, k_3}^{2} |\big),
\end{equation}
\[
\mathcal{J}_{k_1, k_2, k_3}^1 := \frac{1}{t} \int_{\R^2} e^{i t\Phi(\xi, \eta, \sigma)} \widehat{f^{+}_{k_1}}(t,\xi+\eta) \widehat{f^{+}_{k_2}}(t,\xi+\sigma) \widehat{f^{-}_{k_3}}(t,-\xi-\eta-\sigma) \p_{\eta}r_{2}(\xi,\eta, \sigma) d\eta d\sigma ,
\]
\[
\mathcal{J}_{k_1, k_2, k_3}^2 := \frac{1}{t}  \int_{\R^2} e^{it\Phi(\xi, \eta, \sigma)}  \p_{\eta} \big(\widehat{f^{+}_{k_1}}(t,\xi+\eta)  \widehat{f^{-}_{k_3}}(t,-\xi-\eta-\sigma)\big)\widehat{f^{+}_{k_2}}(t,\xi+\sigma) r_{2}(\xi,\eta, \sigma) d\eta d\sigma ,
\]
where
\[
r_2(\xi,\eta, \sigma) = \displaystyle{\frac{c^{\ast}(\xi, \eta, \sigma)\psi_{k_1}(\xi+\eta)\psi_{k_2}(\xi+\sigma)\psi_{k_3}(\xi+\eta+\sigma)}{  \p_\eta\Phi(\xi, \eta, \sigma)}}.
\]
From (\ref{equation7051}), (\ref{equation4000}) in Lemma \ref{symbolbound} and Lemma \ref{Snorm}, the following estimates hold,  
\begin{equation}\label{equation210}
 \|   r_2(\xi, \eta, \sigma)\|_{\mathcal{S}^\infty} \lesssim   2^{ \min\{k_1, k_3\}/2} 2^{\med\{k_1, k_2, k_3\}/2} 2^{2\max\{k_1, k_2, k_3\}},
\end{equation}
\begin{equation}\label{equation211}
 \| \p_\eta r_2(\xi, \eta, \sigma)\|_{\mathcal{S}^\infty} \lesssim   2^{-\min\{k_1, k_3\}/2} 2^{\med\{k_1, k_2, k_3\}/2} 2^{2\max\{k_1, k_2, k_3\}}.
\end{equation}
Therefore, from (\ref{equation210}), (\ref{equation211}) and $L^2-L^2-L^\infty$ type trilinear estimate (\ref{trilinearestimate}) in Lemma \ref{boundness}, the following estimates hold,
\[
|\mathcal{J}_{k_1, k_2, k_3}^1| \lesssim 2^{-\min\{k_i\}/2+\med\{k_i\}/2+2\max\{k_i\}} 2^{-m} \| e^{-is \Lambda}P_{\min} f \|_{L^\infty} \| P_{\med}f\|_{L^2} \| P_{\max}f\|_{L^{2}}
\]
\begin{equation}
  \lesssim  \epsilon_1^3 2^{2\max\{k_i\}} 2^{-(1+100p_0)m} 2^{-4\max\{k_1, k_2, k_3\}_{+}} \lesssim \epsilon_1^3 2^{-(1+100p_0)m},
\end{equation}
\[
|\mathcal{J}_{k_1, k_2, k_3}^2| \lesssim 2^{-m} 2^{\min\{k_1, k_3\}/2} 2^{\med\{k_1, k_2, k_3\}/2} 2^{2\max\{k_1, k_2, k_3\}}\big(
\| \p_{\xi} \widehat{f}(\xi,s) \psi_{k_1}(\xi)\|_{L^2}\]
\[
 \times \|P_{\max\{k_2, k_3\}} f\|_{L^2}\| e^{-is \Lambda} P_{\min\{k_2,k_3\}} f\|_{L^\infty} + \|  \p_{\xi} \widehat{f}(\xi,s) \psi_{k_3}(\xi) \|_{L^2} \| e^{-is\Lambda} P_{\min\{k_1,k_2\}} f\|_{L^\infty} 
\]
\[ 
 \times \|P_{\max\{k_1, k_2\}} f\|_{L^2} \big)\lesssim \epsilon_1^3 2^{-7 \min\{k_{i}\}/10}2^{-3m/2} + \epsilon_1^3 2^{-3m/2} 2^{3\max\{k_i\}_{+}/2}\lesssim \epsilon_1^3 2^{-(1+100p_0)m}.
\]
From above estimates, (\ref{equation7034}) and the fact that $|\xi|\leq 2^{20p_0 m}$, the following estimate holds, 
\begin{equation}\label{equation230}
 \big||\xi|^{3/4-p_0}(1+|\xi|)^{N_2+2p_0}I^{+,+,-}_{k_{1}, k_{2}, k_{3}}(t,\xi )  \big|\lesssim 2^{-(1+100)p_0 m +80 p_0 m} \epsilon_1^3 \lesssim 2^{-(1+p_0)m}\epsilon_1^3.
\end{equation}
Therefore, finishing the proof. We remark that when $|k_2 - k_3|\geq 5$, we can do integration by parts with respect to $\sigma$ and the argument is very similar. 

\noindent\textbf{Case 2:\quad} Now we proceed to consider the case when $\max\{|k_1 - k_3|, |k_2 - k_3|\} \leq 5$ and $\min\geq - (5/7-1000 p_0)m$. Recall that $\max\{|k_1 - k|, |k_2 - k|, |k_3 -k|\} \geq 20$, hence $\min\{k_1, k_2, k_3\} \geq k + 10 $. As a result, the following estimate holds,
\begin{equation}
|\xi+ \eta|\sim  |\xi+\sigma|\sim |\xi+\sigma+\eta|\sim 2^{k_1}, |\xi|\sim 2^{k} \Longrightarrow |\eta|\sim |\sigma| \sim 2^{k_1},
\end{equation}
\begin{equation}\label{equation7050}
|\p_\eta\Phi(\xi, \eta, \sigma)| = |\Lambda'(\xi+\eta) - \Lambda'(\xi+\eta+\sigma)| \gtrsim 2^{-k_1/2}\sim 2^{-\min\{k_1,k_3\}/2}.
\end{equation}
Note that estimate (\ref{equation7050}) if of same type as (\ref{equation7051}).  Hence the argument used in \textbf{Case 1} can be applied to this case. We can use integration by parts in ``$\eta$" to derive our desired estimate (\ref{equation230}).

\noindent\textbf{Case 3:\quad} For the case when $\med\{k_i\} -\min\{k_i\} \leq 10$, we can use the same method used in \textbf{Case 1} and \textbf{Case 2}. After integration by parts in $\eta$ or $\sigma$, the loss of $2^{-\min\{k_i\}/2 }$ can be covered by $2^{\med\{k_i\}/2 }$ from the symbol. As a result, the following estimate holds,
\[ 
\big| I_{k_1, k_2, k_3}^{+, +,-}(t,\xi)\big| \lesssim \epsilon_1^3 2^{-\med\{k_i\} /5}  2^{-3m/2}+\epsilon_1^3 2^{-3m/2} 2^{3\max\{k_i\}_{+}/2} \lesssim \epsilon_1^{3} 2^{-(21-1000p_0)m/20},
\]
where we used the fact that $\med\{k_{i}\} \geq -(1+100p_0)m$ and $\max\{k_{i}\}\leq (1+100p_0)m/5$, see (\ref{equation194}).

\end{proof}

\begin{lemma}\label{znorm5}
If $k_{1}, k_{2}, k_{3}$ satisfies the estimate \textup{(\ref{equation194})} and the following condition holds,
\[
\max\{ |k_{1}-k|, |k_{2}-k|, |k_{3}-k|\}\geq 20,\,  \min\{k_i\}  < -(5/7-1000p_{0})m, \med\{k_i\}  -\min\{k_i\}  \geq 10, 
\]
then  the following estimate holds,
\begin{equation}\label{equation9920}
 \Big| |\xi|^{3/4-2p_0} (1+|\xi|^{N_2+2p_0})\int_{t_{1}}^{t_{2}} e^{ i  L(t,\xi)} [I^{+,+,-}_{k_1, k_2, k_3}(t,\xi)  ] d\,t \Big| \lesssim 2^{-p_{0}m}\epsilon_{1}^{3}.
\end{equation}
\end{lemma}
\begin{proof}
Define $\tilde{k} = \min\{ k, \med\{k_{i}\}\}$.
Recall that $ -21 p_0 m \leq k \leq 20 p_0 m$ and $\min\{k_i\} < -(5/7-1000p_0)m$. Hence $\min\{k_1,k_2, k_3\}$ is much smaller than $k$. As $\med\{k_i\} - \min\{k_i\} \geq 10$, we could see that the smallest number among $k,k_1,k_2,k_3$, which is $\min\{k_i\}$, is much smaller than the second smallest number, which is $\tilde{k}$. As mentioned before,  a very important observation for this case is that the size of phase $\Phi(\xi, \eta, \sigma)$ is determined by the second smallest number instead of the smallest number. More precisely, the following estimate holds, 
\begin{equation}\label{equation5400}
|\Phi(\xi, \eta, \sigma)| \gtrsim 2^{\tilde{k}/2}.
\end{equation}
To improve presentation, we postpone the proof of (\ref{equation5400}) to the end of this subsection. We can take it as granted first.

To prove (\ref{equation9920}), we will integrate by parts in time. 
Hence, after integration by parts in time, we have
\[
 \Big| \int_{t_{1}}^{t_{2}} e^{ i L(\xi, s)} [I^{+,+,-}_{k_1, k_2, k_3;1}(\xi, s)  ] d\,s \Big|  \lesssim |B_{1}( t_{1},\xi)| + |B_{2}( t_{2},\xi)|+ | T_{1}(\xi)|+ |T_2(\xi)|,
\]
where for $i\in \{1,2\},$
\[
B_{i}(t_i,\xi) = \int_{\R^{2}}  { e^{ i t_{i}\Phi(\xi, \eta, \sigma)  + i L(\xi, t_{i})}} \widehat{f_{k_{1}}}(\xi+\eta, t_{i}) \widehat{f_{k_{2}}}(\xi+\sigma, t_{i})\widehat{\overline{f}}_{k_{3}}(-\xi-\eta-\sigma, t_{i})   r_3(\xi, \eta, \sigma) d\eta d\sigma,
\]
\[
T_{1}(\xi) = \int_{t_{1}}^{t_{2}} \int_{\R^{2}}
e^{ i t\Phi(\xi, \eta, \sigma)+i  L(t,\xi)}  \p_t L(t,\xi)
\widehat{f_{k_{1}}}(t,\xi+\eta) \widehat{f_{k_{2}}}(t,\xi+\sigma)\widehat{\overline{f}}_{k_{3}}(t,-\xi-\eta-\sigma)\]
\[\times  r_3(\xi, \eta, \sigma)  d\eta d\sigma d t,\]
\[
T_{2}(\xi) = \int_{t_{1}}^{t_{2}} \int_{\R^{2}}
e^{ i t\Phi(\xi, \eta, \sigma)+i  L(t,\xi)}  \p_t\Big(\widehat{f_{k_{1}}}(t,\xi+\eta) \widehat{f_{k_{2}}}(t,\xi+\sigma)\widehat{\overline{f}}_{k_{3}}(t,-\xi-\eta-\sigma)\Big)\]
\[ \,\, \times  r_3(\xi, \eta, \sigma)  d\eta d\sigma d t,
\,\,
r_3(\xi, \eta, \sigma)= \frac{c^{\ast}(\xi, \eta, \sigma)\psi_{k_1}(\xi+\eta)\psi_{k_2}(\xi+\sigma) \psi_{k_3}(\xi+\eta+\sigma)}{\Phi(\xi, \eta, \sigma)}.
\]
From (\ref{equation5400}), Lemma \ref{Snorm}, and (\ref{equation4000}) in Lemma \ref{symbolbound}, the following estimates hold, 
\begin{equation}\label{equation236}
\| r_3(\xi, \eta, \sigma)\|_{\mathcal{S}^\infty}\lesssim 2^{\med\{k_i\}/2+2\max\{k_i\}-\tilde{k}/2}\lesssim 2^{5\max\{k_i\}/2+11p_0 m }.
\end{equation}
From (\ref{equation236}),  the $L^2-L^2-L^\infty$ type trilinear estimate (\ref{trilinearestimate}) in Lemma \ref{boundness} and the fact that $\med\{k_i\}\geq -(1+100p_0)m$, the following estimate holds, 
\[
 |B_{i}(  t_{i}, \xi)| \lesssim 2^{5\max\{k_i\}/2+11p_0 m } \| e^{-it_i \Lambda} P_{\min\{k_i\}} f\|_{L^\infty} \| P_{\med\{k_i\}} f \|_{L^2} \| P_{\max\{k_i\}} f\|_{L^2}
\]
\begin{equation}
\lesssim  2^{-\med\{k_i\}/5+23\max\{k_i\}/10-4 \max\{k_i\}_{+}-m/2+11p_0 m}\epsilon_1^3\lesssim \epsilon_1^3 2^{-(3-1000p_0)m/10}.
\end{equation}

Recall (\ref{eqn1000}). From (\ref{equation4000}) in Lemma \ref{symbolbound}, the following estimate holds, 
\begin{equation}\label{equation276}
|\p_{t}L(t, \xi)\psi_k(\xi)|\lesssim 2^{-m} |\xi|^{4}\| \widehat{f}(t,\xi)\psi_k(\xi)\|_{L^\infty_\xi}^2 \lesssim 2^{-m + 4k_+}\epsilon_1^2\lesssim 2^{-(1-100p_0)m}\epsilon_0.
\end{equation}
From (\ref{equation276}), (\ref{equation236}), and the $L^2-L^2-L^\infty$ type trilinear estimate (\ref{trilinearestimate}) in Lemma \ref{boundness}, the following estimate holds,
\[
|T_{1}(\xi)| \lesssim \epsilon_0 2^{m-(1-111p_0)m +5\max\{k_i\}/2} \|e^{-it \Lambda} P_{\min} f\|_{L^\infty} \| P_{\med} f \|_{L^2} \| P_{\max} f\|_{L^2}
\]
\begin{equation}
\lesssim \epsilon_1^3 2^{111p_0 m-m/2-\med\{k_i\}/5} \lesssim  2^{-(3-1000p_0)m/10  } \epsilon_1^3.
\end{equation}
Again, here we used the fact that $\med\{k_i\}\geq -(1+100p_0)m$.

To estimate $T_2(\xi)$, we  put the input $\p_t \widehat{f^{\pm}_{k_i}}$ and the input with the smallest frequency in $L^2$ and the other one in $L^\infty$. More precisely, from (\ref{equation236}) and the $L^2-L^2-L^\infty$ type trilinear estimate (\ref{trilinearestimate}) in Lemma \ref{boundness}, the following estimate holds,
\[
|T_2(\xi)| \lesssim \sum_{\{l, m, n\}=\{1,2,3\}}  2^{m +5\max\{k_i\}/2+11p_0 m }\| \p_t \widehat{f}(t,\xi) \psi_{k_l}(\xi) \|_{L^2} \|P_{\max\{k_m, k_n\}} f\|_{L^2}
\]
\begin{equation}
\times \| e^{-it \Lambda} P_{\min\{k_m, k_n\}} f \|_{L^\infty}  \lesssim \epsilon_1^5 2^{-(3-1000p_0)m/10}.
\end{equation}
To sum up, our desire estimate $(\ref{equation9920})$ holds.
\end{proof}

It remains to prove  (\ref{equation4200}) for the case when $(\iota_1, \iota_2, \iota_3)\in \{(+,+,+), (+,-,-),$ $(-,-,-) \}$ and $k_1,k_2$ and $k_3$ satisfy (\ref{equation194}).  An important observation is that we have the following  weak ellipticity  estimate for phases $\Phi^{-,-,-}, \Phi^{+,-,-}, \Phi^{+,+,+}$:
\begin{equation}\label{equation9970}
\sum_{(\iota_1,\iota_2,\iota_3)\in \mathcal{S}, (\iota_1, \iota_2,\iota_3)\neq (+,+,-)}|\Phi^{\iota_{1},\iota_{2},\iota_{3}}(\xi_{1}, \xi_{2}, \xi_{3})\psi_{k_{1}}(\xi_{1})\psi_{k_{2}}(\xi_{2}) \psi_{k_{3}}(\xi_{3})| \gtrsim 2^{\med\{ k_{i}\}/2}.
\end{equation}
Hence, after doing integration by parts in time,  the loss of $2^{-\med\{k_i\}/2}$ can be covered by the size of the symbol, see (\ref{equation194}) in Lemma \ref{symbolbound}. Similar to the argument we used in the proof of Lemma \ref{znorm5}, with minor modifications, one can show the desired estimate (\ref{equation4200}) holds without any problems.

\noindent \textit{Proof of (\ref{equation5400}) and (\ref{equation9970}).}\quad 
Without loss of generality, we assume that $|\xi|\geq |\xi-\eta|\geq |\eta-\sigma|\geq |\sigma|$. As a result, we have $|\eta|\leq 2|\xi-\eta|$ and $\xi\eta>0$, which further gives us $|\eta|\leq 2|\xi|/3$ or $\eta=2\xi$. Recall that $\Lambda(\xi):=|\xi|^{1/2}$. Note that, the following estimate holds,
\begin{equation}\label{equation287}
\big|\Lambda(\xi)-\Lambda(\xi-\eta)\big| = \Big|\frac{|\xi|-|\xi-\eta|}{|\xi|^{1/2}+|\xi-\eta|^{1/2}}\Big| \leq \frac{1}{2}|\eta|^{1/2}.
\end{equation}
When $(\iota_1, \iota_2, \iota_3)\in \{(+,+,+), (+,-,-),$ $(-,-,-) \}$, from (\ref{equation287})   , the following estimate holds,
\[
|\Phi^{\iota_{1},\iota_{2},\iota_{3}}(\xi , \eta, \sigma)\psi_{k_{1}}(\xi-\eta)\psi_{k_{2}}(\eta-\sigma) \psi_{k_{3}}(\sigma)|\]
\[\geq |\Lambda(\eta-\sigma)|+|\Lambda(\sigma)|- |\Lambda(\xi)-\Lambda(\xi-\eta)|\gtrsim \max\{|\eta-\sigma|^{1/2},|\sigma|^{1/2}\}.
\]
Therefore desired estimate (\ref{equation9970}) holds. To prove (\ref{equation5400}), we have another assumption, which is $ |\sigma|\leq 2^{-10}|\eta-\sigma|.$ As a result, we have $\Lambda(|\eta-\sigma|)-\Lambda(|\sigma|) \geq (1-2^{-5})|\eta-\sigma|^{1/2}$ $\geq (1-2^{-5})^2 |\eta|^{1/2}$. Hence, from (\ref{equation287}), the following estimate holds,
\[
|\Phi^{+,+,-}(\xi , \eta, \sigma)\psi_{k_{1}}(\xi-\eta)\psi_{k_{2}}(\eta-\sigma) \psi_{k_{3}}(\sigma)|\geq |\Lambda(\eta-\sigma)|-|\Lambda(\sigma)|\]
\[- |\Lambda(\xi)-\Lambda(\xi-\eta)|\gtrsim |\eta-\sigma|^{1/2}\sim \max\{|\eta-\sigma|^{1/2},|\sigma|^{1/2}\}.
\]
Therefore, finishing the proof of (\ref{equation5400}).
\qed

\subsection{$Z$-norm estimate of the remainder term}
\begin{lemma}\label{znormremainder}
Under the bootstrap assumptions \textup{(\ref{assumption})} and  \textup{(\ref{equation140})}, the following estimate holds for $k\in[-21 p_0 m, 20p_0 m]$,
\begin{equation}\label{equation700001}
  \sup_{t\in[2^{m-1},2^{m+1}]} \| \int_{t_1}^{t_2} e^{ i L(t,\xi) + i t \Lambda(\xi)}\widehat{R}(t,\xi)\psi_k(\xi) d t\|_Z\lesssim 2^{ -(1+ p_{0} )m}\epsilon_{1}^{3}.
\end{equation} 

\end{lemma}
\begin{proof}
From (\ref{equation141}) in  Lemma \ref{znorm6}, (\ref{eqn433})  in Lemma \ref{estimatesofremainder}, and (\ref{vectorfieldestimate3}) in Lemma \ref{finalestimate}, we have the following estimate,
\[
\sup_{t_1, t_2\in[2^{m-1}, 2^{m+1}]} \|\int_{t_1}^{t_2} e^{ i L(t,\xi) + i t \Lambda(\xi)}\widehat{R}(t,\xi)\psi_k(\xi) d t\|_{L^\infty_\xi}   \]
\[\lesssim \sup_{t\in[2^{m-1}, 2^{m+1}]} 2^{-k/2}\| \int_{t_1}^{t_2} e^{ i L(t,\xi) + i t \Lambda(\xi)}\widehat{R}(t,\xi)\psi_k(\xi) d t\|_{L^2}^{\h}\]
\[ 
\times \Big[ \|\int_{t_1}^{t_2}  t\p_t\big(e^{ i L(t,\xi) + i t \Lambda(\xi)}\widehat{R}(t,\xi)\big)\psi_k(\xi) d t\|_{L^2} + 2^{m}\big(\| S P_{k}[R(t)]\|_{L^2} + \| P_k {R}(t)\|_{L^2}\big)  \Big]^{\h}
\]
\begin{equation}\label{equation99994}
\lesssim 2^{- m/2+220 p_0m} \epsilon_0^2+ 2^{ -m/4+20 p_0 m}\|  \int_{t_1}^{t_2}  t\p_t\big(e^{ i L(t,\xi) + i t \Lambda(\xi)}\widehat{R}(t,\xi)\big)\psi_k(\xi) d t\|_{L^2}^{1/2}.
\end{equation}
Note that, in above estimate, we used the following fact,
\[
\sup_{t\in[2^{m-1},2^{m+1}]}\| (t\p_t -2\xi\cdot \p_\xi) e^{i L(t, \xi) }\|_{L^\infty_\xi}\lesssim 2^{p_0m} \epsilon_1^2,
\]
which follows from the definition of $L(t,\xi)$ in (\ref{eqn1000}).

After integration by parts in time, from (\ref{eqn433}) in Lemma \ref{estimatesofremainder}, it is easy to see the following estimate holds,
\[
\|  \int_{t_1}^{t_2}  t\p_t\big(e^{ i L(t,\xi) + i t \Lambda(\xi)}\widehat{R}(t,\xi)\big)\psi_k(\xi) d t\|_{L^2} \lesssim \sup_{t\in[2^{m-1},2^{m+1}]} 2^{m} \| P_k(R)(t)\|_{L^2}
\] 
\begin{equation}\label{equation99995}
\lesssim 2^{-m/2+200 p_0 m}\epsilon_0^2.
\end{equation}
Combining (\ref{equation99994}) and (\ref{equation99995}), it is easy to see our desired estimate (\ref{equation700001}) holds. Hence finishing the proof.
\end{proof}

\appendix

\section{Remainder Estimates}\label{remainderestimate}

\subsection{Explicit   formulas   of   remainder terms}\label{expliciteformulas}
In this subsection, we give all detail formulas for all good remainder terms that are postponed in section \ref{energyestimate} and section \ref{improveddispersion}.

The detail formulas of $\mathcal{R}_1$ and $\mathcal{R}_2$ in the system  (\ref{equation6}) are given as follows,

\begin{equation}\label{equation351}
\mathcal{R}_1 = \d^{1/2} T_{\alpha} U^2- T_{\alpha}\d^{1/2} U^2 -T_{\sqrt{a}}T_V \p_x h + T_{V}\p_x T_{\sqrt{a}} h +T_{\sqrt{a}} F(h)\psi,
\end{equation}
\[ 
\mathcal{R}_2= \d^{1/2}(T_{a} h) -T_{\sqrt{a}} \d^{1/2} T_{\sqrt{a}} h -\d^{1/2}T_{V}\p_x \omega +T_{V}\p_x \d^{1/2} \omega  
\]
\begin{equation}\label{equation352}
-T_{V\p_x h} B +T_{V}T_{\p_x h} B -T_{V}T_{\p_x B} h  +   T_{V\p_x B} h  - \h R_{\mathcal{B}}(V,V) - R_{\mathcal{B}}(V, \p_x h ) - \h R_{\mathcal{B}}(B,B),
\end{equation}
 where bilinear operator $R_{\mathcal{B}}(\cdot, \cdot)$ is defined in (\ref{equation350}), the Taylor coefficient $a$ is defined in (\ref{eqn1004}), $\alpha =\sqrt{a}-1$, and $\omega$ is the good unknown variable, which is defined in (\ref{remainderpara}).

The detailed formulas of $\widetilde{\mathcal{R}}_1$ and $\widetilde{\mathcal{R}}_2$ inside the system (\ref{eqn120000}) are given as follows,
\[
\widetilde{\mathcal{R}}_1 =  \widetilde{Q}_1(\widetilde{A}_1(U^1,U^1)+ \widetilde{A}_2(U^2, U^2), U^2) + \widetilde{Q}_1(U^1, \widetilde{B}(U^1, U^2)) 
\]
\[
 +  2 \widetilde{A}_1(\mathcal{R}_1 , U^1) + 2 \widetilde{A}_2(\mathcal{R}_2 , U^2) - \Lambda_{2}[\mathcal{R}_1](U^1, \widetilde{B}(U^1, U^2))
\]
\begin{equation}\label{eqn425}
- \Lambda_{2}[\mathcal{R}_1](\widetilde{A}_1(U^1,U^1)+\widetilde{A}_2(U^2, U^2), U^2)  + \Lambda_{\geq 3}[\mathcal{R}_1],
\end{equation}
\[\widetilde{\mathcal{R}}_2 =  \widetilde{Q}_2(\widetilde{A}_1(U^1,U^1)+\widetilde{A}_2(U^2, U^2), U^1) + \widetilde{Q}_2( U^1, \widetilde{A}_1(U^1,U^1)+ \widetilde{A}_2(U^2, U^2))
\]
\[
 + \widetilde{Q}_3( U^2, \widetilde{B}(U^1, U^2)) + \widetilde{Q}_3(  \widetilde{B}(U^1, U^2), U^2)+ \widetilde{B}(\mathcal{R}_1 , U^2) + \widetilde{B}(U^1, \mathcal{R}_2)+ \Lambda_{\geq 3}[\mathcal{R}_2]
\]
\[
-\Lambda_{2}[\mathcal{R}_2](\widetilde{A}_1(U^1,U^1)+ \widetilde{A}_2(U^2, U^2), U^1) -\Lambda_{2}[\mathcal{R}_2](U^1, \widetilde{A}_1(U^1,U^1)+ \widetilde{A}_2(U^2, U^2)) 
\]
\begin{equation}\label{eqn426}
-\Lambda_{2}[\mathcal{R}_2](\widetilde{B}(U^1, U^2), U^2) -\Lambda_{2}[\mathcal{R}_2](U^2, \widetilde{B}(U^1, U^2)).
\end{equation}

The detailed formula of $\widetilde{\mathfrak{R}}_1$ and $\widetilde{\mathfrak{R}}_2$ in (\ref{eqn53000}) and (\ref{eqn53001}) are given as follows, 
\[
\widetilde{\mathfrak{R}}_1= \widetilde{C}_{1,\tau}(\mathcal{R}_1^S, U_2) + \widetilde{C}_{2,\tau}(\mathcal{R}_2^S, U^1) + \widetilde{C}_{1,\tau}(SU^1, \Lambda_{\geq 2}[\p_t U^2])+ \widetilde{C}_{2,\tau}(SU^2, \Lambda_{\geq 2}[\p_t U^1 ])
\]
\begin{equation}\label{equation99980}
 + \Lambda_{\geq 3}[\mathcal{R}_1^S] - \widetilde{\mathfrak{Q}}_1^S(W_{1,\tau}-SU^1, SU^2) - \widetilde{\mathfrak{Q}}_1^S(SU^1, W_{2,\tau}-SU^2),
\end{equation}
\[
\widetilde{\mathfrak{R}}_2= \widetilde{C}_{3,\tau}(\mathcal{R}_1^S, U_1) + \widetilde{C}_{4,\tau}(\mathcal{R}_2^S, U^2) + \widetilde{C}_{3,\tau}(SU^1, \Lambda_{\geq 2}[\p_t U^1])+ \widetilde{C}_{4,\tau}(SU^2, \Lambda_{\geq 2}[\p_t U^2 ])
\]
\begin{equation}\label{equation99981}
 + \Lambda_{\geq 3}[\mathcal{R}_2^S] - \widetilde{\mathfrak{Q}}_3^S(W_{1,\tau}-SU^1, SU^2) -\widetilde{\mathfrak{Q}}_3^S(SU^1, W_{2,\tau}-SU^2),
\end{equation}
where the detail formulas of $\mathcal{R}_1^S$ and $\mathcal{R}_2^S$ are given in (\ref{equation968}) and (\ref{equation969}).

At last, we give the detail formula for the remainder term $R(t)$ in (\ref{eqn610}). Recall (\ref{mainequation}) and the fact that $R(t)$ is quintic and higher. The following identity holds,
\[
R(t) = \Lambda_{\geq 4}[C_1 + \mathcal{R}_1] + i \Lambda_{\geq 4}\big[C_2 + \mathcal{R}_2\big] + 2 A_1(C_1+\Lambda_{\geq 3}[\mathcal{R}_1], U^1) \]
\begin{equation}\label{equation929}
+ 2 A_2(C_2+\Lambda_{\geq 3}[\mathcal{R}_2], U^2) + i B(C_1+\Lambda_{\geq 3}[\mathcal{R}_1], U^2) +  i B(U^1, C_2+\Lambda_{\geq 3}[\mathcal{R}_2]).
\end{equation}

\subsection{$L^2$ type estimate of remainder terms}

From the detailed formulas of remainder terms in subsection \ref{expliciteformulas}, it is easy to see that it would be sufficient to successfully estimate the $L^2$ type norms of remainder terms if we have necessary ingredients, which are the following, (i) $L^2$ and $L^\infty$ type estimates of the nontrivial part of Dirichlet-Neumann operator $B(h)\psi$ and $SB(h)\psi$, (ii) $L^2$ type estimate of the remainder term of paralinearization $F(h)\psi$ and $S F(h)\psi$, (iii) $L^2$ and $L^\infty$ type estimates of the Taylor coefficient $a$, $S a$, $\p_t a $, and $S \p_t a$. Those necessary ingredients will be discussed in details in Appendix \ref{appdxb}
 and Appendix \ref{appendxc}.

\begin{lemma}\label{estimatesofremainder}
 Under the  bootstrap assumption \textup{( \ref{assumption})}, we have following estimates for the remainder terms,
 \begin{equation}\label{eqn432}
\sup_{t\in[0,T]} (1+t)^{1-p_0}\big( \| \Lambda_{\geq 3}[\mathcal{R}_1] \|_{H^{N_0,p}} + \| \Lambda_{\geq 3}[\mathcal{R}_2]\|_{H^{N_0,p}} + \| \widetilde{\mathcal{R}}_1 \|_{H^{N_0,p}} + \| \widetilde{\mathcal{R}}_2 \|_{H^{N_0,p}}\big) \lesssim  \epsilon_0^2
\end{equation}
\begin{equation}\label{eqn433}
\sup_{t\in[2^{m-1},2^m]}  2^{3m/2-200p_0 m} \|  R(t)  \|_{L^2}   \lesssim \epsilon_0^2.
\end{equation}
 \end{lemma}
\begin{proof}
Recall the detailed formulas of $\mathcal{R}_1$ and $\mathcal{R}_2$ in(\ref{equation351}) and (\ref{equation352}). From (\ref{equation635}) in Lemma \ref{paralinearizationDN1}, estimates in Lemma \ref{alphaestimate} and estimates in Lemma \ref{L2Dirichlet}, the following estimate holds for $i\in \{2, 3\}$, 
\[
\| \Lambda_{\geq i}[\mathcal{R}_1]\|_{H^{N_0, p}} + \| \Lambda_{\geq i}[\mathcal{R}_2]\|_{H^{N_0, p}} \lesssim \| \Lambda_{\geq i}[F(h)\psi]\|_{H^{N_0}}+ \| (B,\p_x h , V)\|_{H^{N_0-1}}\| (V, \p_x h, B)\|_{\widetilde{W}^1}^{i-1} \]
\[ + \|  h\|_{H^{N_0, p}}\big( \| \Lambda_{\geq i-1} \p_t \alpha \|_{\widetilde{W}^0}  + \| \alpha\|_{\widetilde{W}^1}\| V\|_{\widetilde{W}^0}\big)
+ \| \Lambda_{\geq 2 }[(B, V)]\|_{H^{N_0-1}} \| (B,V)\|_{\widetilde{W}^1} + \| h\|_{H^{N_0,p}} \| \alpha \|_{\widetilde{W}^0}^2
\]
\begin{equation}\label{equation924}
\lesssim \big( \| h \|_{H^{N_0, p}}  +\|\p_x \psi\|_{H^{N_0-1}} \big) \|(\p_x h, \p_x \psi)\|_{W^1}^{i-1}\lesssim (1+t)^{-1+2p_0}\epsilon_0^2.
\end{equation}
From (\ref{equation26}), (\ref{equation8910}), Lemma \ref{boundness}, (\ref{equation924}), and estimates in Lemma \ref{L2Dirichlet}, the following estimate holds, 
\[
\| \Lambda_{\geq 3}[\widetilde{\mathcal{R}}_1]\|_{H^{N_0, p}} + \| \Lambda_{\geq 3}[\widetilde{\mathcal{R}}_2]\|_{H^{N_0, p}} \lesssim \| (U^1, U^2)\|_{H^{N_0, p}} \| (U^1, U^2)\|_{W^3}^2  \]
\[+ \| (U^1, U^2)\|_{H^{N_0, p}} \| (\mathcal{R}_1, \mathcal{R}_2) \|_{\widetilde{W}^1}+ \| (U^1, U^2)\|_{W^1} \| (\mathcal{R}_1, \mathcal{R}_2) \|_{H^{N_0, p}}
\]
\[
\lesssim \big( \| (U^1, U^2, h)\|_{H^{N_0, p} 
} + \| \d\psi\|_{H^{N_0-1}}\big(\| (\p_x h , \p_x \psi)\|_{W^2} + \| (U^1, U^2)\|_{W^3}
\big) \]
\[\lesssim (1+t)^{-1+2p_0}\epsilon_0^2.
\]
In above estimate, we used the following rough estimate on the good error terms
\[
\| \mathcal{R}_1\|_{\widetilde{W}^1} + \| \mathcal{R}_2\|_{\widetilde{W}^1}\lesssim \| (\p_x h , \p_x \psi)\|_{W^2}^2,
\]
which is derived from estimating each term inside  (\ref{equation351}) and (\ref{equation352}).

To prove (\ref{eqn433}), we only need to estimate the $\dot{H}^p$ norm of each  term inside (\ref{eqn433}), as we do not worry about losing derivative for the desired estimate (\ref{eqn433}). Recall   (\ref{equation351}),  (\ref{equation352}), and (\ref{equation929}).  From (\ref{eqn1310}) in Lemma \ref{lemma2}, (\ref{equation6004}) and (\ref{eqn1674}) in Lemma \ref{alphaestimate},  our desired estimate (\ref{eqn433}) follows straightforwardly as it is quintic and higher.

\end{proof}

\begin{lemma}\label{finalestimate}
Under the bootstrap assumption (\textup{\ref{assumption}}), the following estimates hold on the remainder terms,
\begin{equation}\label{vectorfieldestimate2}
 \sup_{t\in[0,T]} (1+t)^{1-p_0}\big( \| \Lambda_{\geq 3}[\mathcal{R}_1^S]\|_{H^{N_1,p}} + \| \Lambda_{\geq 3}[ \mathcal{R}_2^S ]\|_{H^{N_1,p}} + \|\widetilde{\mathfrak{R}}_1 \|_{H^{N_1,p}} + \|\widetilde{\mathfrak{R}}_2 \|_{H^{N_1,p}}\big) \lesssim \epsilon_0^2.
  \end{equation}
\begin{equation}\label{vectorfieldestimate3}
\sup_{t\in[2^{m-1},2^m]} \sup_{k\in [2^{-100 p_0 m}, 2^{100p_0 m }]} 2^{3m/2-200p_0 m}2^{p k}\| P_k(SR(t)) \|_{L^2} \lesssim \epsilon_0^2.
\end{equation}
\end{lemma}
 
\begin{proof}

Recall (\ref{equation968}) and (\ref{equation969}). From  we can derive the following estimate for $i\in\{2,3\}$,
\[
\| \Lambda_{\geq i}[\mathcal{R}_1^S]\|_{H^{N_1,p}}  + \| \Lambda_{\geq i }[\mathcal{R}_2^S]\|_{H^{N_1, p}} \lesssim \| \Lambda_{\geq i}[ S F(\eta)\psi]\|_{H^{N_1,p}} + \| S\alpha \|_{H^{-1}}\| (\p_x h , \p_x\psi)\|_{W^1}^2  \]
\[+ \| S h \|_{H^{N_1, p}} \big( \| \Lambda_{\geq i -1}[\p_t \alpha]\|_{\widetilde{W}^0} + \| \p_x\alpha\|_{\widetilde{W}^0}\|V\|_{\widetilde{W}^1} \big) 
  + \| \Lambda_{\geq i-1} S\p_t \alpha\|_{H^{-2}} \| h\|_{W^{2+N_1}} 
\]
\[
+ \sum_{i_1+i_2=i , i_1,i_2\in \mathbb{N}_{+}} \| \Lambda_{\geq i_1}[(SB, SV )]\|_{L^2} \| \Lambda_{\geq i_2}[(B, V)]\|_{\widetilde{W}^{2}}
\]
\begin{equation}\label{equation9668}
+ \| (SB, SV,S\p_x h)\|_{L^2}\| \p_x h\|_{\widetilde{W}^1}\|(B, V)\|_{\widetilde{W}^1} + \| (\p_x h, \p_x \psi)\|_{H^{N_1+3}}\|(\p_x h , \p_x \psi)\|_{W^2}^2.
\end{equation}
From (\ref{equation6004}), (\ref{equation396}), (\ref{equation9669}) and (\ref{eqn1674}) in Lemma \ref{alphaestimate}, (\ref{scalingparaerror}) in Lemma \ref{vectorfield}, (\ref{equation499}) in Lemma \ref{equation8024}, we have the following estimate, 
\begin{equation}\label{equation9475}
\textup{(\ref{equation9668})} \lesssim \big(\| Sh\|_{H^{N_1, p}}+\|\d S\psi \|_{L^2}\big)\|(\p_x h , \p_x \psi)\|_{W^2}^{i-1}\lesssim (1+t)^{-(i-1)/2+p_0}\epsilon_0^2.
\end{equation}
Recall (\ref{equation99980}) and (\ref{equation99981}). From (\ref{equation477}) in Lemma \ref{sizeinfolemma}, (\ref{equation9653}), (\ref{equation9652}), and (\ref{equation9475}), the following estimate holds, 
\[
 \|\widetilde{\mathfrak{R}}_1 \|_{H^{N_1,p}} + \|\widetilde{\mathfrak{R}}_2 \|_{H^{N_1,p}}\lesssim (1+t)^{-1+p_0} + (1+t)^{-1/2+p_0} \|(U^1, U^2)\|_{W^3} \]
 \[+ \|(SU^1, SU^2)\|_{H^{N_1, p}} \big(\|\Lambda_{\geq 2}[(\p_t U^1, \p_t U^2)]\|_{\widetilde{2}} +\|(U^1,U^2)\|_{W^3}^2 \big)
\]
\[
 + \|(U^1, U^2)\|_{H^{N_0, p}}\|(U^1, U^2)\|_{W^3}^2 \lesssim (1+t)^{-1+p_0}\epsilon_0.
\]
Note that we do not worry about losing derivative for the desired estimate (\ref{vectorfieldestimate3}).    From (\ref{eqn1310}) in Lemma \ref{lemma2}, (\ref{equation6004}) and (\ref{eqn1674}) in Lemma \ref{alphaestimate}, (\ref{equation9745}) and (\ref{equation499}) in Lemma \ref{propositionscaling} our desired estimate (\ref{vectorfieldestimate3}) follows straightforwardly as it is quintic and higher.
\end{proof}

 \section{Properties of the Dirichlet-Neumann operator}\label{appdxb}

 \subsection{A fixed point type formulation of the Dirichlet-Neumann operator}

  The idea of analyzing the Dirichlet-Neumann operator is based on the work of Alazard-Delort \cite{alazard}. A key observation is that there exists a fixed point type structure inside the velocity potential, which provides a good way to estimate the Dirichlet-Neumann operator in the small data regime. From the bootstrap assumption (\ref{assumption}), the following estimate holds, 
  \begin{equation}\label{eqn13005}
\sup_{t\in[0,T]} \| h(t) \|_{W^{N_2}} +  \| h\|_{H^{N_0,p}}^{1/2} \| h \|_{W^{N_2}}^{1/2} \lesssim \epsilon_1.
  \end{equation}

 We transform the domain $\Omega(t):=\{(x, y): y \leq h(t,x), x\in \mathbb{R}\}$ to the negative half-space through the  change of coordinates $(x,y)\longmapsto (x,z):=(x, y-h(t,x))$, $z\in(-\infty, 0]$. Define $\varphi(x,z)=\phi(x, z+h(t,x))$. Hence $\phi(x,y) = \varphi(x,y-h(t,x))$. It is easy to  verify the following identities hold,
\begin{equation}\label{equation1000000}
 \p_y^2 \phi(x,z+h(t,x))=  \p_z^2 \varphi(x,z),\quad  \p_x^2 \phi  =\p_x^2 \varphi-2\p_x \p_z \varphi \p_x h - \p_z \varphi \p_x^2 h + \p_z^2\varphi  (\p_x h)^2.
\end{equation}
Recall (\ref{harmonicequation}). We know that $\phi$ satisfies a harmonic equation inside $\Omega(t)$.  Therefore, from (\ref{harmonicequation}) and (\ref{equation1000000}), the following equation holds,
\begin{equation}\label{eqn1320}
P \varphi:= [(1+(\p_x h)^2) \p_z^2 + \p_x^2 - 2\p_x h\p_x \p_z - \p_x^2 h \p_z]\varphi=0, \quad \varphi\big|_{z=0}=\psi.
\end{equation}
Recall (\ref{eqn1003}). In $(x,z)$ coordinates system, we have
\begin{equation}
B(h)\psi = \p_z \varphi \big|_{z=0},\quad G(h)\psi=  (1+(\p_x h)^2)\p_z \varphi\big|_{z=0} -\p_x h\p_x \psi.
\end{equation}
From above equation, one can see that the only nontrivial term inside $G(h)\psi$ is $B(h)\psi$. Therefore, to estimate the Dirichlet-Neumann operator in a $X$-normed space, it is sufficient to estimate  $\p_z \varphi(z)$ in the $L^\infty_z X$-normed space.

\begin{lemma} 
Let $\psi$ be in the space $\dot{H}^{1/2+p}\cap \mathcal{C}_0$ 
and $h(x)\in W^{\gamma}(\R)$ with 
$\gamma > 2$, then $\varphi(z,x)$ 
 satisfies the following fixed-point type formulation:
\[
\varphi(z ) = e^{z \d}\psi  + \frac{1}{2} \int_{-\infty}^{0} e^{(z+z')\d} 
[ \p_{x}\d^{-1} (\p_x h \p_{z}\varphi ) - \p_x h \p_{x}\varphi- (\p_x h )^2 \p_{z}\varphi ]d z' + 
\]
\begin{equation}\label{equation5600}
\frac{1}{2}  \int_{-\infty}^{0} e^{- |z-z'|\d} 
[-\p_{x}\d^{-1} (\p_x h\p_{z}\varphi) + \textup{sign}(z- z') (\p_x h\p_{x}\varphi - (\p_x h)^2\p_{z}\varphi )] d z',
\end{equation}
if $\varphi(\cdot, \cdot)$ satisfies $P \varphi =   0$ and $ \varphi |_{z=0}=\psi$. 
\end{lemma}
\begin{proof}
 
One can see \cite{alazard}[Lemma 1.1.5] for the detailed proof. We only give a sketch proof here. 

We rewrite (\ref{eqn1320}) as follows,
\begin{equation}\label{equation101}
(\p_z+\d)(\p_z - \d) \varphi = g(z)=\p_z g_1(z) + \p_x g_2(z),
\end{equation}
where
\[
g_1(z)= -(\p_x h)^2 \p_z \varphi +    \p_x h \p_x \varphi, \quad g_2(z)= \p_x  h \p_z \varphi.
\]
By treating the nonlinearity $g(z)$ as given, we can solve $\varphi$ explicitly from (\ref{equation101}). Note that $g(z)$ has term of type $\p_z^2 \varphi$, which is not in terms of $\nabla_{x,z}\varphi$. Hence we first decompose it as $\p_z g_1(z) +\p_x g_2(z)$, where $g_1(z)$ and $g_2(z)$    linearly depend  on $\nabla_{x,z}\varphi$. As a result, after doing integration by parts in $z$ once for $\p_z g_1(z)$, we can derive (\ref{equation5600}).
\end{proof}
 
 From (\ref{equation5600}), we can derive the following fixed-point formulation for $\nabla_{x,z}\varphi$:
\[
\nabla_{x, z} \varphi(z) = e^{z\d} [\p_{x} \psi, |D_{x}| \psi ]
\]
\begin{equation}\label{equation5700}
 + \int_{-\infty}^{0} K(z, z') M(\p_x h)  \nabla_{x, z}\varphi(z' ) d z' + [0, \p_x h \p_{x}\varphi - (\p_x h)^{2} \p_{z}\varphi ],
\end{equation}
where  $K(z, z')$ and $M(\eta')$ are the matrices of operators:
\begin{equation}\label{eqn164}
K(z, z') = \frac{1}{2} e^{(z+z')\d} \Big[ \begin{array}{cc}
\p_{x} & \p_{x}\\
\d & \d\\
\end{array}\Big] \,+\, \frac{1}{2} e^{- |z-z'| |D_{x}|} \Big[ \begin{array}{cc}
-\p_{x} & - (\textup{sign}(z-z'))\p_{x} \\
(\textup{sign}(z-z')) \d & \d\\
\end{array} \Big],
\end{equation}
\begin{equation}
M(\p_x h) = \Big[ \begin{array}{cc}
0 & \p_{x} |D_{x}|^{-1}\circ \p_x h \\
-\p_x h & (\p_x h)^2\\
\end{array}\Big].
\end{equation}
Now, it is easy to see there exists a fixed point type structure of $\nabla_{x,z}\varphi$ in (\ref{equation5700}). We remark that (\ref{equation5700}) is the starting point of the   $L^2-$type and $L^\infty-$type estimates of the Dirichlet-Neumann operator.

We have the following lemma regarding on the boundedness of the operator $K(z,z')$.
\begin{lemma}\label{kernelestimate}
For well defined vector function $g: (-\infty,0]\times \R  \longmapsto \R^2$, any $j\in \mathbb{Z}$ and $p_1,q_1$, $p_2$ and $q_2$ s.t. $ 2\leq q_1\leq p_1 \leq +\infty$, $1\leq q_2
\leq p_2\leq + \infty$ we have
\begin{equation}\label{eqn1300}
 \| \int_{-\infty}^0 K(z,z') P_{j}[g(z')] d z'\|_{L^{p_2}_zL^{p_1}_x} \lesssim  2^{j(1/q_1+ 1/q_2-1/p_1-1/p_2)} \| P_{j} g \|_{L^{q_2}_{z} L^{q_1}_x}.
\end{equation}
\end{lemma}
\begin{proof}
From the explicit formula of operator $K(z,z')$ in (\ref{eqn164}), it's easy to find that the kernel of $K(z,z')\circ P_j$ is given by
\begin{equation}
\widetilde{K}_j(z,z',x):=\int_{\R} \frac{1}{2}  e^{i x \xi +(z+z')|\xi|} m_1(\xi)\psi_{j}(\xi)  + \frac{1}{2} e^{ix\xi-|z-z'||\xi|} m_2(\xi) \psi_{j}(\xi) d \xi,
\end{equation}
where matrices $m_1(\xi)$ and $m_2(\xi)$ are given by
\[
m_1(\xi)=\Big[\begin{array}{cc}
-i\xi & -i\xi\\
|\xi| & |\xi| \\
\end{array}\Big], \quad m_2(\xi) = \Big[ \begin{array}{cc}
i \xi & i\textup{sign}(z-z') \xi\\
\textup{sign}(z-z')|\xi| & |\xi|\\ 
 \end{array}\Big].
\]
We can integration by parts in $\xi$ and gain $(|x|+|z-z'|)^{-1}$ with the price of $2^{-j}$, hence we have the following pointwise estimate for the kernel
\begin{equation}\label{equation5707}
|\widetilde{K}_j(z,z',x)|\lesssim_{N} 2^{2j} (1+ 2^j |x|+ 2^j |z\pm z'|)^{-N}.                                                
\end{equation}
With above estimate on the kernel and after using H\"older inequality and Bernstein inequality, it's not difficult to see estimate (\ref{eqn1300}) hold.
\end{proof}

\begin{lemma}\label{L2Dirichlet}
Under the assumption \textup{(\ref{eqn13005})}, the following estimates hold for $\gamma\leq \gamma'$,
\begin{equation}\label{equation5851}
  \| \nabla_{x,z} \varphi \|_{L^\infty_z\widetilde{W^\gamma}} \lesssim \| (\p_x \psi )
\|_{ {W^\gamma}},\,\,  \| \nabla_{x,z} \varphi\|_{L^\infty_z H^k}\lesssim \| \d\psi\|_{H^k}+ \| \p_x h\|_{H^k}\|  \p_x \psi  \|_{ {W^0}},
\end{equation}
\begin{equation}
 \| \Lambda_{\geq 2}[\nabla_{x,z}\varphi]\|_{L^\infty_z\widetilde{W^\gamma}} \lesssim \| \p_x h\|_{\widetilde{W^\gamma}}\| (\p_x \psi )\|_{ {W^\gamma}},
\end{equation}
\begin{equation}
 \| \Lambda_{\geq 2}[\nabla_{x,z}\varphi]\|_{L^\infty_zH^k} \lesssim \| \p_x h\|_{\widetilde{W^0}} \| \d\psi\|_{H^k} + \|  \p_x \psi  \|_{ {W^0}}\|\p_x h\|_{H^k},
\end{equation}
\begin{equation}\label{equation5867}
 \| \Lambda_{\geq 3}[\nabla_{x,z}\varphi]\|_{L^\infty_z H^k} \lesssim \|\p_x h\|_{\widetilde{W^0}}^2\| \d\psi\|_{H^k} + \| \p_x h\|_{\widetilde{W^0}}\| \p_x \psi \|_{ {W^0}} \| \p_x h\|_{H^k}.
\end{equation}
\end{lemma}
\begin{proof}
From (\ref{equation5700}) and    (\ref{eqn1300}) in lemma \ref{kernelestimate}, we have
\[
\| \nabla_{x, z}\varphi \|_{L^{ \infty}_z H^k}    \lesssim \| \d\psi\|_{H^{k}}  + + \| | \big[ M(\p_x h)\nabla_{x,z}\varphi\big]\|_{L^\infty_z H^k}
\]
\begin{equation}\label{equation5705}
 \lesssim \| \d\psi\|_{H^k} + \| \p_x h\|_{
\widetilde{W^0}} \| \nabla_{x,z}\varphi\|_{L^\infty_z H^{k}} + \| \nabla_{x,z}\varphi\|_{L_z^\infty \widetilde{W^0}}\| \hx\|_{H^k}.
\end{equation}
Again, from (\ref{equation5700}) and (\ref{eqn1300}) we have
\[
\| \nabla_{x, z}\varphi \|_{L^\infty_z \widetilde{W^\gamma}} \lesssim \| e^{z\d}[\p_x \psi, \d\psi]\|_{L^\infty_z \widetilde{W^\gamma}} + \| \int_{-\infty}^0 K(z, z') M(\hx) (\nabla_{x, z}\varphi(z')) d\, z' \|_{L^\infty_z W^{\gamma}} 
\]
\begin{equation}\label{eqn1306}
+ \sum_{l=1,2}\| (\hx)^l \nabla_{x,z}\varphi\|_{L^\infty_z \widetilde{W^\gamma}}\lesssim \|  \p_x \psi \|_{ {W^\gamma}} + \big[  \| \hx\|_{L^2}^{2/q}\|\hx\|_{L^\infty}^{1-2/q} + \| \hx\|_{\widetilde{W^\gamma}} \big]\|\nabla_{x,z}\varphi\|_{L^\infty_z \widetilde{W^\gamma}},
\end{equation}
which further   gives us  the following estimate from  the smallness assumption (\ref{eqn13005}), 
\begin{equation}\label{equation6005}
 \|  \nabla_{x,z}\varphi \|_{L^\infty_z \widetilde{W^\gamma}} \lesssim \|  \p_x \psi  \|_{ {W^\gamma}}.
\end{equation}
Combining estimates (\ref{equation6005}) and (\ref{equation5705}), we have 
\begin{equation}
\| \nabla_{x, z}\varphi \|_{L^{ \infty}_z H^k} \lesssim 
\| \d\psi\|_{H^{k}} + \|\hx\|_{\widetilde{W^0}} \| \nabla_{x, z}\varphi \|_{L^\infty_z H^k} + \| \hx\|_{H^{k}} \| \p_x \psi \|_{{W^0}},
\end{equation}
which  further gives us the following estimate  from  the smallness assumption (\ref{eqn13005}), 
\begin{equation}\label{equation00090}
\| \nabla_{x, z}\varphi \|_{L^\infty_z H^k} \lesssim \| \d\psi \|_{H^k}+ \| \hx\|_{H^k}\|  \p_x \psi \|_{ {W^0}}. 
\end{equation}
From (\ref{equation6005}) and (\ref{equation00090}), we can see our desired estimate (\ref{equation5851}) holds. From (\ref{equation5700}), one can derive a fixed point type formulation for $\Lambda_{\geq i 
}[\nabla_{x,z}\varphi]$, $i\in \{2,3\}$.  Hence, all other estimates can be derived very similarly. We omit details here.
\end{proof}

By the same fixed point type argument, we can also derive the following Lemma, which is very helpful when the scaling vector field is applied to the Dirichlet-Neumann operator.  Note that, different from the estimates in Lemma \ref{L2Dirichlet},  in Lemma \ref{lemma2}, only $\psi$ is putted in $L^2$-type space.
\begin{lemma}\label{lemma2}
Under the smallness assumption \textup{(\ref{eqn13005})}, the   following estimates hold for $i\in\{1,2,3\}$ and $k\leq \gamma'-1$,
\begin{equation}\label{eqn1475}
\| \Lambda_{\geq i}[\nabla_{x,z}\varphi]\|_{L^\infty_z H^{k}} \lesssim \| \p_x h\|_{{W^{k}}}^{i-1}\| \d\psi\|_{H^{k}},
\end{equation}
\begin{equation}\label{eqn1310}
\| \Lambda_{\geq 4}[\nabla_{x,z}\psi]\|_{L^\infty_z L^2} \lesssim \| \p_x h\|_{{W^0}}^3 \| \d\psi\|_{L^2}.
\end{equation}
\end{lemma}

 The following Lemma is very helpful if one wants to estimate the term of type $\psi_2B(h)\psi_1 -B(h)(\psi_1\psi_2)$. This type of terms will appear when the scaling vector field hits the Dirichlet-Neumann operator.

\begin{lemma}\label{DNofproduct}
Given any two well defined smooth functions $f(x)$ and $g(x)$, if $\varphi_1(x,z)$ satisfies $P \varphi_1=0$ and $\varphi_1\big|_{z=0}=f g$, meanwhile $\varphi_2(x, z)$ satisfies $P \varphi_2=0$ and  $\varphi_2\big|_{z=0}=g(x)$, then under the smallness assumption that $\| \eta\|_{\widetilde{W^{\gamma+1}}} \leq \delta < 1$, we have the following estimate for $k\leq \gamma$,
\begin{equation}\label{equation8030}
\| \nabla_{x,z}\varphi_1 ( z) - f \nabla_{x,z}\varphi_2 (z)\|_{L^\infty_z H^k} \lesssim \| g \|_{\widetilde{W^k}} \| f\|_{H^{k+1,1-\epsilon}},
\end{equation}
\begin{equation}\label{eqn1670}
\| \Lambda_{\geq 3}[\nabla_{x, z}\varphi_1 - f \nabla_{x,z}\varphi_2]\|_{L^\infty_z H^k} \lesssim \| \eta\|_{\widetilde{W^{k+1}}}\| g\|_{\widetilde{W^k}} \|  f\|_{H^{k+1,1-\epsilon}},
\end{equation}
\begin{equation}\label{eqn1671}
\| \Lambda_{\geq 4}[\nabla_{x, z}\varphi_1 - f \nabla_{x,z}\varphi_2]\|_{L^\infty_z L^2} \lesssim \| \eta\|^2_{\widetilde{W^{1}}}\| g\|_{\widetilde{W^0}} \|  f\|_{H^{1,1-\epsilon}}.
\end{equation}

\end{lemma}
\begin{proof}
Recall  (\ref{equation5700}). We have 
\[
\nabla_{x,z}\varphi_1(z,) = e^{z\d}[\p_x(f g), \d(f g)] + \int_{-\infty}^{0} K(z, z') M(\hx)\nabla_{x,z}\varphi_1(z') d z'\]
\[+[0,\hx\p_x \varphi_1-(\hx)^2\p_z \varphi_1],
\]
\[
f \nabla_{x,z}\varphi_2 (z) = f e^{z\d}[\p_x g, \d g] + \int_{-\infty}^0 f K(z, z') M(\hx)\nabla_{x, z}\varphi_2(z')  d z' \]
\[+[0,\hx f\p_x\varphi_2 -(\hx)^2 f \p_z \varphi_2].
\]
Note that,
\[
\mathcal{F}\big[ e^{z\d}[\p_x(f g), \d(f g)] - f(\cdot)  e^{z\d}[\p_x g, \d g]\big](z, \xi) = \]
\[
 \int_{\R} e^{z|\xi|} a(\xi) \widehat{f}(\xi-\eta) \widehat{g}(\eta)\, d \eta - \int_{\R} \widehat{f}(\xi-\eta) \widehat{g}(\eta) e^{z|\eta|} a(\eta)\,d\eta, \quad \textup{where}\, a(\xi)=[i\xi,\,\,|\xi|].
\]
Note that, the following estimate holds for any fixed $z\leq 0$,
\[
\big| e^{z|\xi|} a(\xi)- e^{z |\eta|}a(\eta) \big| \lesssim \big||\xi|- |\eta|\big|.
\]
Hence
\begin{equation}\label{eqn140000}
\| e^{z\d}[\p_x(f g), \d(f g)] - f(\cdot)  e^{z\d}[\p_x g, \d g]\|_{L^\infty_z H^k} \lesssim \| f\|_{H^{k+1, 1-\epsilon}}\| g \|_{\widetilde{W^k}}.
\end{equation}
We can write the difference of $\nabla_{x,z}\varphi_1$ and $f\nabla_{x,z}\varphi_2$ as follows:
\[
\nabla_{x,z}\varphi_1(z)- f \nabla_{x,z}\varphi_2(z)= e^{z\d}[\p_x(f g), \d(f g)]-f e^{z\d}[\p_x g, \d g] +
\]
\[
[0,\hx(\p_x \varphi_1-f \p_x \varphi_2)-(\hx)^2(\p_z\varphi_1- f \p_z\varphi_2)]\]
\[ + \int_{-\infty}^0 K(z,z') M(\hx)[\nabla_{x,z} \varphi_1(z') - f \nabla_{x,z}\varphi_2(z')] d z' 
\]
\[
- \int_{-\infty}^0 f K(z,z') M(\hx) \nabla_{x,z}\varphi_2(z')+ K(z,z') M(\hx)[f\nabla_{x,z}\varphi_2(z')] d z'.
\]
Very similar to the estimate of (\ref{eqn140000}), we have
\[
\| \nabla_{x, z}\varphi_{1} - f\nabla_{x, z}\varphi_2\|_{L^\infty_z H^k}\lesssim  \|  f\|_{H^{k+1,1-\epsilon}}\| g \|_{\widetilde{W^k}} + 
\]
\begin{equation}\label{equation6015}
\| \p_x h \|_{{W^{k}}} \|\nabla_{x,z}\varphi_1 - f \nabla_{x,z}\varphi_2 \|_{L^\infty_z H^k} + \|  f\|_{H^{k+1,1-\epsilon}} \| \hx\|_{\widetilde{W^k}} \| \nabla_{x,z}\varphi_2\|_{L^1_z \widetilde{W^k}}.
\end{equation}
From the fixed point type formulation  in (\ref{equation5700}) for  $\nabla_{x,z}\varphi_2$, the following estimate holds, 
\begin{equation}
\| \nabla_{x,z}\varphi_2 \|_{L^1_z \widetilde{W^k}} \lesssim 
\| g \|_{\widetilde{W^k}} + \| M(\hx)\nabla_{x,z}\varphi_2\|_{L^1_z \widetilde{W^k}}\lesssim \| g \|_{\widetilde{W^k}} + \| \hx\|_{\widetilde{W^{k}}} \| \nabla_{x,z}\varphi_2\|_{L^1_z \widetilde{W^k}},
\end{equation}
which further gives us the following estimate from the smallness assumption on $\eta$,
\begin{equation}\label{equation6014}
\| \nabla_{x,z}\varphi_2\|_{L^1_z \widetilde{W^k}}\lesssim \| g \|_{\widetilde{W^k}}.
\end{equation}
From (\ref{equation6015}) and (\ref{equation6014}), the following estimate holds,
\begin{equation}
\| \nabla_{x,z}\varphi_1 - f\nabla_{x,z} \varphi_2 \|_{L^\infty_z H^k}\lesssim  \| g \|_{\widetilde{W^k}} \| f\|_{H^{k+1, 1-\epsilon}}.
\end{equation}
Similar to the proof of (\ref{eqn1310}) in Lemma \ref{lemma2}, with minor modifications, we can show our desired estimates (\ref{eqn1670}) and (\ref{eqn1671}) hold.  We omit the details here.
\end{proof}

\section{Paralinearization of the gravity water waves system}\label{appendxc}

 \begin{lemma}\label{paralinearizationDN1}
The following paralinearization for the Dirichlet-Neumann operator holds,
\begin{equation}\label{equation224}
G(\eta)\psi = \d \omega -T_{V}\p_x h +F(h)\psi, \quad \omega:= \psi -T_{B(h)\psi}\eta,
\end{equation}
where $F(h)\psi$ is the good  quadratic and higher error term that does not lose derivative. Moreover, under the smallness condition \textup{(\ref{eqn13005})}, the following estimate holds for $i\in\{2,3\}$, and $k\geq 0$,
\begin{equation}\label{equation635}
\| \Lambda_{\geq i}[F(h)\psi]\|_{H^k} \lesssim \|(\p_x h,\p_x \psi) \|_{H^{k-1}} \| (\p_x h,\p_x \psi ) \|_{ {W^1}}^{i-1}.
\end{equation}
\end{lemma}
\begin{proof}
Recall that 
\[
G(\,h)\psi = \big((1+|\p_x h|^2)\p_z \varphi - \p_x h  \p_x \varphi\big)\big|_{z=0}.
\]
 Define $W= \varphi - T_{\p_z \varphi} h$ and $\uline{V}= \p_x \varphi - \p_x h \p_x \p_z \varphi$. By using (\ref{equation340}), as a result, we have the paralinearization of above equation as follows,
  \[
\big((1+|\p_x h|^2)\p_z \varphi - \p_x h  \p_x \varphi\big)= T_{1+|\p_x h |^2} \p_z \varphi + 2T_{\p_z\varphi} T_{\p_x h}\p_x h  + T_{\p_x \varphi}R_{\mathcal{B}}(\p_x h , \p_x h ) 
\]
\[-T_{\p_x h}\p_x \varphi - T_{\p_x \varphi} \p_x h=T_{1+|\p_x h |^2} \p_z (W + T_{\p_z \varphi} h) + 2T_{\p_z\varphi} T_{\p_x h}\p_x h \]
\[-T_{\p_x h}\p_x( W +T_{\p_z \varphi} h)   - T_{\p_x \varphi} \p_x h +  T_{\p_x \varphi}R_{\mathcal{B}}(\p_x h , \p_x h )\]
\[
 =\d W + T_{1+|\p_x h |^2} \p_z W -  T_{\p_x h} \p_x W -\d W - T_{\uline{V}}\p_x h + \tilde{\mathcal{R}}_1,
\]
\begin{equation}\label{equation222}
=\d W -T_{\uline{V}} \p_x h + T_{1+|\p_x h|^2}(\p_z -T_A) W  +  T_{|\p_x h|^2}T_{A-|\xi|} W - T_{|\p_x h|^2(A-|\xi|)}W + \tilde{\mathcal{R}}_1,
\end{equation}
\[ =\d W -T_{\uline{V}} \p_x h + \widetilde{F}(h)\psi, 
 \]
 where
 \[
  A =  \frac{1}{1+|\p_x h|^2}(i \p_x h \cdot \xi + |\xi|),
\]
 \[
\tilde{\mathcal{R}}_1:=2T_{\p_z \varphi}T_{\p_x h } \p_x h- 2 T_{\p_z \varphi\p_x h}\p_x h-T_{\p_x h }T_{\p_z\varphi}\p_x h+ T_{\p_z \varphi\p_x h}\p_x h \]
\begin{equation}\label{equation000921}
-T_{\p_x h }T_{\p_x \p_z\varphi}h +  T_{\p_x \varphi}R_{\mathcal{B}}(\p_x h , \p_x h ),
 \end{equation}
 \begin{equation}\label{equation836}
 \widetilde{F}(h)\psi = 
 \big(T_{1+|\p_x h|^2}(\p_z -T_A) W  +  T_{|\p_x h|^2}T_{A-|\xi|} W - T_{|\p_x h|^2(A-|\xi|)}W \big)  + \tilde{\mathcal{R}}_1 .  
\end{equation}
From the explicit formula  of $\tilde{\mathcal{R}}_1$ in (\ref{equation000921}) and  the composition Lemma \ref{composi}, it is easy to see that all terms inside $\widetilde{F}(h)\psi$ except $T_{1+|\p_x h|^2}(\p_z -T_A) W $ are good error terms, which do not lose derivatives at the high frequency part or the low frequency part.

After evaluation (\ref{equation836}) at the boundary $z=0$, we can see  the following estimate holds,
\[
\| \Lambda_{\geq i}[F(h)\psi]\|_{H^k} \lesssim \| \Lambda_{\geq i}\big[(\p_z -T_A) W\big|_{z=0}\big]\|_{H^k} + \| \p_x h \|_{\widetilde{W}^1}^{i-1} \| \omega \|_{H^{k-1}}\] 
\[+ \| (\p_x h , \p_x \psi , \d \psi)\|_{\widetilde{W}^1}^{i-1} \|(\p_x h , \p_x \psi)\|_{H^{k-1}},
\lesssim \|(\p_x h,\p_x \psi) \|_{H^{k-1}} \| (\p_x h,\p_x \psi ) \|_{ {W^1}}^{(i-1)}.
\]
Note that, we used (\ref{equation887}) in Lemma \ref{forwardparabolic} in above estimate.
\end{proof}

In the following lemma, we show that term $(\p_z -T_A) W $ in (\ref{equation836}) also does not lose derivatives. More precisely, the following estimate holds.
\begin{lemma}\label{forwardparabolic}
The following estimate holds for $i\in\{2,3 \}$, $k \geq 0$,
\begin{equation}\label{equation887}
\| \Lambda_{\geq i}\big[(\p_z -T_{A}) W \big|_{z=0}\big]\|_{H^k } \lesssim \|(\p_x h,\p_x \psi) \|_{H^{k-1}} \| (\p_x h,\p_x \psi ) \|_{ {W^1}}^{(i-1)}.
\end{equation}
\end{lemma}
\begin{proof}
Note that $\Lambda_{1}[(\p_z - T_{A} )W]=0$, i.e. $(\p_z-T_{A})W$ itself is quadratic and higher. From (\ref{equation254}) in Lemma \ref{paralinearization1}, we have
\begin{equation}
(\p_{z}-T_{\tilde{a}}) (\p_{z} -T_{\tilde{A}} ) W =f ,
\end{equation}
where $\tilde{a}$ and $\tilde{A}$ is given in (\ref{equation8349347}).

  From (\ref{coreest1}) in Lemma \ref{elliptic} and (\ref{equation264}) in Lemma \ref{paralinearization1}	, the following estimate holds, 
\[
\| (\p_z -T_A ) W\big|_{z=0}\|_{H^k} \lesssim \sup_{z\in[-1/4,0]} \| (\p_z -T_{\tilde{A}} ) W(z,x)\|_{H^k}  + \| T_{A'} W \big|_{z=0}\|_{H^k}
\]
\[
\lesssim \sup_{z\in[- 1,0]} \| (\p_z -T_{\tilde{A}} ) W(z,x)\|_{H^{k-1}}  + \| f(z,\cdot)\|_{H^k} + \|(\p_x h,\p_x \psi) \|_{H^{k-1}} \| (\p_x h,\p_x \psi ) \|_{ {W^1}}
\]
\[
\lesssim \|(\p_x h,\p_x \psi) \|_{H^{k-1}} \| (\p_x h,\p_x \psi ) \|_{ {W^1}}.
\]
It remains to consider the case when $i=3$. We define,
 \begin{equation}\label{equation5800}
\widetilde{W}(z, x) = (\p_{z}- T_{\tilde{A}}) W(z,x) - h(z, x),\quad z \leq 0,
\end{equation}
where the function $h(z, x)$ is to be determined. Let $h(0,x)$ to be  given as follows,
\[
h(0, x):=  \Lambda_{2}\big[(\p_{z}- T_{\tilde{A}})  W(z,x)\big|_{z=0}\big].
\]
 Hence, from the definition, we know that $\widetilde{W}(0,\cdot)$ is the cubic and higher order terms of $(\p_{z}- T_{\tilde{A}}) W(z,x)|_{z=0}$. 
  
We let $h(z,x)$ to be the solution of the following parabolic equation: 
\begin{equation}
\left\{ \begin{array}{l}
(\p_{z } - T_{\tilde{a}}) h = \Lambda_{2}\big[f\big],\\
\\
h(0, x) =  \Lambda_{2}\big[(\p_{z}- T_{\tilde{A}})  W(z,x)\big|_{z=0}\big].\\
\end{array}\right.
\end{equation}
Therefore, we can derive the following parabolic equation satisfied by $\widetilde{W}(z, x)$ as follows,
\begin{equation}\label{equation5805}
(\p_{z}- T_{\tilde{a}}) \widetilde{W} = 
\widetilde{\mathcal{N}}:= f - (\p_{z} -T_{\tilde{a}} ) h=  \Lambda_{\geq 3}[f].
\end{equation}
From (\ref{coreest1}) in Lemma \ref{elliptic} and (\ref{equation264}) in Lemma \ref{paralinearization1}, the following estimate holds, 
\[
\| \Lambda_{\geq 3}\big[(\p_z -T_{ {A}} ) W\big|_{z=0}\big]\|_{H^k} \lesssim \sup_{z\in[-1/4,0]} \| (\p_z -T_{\tilde{A}} ) \widetilde{W}(z,x)\|_{H^k} + \| \Lambda_{\geq 3}[T_{A'} W\big|_{z=0}]\|_{H^k}
\]
\[
\lesssim \sup_{z\in[- 1,0]} \| (\p_z -T_A )  \widetilde{W}(z,x)\|_{H^{k-1}}  + \| \Lambda_{\geq 3}[f(z,\cdot)]\|_{H^k} + \|(\p_x h,\p_x \psi) \|_{H^{k-1}} \| (\p_x h,\p_x \psi ) \|_{ {W^1}}^{2}
\]
\[
\lesssim  \|(\p_x h,\p_x \psi) \|_{H^{k-1}} \| (\p_x h,\p_x \psi ) \|_{ {W^1}}^{2}.
\]
\end{proof}

\begin{lemma}\label{paralinearization1}
The following decomposition holds,
\begin{equation}\label{equation254}
(\p_z -T_{\tilde{a}})(\p_z -T_{\tilde{A}}) W =f,
\end{equation}
where
\begin{equation}\label{equation8349347}
\tilde{a}= a + a', \quad \tilde{A}= A +A',
\end{equation}
\begin{equation}\label{equation000790}
a = \frac{1}{1+|\p_x h|^2}(i \p_x h \cdot \xi - |\xi|), \quad A =  \frac{1}{1+|\p_x h|^2}(i \p_x h \cdot \xi + |\xi|),
\end{equation}
\begin{equation}\label{equation000793}
a'= \frac{1}{A-a} \big( i \p_\xi a \p_x A - \frac{\p_x^2 h a}{1+|\p_x h |^2} \big),
\quad 
A'= \frac{1}{a-A}\big(i \p_\xi a \p_x A - \frac{\p_x^2 h A}{1+|\p_x h |^2} \big).
\end{equation}
and ``$f$'' is a good error term. Its precise formula is given in \textup{(\ref{equation294u39})}. Moreover, under the smallness condition \textup{(\ref{eqn13005})}, the following estimate holds for $i\in\{2,3\}$,
\begin{equation}\label{equation264}
\sup_{z\leq 0}\|\Lambda_{\geq i }[ f(z)]\|_{H^{k}}\lesssim \|(\p_x h,\p_x \psi) \|_{H^{k-1}} \| (\p_x h,\p_x \psi ) \|_{ {W^1}}^{(i-1)}.
\end{equation}
\end{lemma}
\begin{proof}
 Recall that $P \varphi=0$. After paralinearizing this equation, we have
\begin{equation}\label{equation190}
\mathcal{P}  \varphi + 2T_{\p_z^2\varphi \p_x h} \p_x h   - 2T_{\p_x \p_z \varphi}\p_x h    -T_{\p_z \varphi}\p_x^2 h+f_1=0, 
\end{equation}
where
\begin{equation}\label{equation000678}
\mathcal{P}:= T_{(1+|\p_x h|^2)} \p_z^2 +\p_x^2 -2T_{\p_x h}\p_x \p_z -T_{\p_x^2 h}\p_z,
\end{equation}
\[f_1 = {R}_{\mathcal{B}}(|\p_x h|^2, \p_z^2 \varphi) -2R_{\mathcal{B}}(\p_x\p_z \varphi, \p_x h ) - R_{\mathcal{B}}(\p_x^2 h, \p_z \varphi)\]
\[+ T_{\p_z^2 \varphi} R_{\mathcal{B}}(\p_x h, \p_x h )
+2\big(T_{\p_z^2 \varphi}T_{\p_x h}- T_{\p_z^2\varphi \p_x h}\big)\p_x h .
\]
Recall that $W= \varphi - T_{\p_z \varphi }h$, then from (\ref{equation190}), we have
\begin{equation}\label{equation000578}
\mathcal{P} W = f_2,\quad f_2= - {R}_{\mathcal{B}}(|\p_x h|^2, \p_z^2 \varphi)+2R_{\mathcal{B}}(\p_x\p_z \varphi, \p_x h )
\end{equation}
\[ 
 + R_{\mathcal{B}}(\p_x^2 h, \p_z \varphi)- T_{\p_z^2 \varphi} R_{\mathcal{B}}(\p_x h, \p_x h )-T_{|\p_x h |^2} T_{\p_z^3 \varphi} h +T_{\p_z^3\varphi |\p_x h |^2} h
\]
\begin{equation}\label{equation220}
 + 2 T_{\p_x h }T_{\p_x \p_z^2 \varphi} h - 2T_{\p_x h \p_x \p_z^2\varphi} h + T_{\p_x^2 h }T_{\p_z^2 \varphi} h -T_{\p_x^2 h \p_z^2 \varphi} h.
\end{equation}
From the composition Lemma \ref{composi}, we can see that $f_2$ does not lose derivative.

Note that the following identities holds from (\ref{equation000790}) and (\ref{equation000793}),
\[
\tilde{a}+ \tilde{A} = \frac{2i \p_x h \xi}{1+|\p_x h |^2}   +\frac{\p_x^2 h}{1+|\p_x h |^2}, 
\]
\[
\tilde{a}\sharp \tilde{A}  := \tilde{a} \tilde{A}  +\frac{1}{i} \p_\xi \tilde{a}  \p_x \tilde{A} 
= -\frac{|\xi|^2}{1+|\p_x h |^2} + a'A'- i \p_\xi \tilde{a}\p_x A'-i \p_\xi a' \p_x A.
\]

Hence, from (\ref{equation000578}), we have 
\begin{equation}\label{equation000339}
(T_{(1+ |\p_x h |^2)}\p_z^2 -T_{(\tilde{a}+\tilde{A})(1+ |\p_x h |^2)} +T_{\tilde{a}\sharp\tilde{A}(1+ |\p_x h |^2)}) W = f_2 + R_0  ,
\end{equation}
where
\[
R_0:= T_{(a'A'-i \p_\xi \tilde{a} \p_x A'-i\p_\xi a'\p_x A)(1+|\p_x h|^2)} W.
\]
It is easy to verify that $R_0$ does not lose derivative.  From (\ref{equation000339}),  we have
\[
\big(T_{(1+|\p_x h|^2)}(\p_z -T_{\tilde{a}} )(\p_z -T_{\tilde{A}})\big)W = R_1,\]
where
\[  R_1 = f_2 + R_0	+ T_{(1+|\p_x h|^2)}\big( T_{\tilde{a}}T_{\tilde{A}}-T_{\tilde{a}\sharp \tilde{A}}\big) \]
\[ + T_{ |\p_x h |^2 } T_{\tilde{a}\sharp \tilde{A}}W - T_{\tilde{a}\sharp \tilde{A}  |\p_x h|^2 }W -T_{ |\p_x h |^2 }T_{\tilde{a}+\tilde{A}} W + T_{(\tilde{a}+\tilde{A}) |\p_x h |^2}W,
\]
which further implies the following
\[
(\p_z -T_{\tilde{a}} )(\p_z -T_{\tilde{A}})W = f, \quad  f =  T_{1/(1+|\p_x h|^2)} R_1\]
\begin{equation}\label{equation294u39}
+ \big( T_{-|\p_x h|^4/(1+|\p_x h |^2)} - T_{-|\p_x h|^2/(1+|\p_x h |^2)}T_{|\p_x h|^2}\big)(\p_z -T_{\tilde{a}})(\p_z - T_{\tilde{A}}) W.
\end{equation}
Again, from the composition Lemma \ref{composi}, we can see that  $R_1$ and $f$ do not lose derivative from above explicit formulas. After checking terms in (\ref{equation294u39}) very carefully, the desired estimate (\ref{equation264}) follows very easily.
\end{proof}

\begin{lemma}\label{elliptic}
Let $a\in \Gamma^{1}_{2}(\R^2)$ and satisfies the following assumption,
$
Re[a(x,\xi)]\geq c |\xi|,$
for some positive constant $c$. If $u$ solves the following equation 
\[
(\p_w + T_a) u(w,\cdot) = g(w,\cdot),
\]
then we have the following estimate holds for  any fixed and sufficiently  small constant  $\tau$, and arbitrarily small constant $\epsilon >0.$
\begin{equation}\label{coreest1}
\sup_{w\in[\tau,0]}\| u(w)\|_{H^{k}} \lesssim M_2^1(a) \frac{1+|\tau|}{|\tau|}\big[ \sup_{z\in[4\tau,0]} \| u(w)\|_{H^{k-2(1-\epsilon)}} + \sup_{z\in[4\tau, 0]} \| g(z)\|_{H^{\mu-(1-\epsilon)}}\big].
\end{equation}
\end{lemma}
\begin{proof}
A detailed proof can be found in \cite{alazard}, above result is the combination of \cite{alazard}[Lemma 2.1.9] and the proof part of \cite{alazard}[Lemma 2.1.10].

\end{proof}

\subsection{Estimate of the scaling vector field part}
\begin{lemma}\label{propositionscaling}
Under the smallness assumption \textup{(\ref{assumption})},  the following estimates hold for $i\in\{1,2,3\}$ and any $0< \epsilon \ll 1$,
 \begin{equation}\label{equation499}
\|\Lambda_{\geq i }[SB(h)\psi]\|_{H^{k}} \lesssim    \|(\p_x h,\p_x \psi)\|_{ {W^{k}}}^{i-1} \big[ \| (S \eta, \eta)\|_{H^{k+1,1-\epsilon}}  +\|(\d (S \psi),\d\psi)\|_{H^k} \big],
 \end{equation}
 \begin{equation}\label{equation9745}
\|\Lambda_{\geq 4 }[SB(h)\psi]\|_{L^2} \lesssim \big(\| (Sh, h ) \|_{H^{1,1-\epsilon}}+\| \d (S \psi, \psi)\|_{L^2}\big)\|(\p_x h, \p_x \psi)\|_{W^2}^3.
 \end{equation}
\end{lemma}

\begin{proof}
From \cite{alazard}[Lemma 2.3.3], the following identity holds, 
\[
S B(h)\psi = B(h) (S\psi - (B(h)\psi)S h) + \frac{1}{1+(\p_{x}h)^{2}} \Big[ - 2(V(h)\psi)\p_{x}h   - 2G(h)\psi
\]
\begin{equation}\label{equation5732}
 + 2G(h)(h B(h)\psi) - 2\eta G(h)(B(h)\psi) + \big[ \p_{x}h \p_{x}(B(h)\psi) - \p_{x}(V(h)\psi)\big] S h   \Big].
\end{equation}
We remark that above identity is derived from comparing the $S\p_z \varphi|_{z=0}$ with the $\p_z \varphi_1\big|_{z=0}$, where $\varphi_1(0)=S\psi$. After studying the two harmonic equations in the $(x,z)$ coordinates formulation, one should derive (\ref{equation5732}) without any problem.

From the facts that $P \varphi =0$ and $B(\eta)\psi =\p_z\varphi|_{z=0}$, after evaluating $P\varphi(z)$ at the boundary $z=0$, we can derive the following identity,
\begin{equation}\label{equation6010}
 -\p_x [V(h)\psi]=G(h)(B(h)\psi).
\end{equation}
After combining the identities (\ref{equation5732}), (\ref{equation6010}) , and (\ref{eqn1003}), we have the following identity
\[
S [B(h)\psi] = B(h)(S\psi)  - \frac{2 \p_x h V(h)\psi + 2 G(h)\psi}{1+ (\p_x h)^2} - \frac{2\p_x h \p_xh B(h)\psi }{1+(\p_x h)^2} +
\]
\begin{equation}\label{equation8024}
2\big[B(h)(h B(h)\psi) - h B(h)(B(h)\psi)\big]  + \big[S h B(h)(B(h)\psi) - B(h)(S h B(h)\psi)\big]. 
 \end{equation}
Since terms inside (\ref{equation8024}) that do not depend on the scaling vector field ``$S$'' can be estimated very easily  by using estimates in Lemma \ref{L2Dirichlet}, we omit details for those terms. From Lemma \ref{DNofproduct}, the following estimate holds  for $i\in\{1,2,3\}$, 
\[
\| \Lambda_{\geq i } \big[B(h)(h B(h)\psi) - h B(h)(B(h)\psi)\big]\|_{H^k}+ \| \Lambda_{\geq i } \big[Sh B(h)(B(h)\psi) - B(h)(S h B(h)\psi)]\|_{H^k}\] 
\[\lesssim \| (Sh, h) \|_{H^{k+1, 1-\epsilon}} \| \Lambda_{\geq i-1}[B(h)\psi] \|_{\widetilde{W}^k}
\lesssim \| (Sh, h) \|_{H^{k+1, 1-\epsilon}} \|(\p_x h , \p_x \psi, \d\psi)\|_{\widetilde{W}^k}^{(i-1)}.
\]
From Lemma \ref{lemma2}, the following estimate holds, 
\[	
\|\Lambda_{\geq i }[B(h)S\psi]\|_{H^k} \lesssim \| \p_x h \|_{W^k}^{i-1}\| \d(S\psi)\|_{H^k}.
\]
To sum up, we can see our desired estimate (\ref{equation499}) holds.
With an extra estimate (\ref{eqn1310}) in Lemma \ref{lemma2}, one can derive the desired estimate (\ref{equation9745}) by redo the argument used in above. We omit details here.
\end{proof}

\begin{lemma}\label{vectorfield}
Under the smallness assumption \textup{(\ref{eqn13005})}, the following estimate holds for $1\leq k \leq \gamma' $, $p\in(0,1)$, and $i\in\{2,3\}$,
\[
\| \Lambda_{\geq i }[SF(\eta)\psi]\|_{H^{k,p}} \lesssim \big(\| S h \|_{H^{k, p}} +\| h \|_{H^{k+1, 1-p}} + \| \d \psi\|_{H^{k}} + \|\d S\psi \|_{H^{k-1}}\big)
\]
\begin{equation}\label{scalingparaerror}
\times \| (\p_x h, \d\psi)\|_{W^{k+1}}^{(i-1)}.
\end{equation}

\end{lemma}
\begin{proof}
Recall (\ref{equation224}), we have $F(h)\psi = G(h)\psi -\d(\psi - T_{B} h) + \p_xT_{V} h$. Hence, from (\ref{equation8024}) in Proposition \ref{propositionscaling}, we can derive the following equality,
\[
S(F(h)\psi)= F(h)(S\psi) +\big(- \p_x( Sh) V + \p_x T_{V} S h\big)
\]
\begin{equation}\label{eqn220}
 -\big(  G(h)( Sh B)-Sh (G(h)B)-\d T_{B} S h\big) + \mathcal{G}, 
\end{equation}
where the good error term $\mathcal{G}$ is given as follows, 
\[
\mathcal{G}=\mathcal{G}_1 + \mathcal{G}_2+\mathcal{G}_3 ,\]
where
\[ \mathcal{G}_1= 2\big(G(h)(h B)- h G(h)(B)\big),
\]
\[  \mathcal{G}_2=  2 \p_x h V -2 \Lambda_{\geq 2}[G(h)\psi]  - \d T_{B(\eta)(S\psi)}\eta  - \p_x T_{V(\eta)(S\psi)} \eta  + \d T_{SB}\eta + \p_x T_{SV}\eta,\]
\[
\mathcal{G}_3=S[\d (T_{B}\eta) + \p_x (T_{V}\eta) ]- \d T_{S B}\eta- \d T_{B} S \eta - \p_x T_{S V}\eta -\p_x T_{V} S\eta.
\]
One can also refer to \cite{alazard} for more detailed computations to  see (\ref{eqn220}) holds. From Lemma \ref{DNofproduct}, the following estimate holds, 
\[
\| \Lambda_{\geq i } [\mathcal{G}_1]\|_{H^k}\lesssim \| h \|_{H^{k+1, 1-\epsilon}} \| \Lambda_{\geq i-1}[B(h)\psi]\|_{\widetilde{W}^k} \lesssim \| h \|_{H^{k+1, 1-\epsilon}}\|(\p_x h, \p_x \psi )\|_{ {W}^k}.
\]
From (\ref{equation499}) in Lemma \ref{propositionscaling}, the following estimate holds,
\[
\| \Lambda_{\geq i } [\mathcal{G}_2]\|_{H^k}\lesssim \big(\|\Lambda_{\geq i-1}\big[ \big(B(h)(S\psi), V(h)(S\psi)\big)\big]\|_{L^2} \]
\[ +\| \Lambda_{\geq i-1}\big[\big(SB(h)\psi, SV(h)\psi\big)]\|_{L^2}\big)  \| (\p_x h , \p_x \psi )\|_{	 {W}^{k}} + \| (\p_x h , \p_x \psi\|_{H^{k}}  \| (\p_x h , \p_x \psi )\|_{\widetilde{W}^{0}}^{i-1} \]
\[\lesssim \big(\| Sh\|_{H^{1, 1-\epsilon}} +\| (\p_x h,\d \psi)\|_{H^{k}} + \| \d S\psi  \|_{L^2} \big)\| (\p_x h , \p_x \psi )\|_{	 {W}^{k}}^{i-1}.
\]
Note that the commutator term $\mathcal{G}_3$ does not depend on the scaling vector field. It is easy to derive the following estimate,
\[
\| \Lambda_{\geq i } [\mathcal{G}_3]\|_{H^k}\lesssim  \| \Lambda_{\geq i-1}\big[\big(B(h)\psi, V(h) \psi\big)\big]\|_{H^{k}}  \| (\p_x h , \p_x \psi )\|_{ {W^{0}}}\]
\[ + \|(\p_x h , \p_x \psi)\|_{H^k}  \| \Lambda_{\geq i-1}\big[\big(B(h)\psi, V(h) \psi\big)\big]\|_{\widetilde{W^0}}  \lesssim \| (\p_x h , \p_x \psi)\|_{H^k} \| (\p_x h , \p_x \psi  )\|_{ {W^0}}^{i-1}.
\]
To sum up, we have
\begin{equation}\label{equation893}
\| \Lambda_{\geq i } [\mathcal{G} ]\|_{H^k}\lesssim \big(\| S h \|_{H^{1,1-\epsilon}} + \|\d S \psi\|_{L^2} + \| h \|_{H^{k+1,1-\epsilon}} + \| \d\psi\|_{H^k} \big)\| (\p_x h , \p_x \psi )\|_{ {W}^{k}}^{i-1}. 
\end{equation}
To see the structure inside $\big(  G(h)( Sh B)-Sh (G(h)B)-\d T_{B} S h\big)$, we decompose it as follows, 
\begin{equation}
 G(h)( Sh B)-Sh (G(h)B)-\d T_{B} S h  =\sum_{k\in\mathbb{Z}} \mathcal{J}_{1,k} + \mathcal{J}_{2,k},
\end{equation}
where
\[
\mathcal{J}_{1,k}:=  G(h)(P_k B P_{\leq k-1} (S h)) - P_{\leq k-1}(S h) G(h)(P_k B) 
\]
\begin{equation}\label{eqn210}
=   (1+ (\p_x h )^2) \big[B(h)(P_k B P_{\leq k-1}(S h)) - P_{\leq k-1}(S h) B(h)(P_{k}B)]-\p_x h P_k B \p_x(P_{\leq k-1}(S h)),
\end{equation}
\begin{equation}\label{equation145}
\mathcal{J}_{2,k}:=   G(h)(P_{k}(S h) P_{\leq k} B) - P_{k}(S h) G(h)(P_{\leq k}B) - \d T_B S h := \mathcal{J}_{2,k}^1 + \mathcal{J}_{2,k}^2,
\end{equation}
\[
  \mathcal{J}_{2,k}^1: =  F(h)(P_k(S h) P_{\leq k} B) + \d( P_{k}(S h)P_{\leq k-1} B - T_{B} P_{k} S h) 
\]
\begin{equation}\label{eqn211}
=   F(h)(P_k(S h) P_{\leq k} B) +\d( P_{k}(S h)P_{k-10 \leq \cdot\leq k-1} B ).
\end{equation}
\begin{equation}\label{eqn212}
\mathcal{J}_{2,k}^2:= -\big[  \d T_{B(h)(P_k(S h)P_{\leq k}B)} h + \p_x T_{V(h)(P_k(S h)P_{\leq k}B)} h + P_{k}(S h) G(h)(P_{\leq k-1}B )\big].
\end{equation}
Note that in the decomposition (\ref{equation145}), we used again the fact that $F(h)\psi = G(h)\psi -\d(\psi - T_{B} h) + \p_xT_{V} h$.

From Lemma \ref{DNofproduct}, the following estimate holds for some $\epsilon < (1-p)/100$, 
\[
\|\sum_{j\in \mathbb{Z}} \Lambda_{\geq i } [ \mathcal{J}_{1,j}]\|_{H^k} \lesssim \sum_{j\in \mathbb{Z}} \| P_{\leq j-1} (Sh)\|_{H^{k+1, 1-\epsilon/2}} \| \Lambda_{\geq i-1}[P_{j}[B(h)\psi]]\|_{\widetilde{W}^k}  
\]
\[
+  \| S h \|_{H^{k,1-\epsilon}} \| (\p_x h , \p_x \psi) \|_{{W}^{k+1}}^{(i-1)}\lesssim \| S h \|_{H^{k,1-\epsilon}} \| \Lambda_{\geq i-1}[B(h)\psi]\|_{\widetilde{W}^{k+1}}  \]
\begin{equation}\label{equation888}
+  \| S h \|_{H^{k,1-\epsilon}} \| (\p_x h , \p_x \psi) \|_{ {W}^{k+1}}^{(i-1)}\lesssim   \| S h \|_{H^{k,1-\epsilon}} \| (\p_x h , \p_x \psi) \|_{ {W}^{k+1}}^{i-1}.
\end{equation}
From (\ref{equation635}) in Lemma \ref{paralinearizationDN1}, the following estimate holds,
\[
\|\sum_{j\in \mathbb{Z}} \Lambda_{\geq i }[\mathcal{J}_{2,j}^1]\|_{H^k} \lesssim \sum_{j\in \mathbb{Z}}\| \p_x \big(P_{j}(Sh) P_{\leq j}B\big)\|_{H^{k-1}} \|(\p_x h , \p_x \psi)\|_{{W}^1}^{(i-2)}
\]
\begin{equation}\label{equation889}
\lesssim \| S h \|_{H^{k, 1-\epsilon}} \|(\p_x h , \p_x \psi)\|_{{W}^1}^{i-1}.
\end{equation}
\[
\| \sum_{j\in \mathbb{Z}} \Lambda_{\geq i}[\mathcal{J}_{2,j}^2]\|_{H^{k,p}} \lesssim \| Sh\|_{H^{k,p}} \| \Lambda_{\geq i -1}[B(h)\psi]\|_{\widetilde{W}^0}
\]
\[
+ \sum_{j\in \mathbb{Z}}\| \p_x h \|_{{W}^k} \|\Lambda_{\geq i-1}\big[\big(B(h)(P_j(S h)P_{\leq j}B)),V(h)(P_j(S h)P_{\leq j}B)) \big) \big]\|_{L^2}
\]
\begin{equation}\label{equation906}
\lesssim \| Sh \|_{H^{k,p}} \|(\p_x h, \p_x \psi)\|_{W^k}^{i-1}.
\end{equation}
To sum up, from (\ref{eqn220}), (\ref{equation893}), (\ref{equation888}), (\ref{equation889}), and (\ref{equation906}), our desired esitmate (\ref{scalingparaerror}) holds.
\end{proof}

\subsection{Estimate of the Taylor coefficient}
\begin{lemma}\label{alphaestimate}
 Under the smallness assumption \textup{(\ref{eqn13005})}, the following estimates hold for$k$, $\gamma \geq 0 $,  $\gamma \leq \gamma'-2$, and $i\in\{1,2,3\}$,
\begin{equation}\label{equation6004}
 \| \Lambda_{\geq i} [\alpha]\|_{\widetilde{W^\gamma}} \lesssim\|(\p_x h ,\p_x \psi) \|_{ {W^{\gamma+1}}}^{i},  \quad \| \Lambda_{\geq i }[\alpha]\|_{H^k} \lesssim
\| (\p_x h ,\p_x \psi )\|_{ {W^1}}^{i-1} \| (\p_x h, \d\psi)\|_{H^{k+1}}, 
 \end{equation}
\begin{equation}\label{equation396}
\| \Lambda_{\geq i }[S\alpha]\|_{H^k} \lesssim   \|(\p_x h,\p_x \psi)\|_{ {W^{k+1}}}^{i-1} \big[ \| (S \eta, \eta)\|_{H^{k+2,1-\epsilon}}  +\|(\d (S \psi),\d\psi)\|_{H^{k+1}} \big],
\end{equation}
\begin{equation}\label{equation9669}
 \| \Lambda_{\geq i}[\p_t \alpha]\|_{\widetilde{W^\gamma}}  \lesssim \|(\p_x h ,\p_x \psi ) \|_{ {W^{\gamma+2}}}^i,
\end{equation}
 \[\| \Lambda_{\geq i} [\p_t \alpha]\|_{H^k} \lesssim 
\| \|(\p_x h ,\p_x \psi ) \|\|_{{W^2}}^{(i-1)} \| (\p_x h, \d\psi)\|_{H^{k+2}},\]
 \begin{equation}\label{eqn1674}
\| \Lambda_{\geq i }[S\p_t \alpha]\|_{H^k} \lesssim   \|(\p_x h,\p_x \psi)\|_{ {W^{k+2}}}^{i-1} \big[ \| (S \eta, \eta)\|_{H^{k+3,1-\epsilon}}  +\|(\d (S \psi),\d\psi)\|_{H^{k+2}} \big].
 \end{equation}
 
\end{lemma}
\begin{proof}
Recall that $\alpha = \sqrt{a}-1$. It is sufficient to estimate the Taylor coefficient ``$a$''.  We cite the following identity from \cite{alazard}[Lemma A.4.2],
\begin{equation}
a = \frac{1}{2(1+ (\p_x h )^2)}\Big[ 2+ 2V\p_x B - 2B\p_x V - G(h)[V^2 + B^2 + 2h]  \Big].
\end{equation}
Recall that $a=1+\p_t B + V\p_x B$ and $\alpha = \sqrt{a}-1$. Therefore, we have the following identity,
\begin{equation}\label{equation942}
\p_t B =\frac{1}{2(1+ (\p_x h )^2)}\Big[ 2+ 2V\p_x B - 2B\p_x V - G(h)(V^2 + B^2 + 2h)  \Big]-1 - V\p_x B.
\end{equation}
As we already have $L^2$ type and $L^\infty$ type estimate on $B(h)\psi$, we can derive $L^2$ type and $L^\infty$ type estimate for $\p_t B$ from  (\ref{equation942}). As a result, from estimates in Lemma \ref{L2Dirichlet}
, the following estimate holds for $i \in \{1,2,3\}$,
\[
\| \Lambda_{\geq i }[\p_t B]\|_{H^k} \lesssim \| (\p_x h, \p_x \psi)\|_{H^{k+1}} \|(\p_x h , \p_x \psi)\|_{W^1}^{i-1}.
\]
\[
\| \Lambda_{\geq i }[\p_t B]\|_{\widetilde{W}^{\gamma}} \lesssim \|(\p_x h , \p_x \psi)\|_{W^{\gamma +1}}^{i}.
\]
From (\ref{equation499}) in Lemma \ref{propositionscaling}, the following estimate holds,
\[
\| \Lambda_{\geq i }[S\p_t B]\|_{H^k} \lesssim   \|(\p_x h,\p_x \psi)\|_{ {W^{k+1}}}^{i-1} \big[ \| (S \eta, \eta)\|_{H^{k+2,1-\epsilon}}  +\|(\d (S \psi),\d\psi)\|_{H^{k+1}} \big].
\]

After taking another derivative with respect to time ``$t$'' on both hands side of (\ref{equation942}), we have an identity for $\p_t^2 B $. As a result,  from the $L^2$ type and $L^\infty$ type estimates of $\p_t B$  and $B$, the following estimates hold for $i\in\{1,2,3\}$,
\[
\| \Lambda_{\geq i }[\p_t^2 B]\|_{H^k} \lesssim \| (\p_x h, \p_x \psi)\|_{H^{k+2}} \|(\p_x h , \p_x \psi)\|_{W^2}^{i-1}.
\]
\[
\| \Lambda_{\geq i }[\p_t^2 B]\|_{\widetilde{W}^{\gamma}} \lesssim \|(\p_x h , \p_x \psi)\|_{W^{\gamma +2}}^{i}.
\]
\[
\| \Lambda_{\geq i }[S\p_t^2 B]\|_{H^k} \lesssim   \|(\p_x h,\p_x \psi)\|_{ {W^{k+2}}}^{i-1} \big[ \| (S \eta, \eta)\|_{H^{k+3,1-\epsilon}}  +\|(\d (S \psi),\d\psi)\|_{H^{k+2}} \big].
\]
Hence, our desired estimates follows from   (\ref{equation499}) in Lemma \ref{propositionscaling}  and estimates in Lemma \ref{L2Dirichlet}.
\end{proof}

\end{document}